\documentclass[a4paper, 12pt]{article}

\usepackage{amsmath} %odgovoran za \newcommand{\Bi}{\operatorname{Im}}
\usepackage{theorem}
\usepackage{amssymb}
\usepackage{amsfonts}
\usepackage{latexsym}
\usepackage{amscd}
\usepackage{diagramb}
\input GrCalc3.sty
\usepackage{graphicx}

\usepackage{color}

\definecolor{rojo}{rgb}{1,0,0}

\theoremstyle{plain}
\theoremheaderfont{\normalfont\bfseries}

\newtheorem{thm}{Theorem}[section]

\newtheorem{lma}[thm]{Lemma}

\newtheorem{propdefn}[thm]{Proposition and Definition}

\theorembodyfont{\rmfamily}
\newtheorem{prel}[thm]{}
\newtheorem{rem}[thm]{Remark}
\newtheorem{prop}[thm]{Proposition}
\newcommand{\qed}{\hfill\quad\fbox{\rule[0mm]{0,0cm}{0,0mm}}  \par\bigskip}

\newcommand{\w}{\hspace{-0,12cm}}
\newcommand{\ö}{\hspace{-0,03cm}}

\newcommand{\iso}{\cong}
\newcommand{\ot}{\otimes}
\newcommand{\M}{{\mathcal M}}
\newcommand{\YD}{{\mathcal YD}}
\newcommand{\R}{{\mathcal R}}
\newcommand{\C}{{\mathcal C}}
\newcommand{\D}{{\mathcal D}}
\newcommand{\F}{{\mathcal F}}
\newcommand{\G}{{\mathcal G}}
\newcommand{\A}{{\mathcal A}}
\newcommand{\B}{{\mathcal B}}
\newcommand{\E}{{\mathcal E}}
\def\S{{\mathcal S}}
\newcommand{\T}{{\mathcal T}}
\newcommand{\Ll}{{\mathcal L}}
\newcommand{\ch}{{\mathcal H}}
\newcommand{\U}{{\mathcal U}}
\newcommand{\Oo}{{\mathcal O}}
\def\K{{\mathcal K}}

\def\Dd{{\mathbb D}}

\newcommand{\crta}{\overline}
\newcommand{\ev}{\it ev}
\newcommand{\db}{\it db}

\newcommand{\Id}{\operatorname {Id}}
\newcommand{\alfa}{\alpha}

\newcommand{\ro}{\rho}

\newcommand{\Epsilon}{\varepsilon}
\newcommand{\id}{\it id}

\newcommand{\Aut}{\operatorname {Aut}}
\def\Z{{\mathcal Z}}

\def\r{{\mathfrak r}}
\newcommand{\tr}{\hspace{-0,08cm}\triangleright}
\newcommand{\tl}{\hspace{-0,08cm}\triangleleft}

\newcommand{\cref}[1]{C.~\ref{c:#1}}

\newcommand{\inlabel}[1]{\label{in:#1}}
\newcommand{\inref}[1]{~\w\ref{in:#1}}

\newcommand{\lelabel}[1]{\label{le:#1}}
\newcommand{\leref}[1]{Lemma~\ref{le:#1}}
\newcommand{\eqlabel}[1]{\label{eq:#1}}
\newcommand{\equref}[1]{(\ref{eq:#1})}

\newcommand{\prlabel}[1]{\label{pr:#1}}
\newcommand{\prref}[1]{Proposition~\ref{pr:#1}}

\newcommand{\rmlabel}[1]{\label{rm:#1}}
\newcommand{\rmref}[1]{Remark~\ref{rm:#1}}
\newcommand{\selabel}[1]{\label{se:#1}}
\newcommand{\seref}[1]{Section~\ref{se:#1}}
\newcommand{\sslabel}[1]{\label{ss:#1}}
\newcommand{\ssref}[1]{Subsection~\ref{ss:#1}}

\begin{document}

\title{{\bf Transparency condition in the categories of Yetter-Drinfel'd modules over Hopf algebras
in braided categories}}
\author{{\large B. Femi\'c \vspace{2pt}} \\
{\small Facultad de Ingenier\'ia, \vspace{-2pt}}\\
{\small  Universidad de la Rep\'ublica} \vspace{-2pt}\\
{\small  Julio Herrera y Reissig 565, \vspace{-2pt}}\\
{\normalsize 11 300 Montevideo, Uruguay}\\
%{\normalsize fax: (+598 2) 522-0653}\\
{\normalsize e-mail: bfemic@fing.edu.uy}}

\date{June 10, 2011}
\maketitle

\vspace{-0.0cm}
\begin{abstract}
We study versions of the categories of Yetter-Drinfel'd modules over a Hopf algebra $H$
in a braided monoidal category $\C$. Contrarywise to Bespalov's approach, all our structures
live in $\C$. This forces $H$ to be transparent or equivalently to lie in M\"uger's center
$\Z_2(\C)$ of $\C$. We prove that versions of the categories of Yetter-Drinfel'd modules in
$\C$ are braided monoidally isomorphic to the categories of (left/right) modules over the
Drinfel'd double $D(H)\ö\in\C$ for $H$ finite. We obtain that these categories polarize into
two disjoint groups of mutually isomorphic braided monoidal categories. We conclude that if
$H\in\Z_2(\C)$, then ${}_{D(H)}\C$ embeds as a subcategory into the braided center category
$\Z_1({}_H\C)$ of the category ${}_H\C$ of left $H$-modules in $\C$. For $\C$ braided, rigid
and cocomplete and a quasitriangular Hopf algebra $H$ such that $H\in\Z_2(\C)$ we prove that
the whole center category of ${}_H\C$ is monoidally isomorphic to the category of left modules
over $\Aut({}_H\C)\rtimes H$ - the bosonization of the braided Hopf algebra $\Aut({}_H\C)$ which
is the coend in ${}_H\C$. A family of examples of a transparent Hopf algebras is discussed.
\end{abstract}

\vspace{0.2cm}
\indent \hspace{0.3cm} {\small 2000 Mathematics Subject Classification: 16T05, 18D10.}
%\indent \hspace{0.3cm} {\small 2000 Mathematics Subject Classification: 18D10, 18D35, 16W30.}
\vspace{0.5cm}


\setcounter{tocdepth}{1}

\renewcommand{\theequation}{\thesection.\arabic{equation}}

\section{Introduction}

Yetter introduced in \cite{Y} ``crossed bimodules'' generalizing to Hopf algebras the notion
of crossed modules over finite groups, which appeared in topology. These new objects are modules
and comodules over a Hopf algebra $H$ over a commutative ring with a certain compatibility condition.
In \cite{LR} they were used to generate solutions to the Yang-Baxter equation and accordingly were
called ``Yang-Baxter modules''. Yetter's construction and its variations were studied %by Radford and Towber
in \cite{RT} %for a bialgebra over a commutative ring $K$.
where they were termed  %The authors refer to them as to
Yetter-Drinfel'd structures. %, where the
The initial Yetter's category is denoted by ${}_H ^H\YD$. %for a Hopf algebra $H$ over a commutative ring.
%Objects of %such a category is a
%the respective categories are modules and comodules over $H$ with certain compatibility condition.
%Monoidal structures of some of Yetter-Drinfel'd categories were studied in \cite[Proposition 4.3.1]{LR},
%where the term ``Yang-Baxter modules'' is used. The respective braidings were treated in \cite{RT}.
%%\textcolor{rojo}{and these in turn are related to the solutions of
%%the Yang-Baxter equation for ${}_H ^H\YD$ studied already in \cite{Y}}.

For a finite-dimensional Hopf algebra $H$ Majid
%defined in \cite{Maj} a version $D'(H)$ of the Drinfel'd double $D(H)$ from \cite{D} and
proved that the category ${}_{D'(H)}\M$ of modules over the Drinfel'd double $D'(H)=H\bowtie H^{* op}$
is isomorphic to ${}_H ^H\YD$. %As observed by Majid, $id\ot S^*: D'(H)\to D(H)$
%is a Hopf algebra isomorphism, where $S^*$ is the antipode of the dual Hopf algebra $H^*$.
In \cite[Proposition 2.4]{Rad3} the analogous result to the former is proved for the left-right
version of the Yetter-Drinfel'd category: ${}_{D(H)}\M\iso {}_H\YD^H$, where $D(H)=(H^{op})^*\bowtie H$.

Another, categorical interpretation of the Yetter-Drinfel'd categories is that they can be seen as the
{\em center} (or the {\em inner double}) of the category of modules over the Hopf algebra. The
center construction (which to any monoidal category assigns a braided monoidal category)
%$\Z(\C)$ called the center of $\C$)
is a special case of Pontryagin dual monoidal category,
\cite{Maj1}. As observed by Drinfel'd \cite{Dp} and
%stated in \cite[Example 3.4]{Maj1}, $\Z({}_H\M$,
proved in \cite[Example 1.3]{Maj4} and \cite[Theorem XIII.5.1]{K}
the left (resp. right) center of the category of left modules over $H$ is isomorphic to ${}_H ^H\YD$
(resp. ${}_H \YD^H$). % - one can distinguish a left and a right center
For the details on the center construction we refer to \cite{K}.

In Radford biproduct Hopf algebra $B\times H$ \cite{Rad1},
%which is a smash product and smash coproduct bialgebra,
Majid observed that $B$ is a Hopf algebra in the category ${}_H ^H\YD$. If $H$ is
quasitriangular, a left $H$-module $B'$ can be equipped with a left $H$-comodule structure in
such a way that one gets a Yetter-Drinfel'd module. In this particular case,
%the process of obtaining an ordinary Hopf algebra $B\times H$ out of a categorical
%Hopf algebra $B\in{}_H\M$ Majid
the Hopf algebra $B'\times H$ is named {\em bosonization} in \cite{Maj3}. The reversed process
- recovering a braided Hopf algebra out of an ordinary one -
was studied in \cite[Section 2]{Maj3} and is called {\em mutation}.
%Concretely, if there is a Hopf algebra map $H_1\to H_2$ and $H_1$ is quasitriangular, then one
%may construct a Hopf algebra $B$ in the category of $H_1$-modules so that $B\times H_1\iso H_2$.

Yetter-Drinfel'd modules through their equivalence with Hopf bimodules, \cite{Sch},
emerge %as the basic notion
in Woronowicz's approach to bicovariant differential calculi on quantum groups, \cite{W}.
The first order differential calculi over a Hopf algebra $H$ over a field consist of a derivation
$d: H\to\Omega^1(H)$, where $\Omega^1(H)$ is the {\em bicovariant bimodule} and has a structure of a
Hopf bimodule. % termed in \cite{W}.
Another and exotic appearance of left-right Yetter-Drinfel'd modules we find in 3D-topological
quantum field theories, \cite[Theorem 3.4]{CY}.
\par\medskip

Some of the above-mentioned constructions have been generalized to any braided monoidal category.
For a Hopf algebra $H$ in a braided monoidal category $\C$ which admits split idempotents the
equivalence of the categories of Hopf bimodules %${}_H^H\C_H ^H$
and of Yetter-Drinfel'd modules $\YD(\C)_H ^H$ was proved in
\cite{BD1}. In the same paper the authors %introduce admissible pairs %$(H, X)$ in $\C$ %, where
%$X\in\C$ are right $H$-module algebras and right $H$-comodule coalgebras so that the smash
%product and the smash coproduct on $H\times X$ make it a bialgebra in $\C$. It is proved and
prove that the category of bialgebras in $\YD(\C)_H ^H$ is isomorphic to the category of admissible pairs
in $\C$. The proof relies on the previously generalized Radford-Majid theorems to the braided
case, \cite[Theorems 4.1.2 and 4.1.3]{Besp}. The former result provides a natural and easy
description for the Radford-Majid criterion for when a Hopf algebra is a cross product.
%The relation between the center category and the category of Yetter-Drinfel? modules was generalized
%to braided monoidal categories in \cite[Proposition 3.6.1]{Besp}. %The center construction for a monoidal
%category $\C$ which includes the upgraiding of duality was elaborated in \cite{KaTu} leading to
%a ribbon category $\D(\C)$ whose definition is an elaboration of that of the center of $\C$.
%It encodes ribbon extensions of the quantum double $D(H)$ of a Hopf algebra for $\C=Vec$.
%An enriched center construction which includes the upgraiding of duality and gives rise to a
%ribbon category was elaborated in \cite{KaTu}.
%An approach to the center category and Drinfel'd double via monads in a non-braided
%setting was treated in \cite{Al}.
\par\medskip

In this paper we study categories of Yetter-Drinfel'd modules over a Hopf algebra $H$ in a braided
monoidal category $\C$ with a different approach than in \cite{Besp}. Moreover, we address the question
of their isomorphism with the categories of left and right modules over the Drinfel'd double in $\C$.
When studying the monoidal structures of the respective categories, one is tempted to impose the
symmetricitity of the base category $\C$ as a necessary condition. To avoid this obstacle, Bespalov
works in \cite{Besp} both with $\C$ and with its opposite and co-opposite categories, $\C^{op}$ and
$\C^{cop}$ respectively, and with a category $\crta\C$. The opposite category of $\C$ has the same
objects as $\C$, but the arrows go in the reversed order. The braiding in $\C^{op}$ is given by
$X\ot Y\stackrel{\Phi_{Y,X}}{\leftarrow} Y\ot X$, where $\Phi$ is the braiding of $\C$. The category
$\C^{cop}$ has reversed tensor product and the braiding
$X\ot^{cop} Y=Y\ot X\stackrel{\Phi_{Y,X}}{\to} X\ot Y=Y\ot^{cop} X$. The category $\crta\C$ has the
same tensor product and its braiding is $\Phi^{-1}$.
Contrarywise, %when studying braided monoidal structures of Yetter-Drinfel'd categories
in the present paper we work only with the base
category $\C$ and investigate which conditions we have to impose in order that the construction works.
We find that it is sufficient to require that the braiding $\Phi$ of $\C$ fulfills
%\textcolor{rojo}{
$\Phi_{H,X}=\Phi_{X,H}^{-1}$ for every $X\in\C$. This condition we have encountered also in \cite{CF1}.
It had already appeared in the literature in \cite{Al1} and \cite[Definition 2.9]{M}.
In the terminology of the former reference we have that $H$ is {\em transparent}, while due to
the latter $H$ belongs to M\"uger's center
$\Z_2(\C)=\{X\in\C \vert \Phi_{Y,X}\Phi_{X,Y}=id_{X\ot Y} \hspace{0,2cm}\textnormal{for all} \hspace{0,2cm} Y\in\C\}$
of the braided monoidal category $\C$. The notation $\Z_1(\C)$ M\"uger reserved for the center of the monoidal
category $\C$ that we mentioned above. If $\Phi_{X,Y}=\Phi_{Y,X}^{-1}$ for some $X, Y\in\C$, we say that
{\em $\Phi_{X,Y}$ is symmetric}.

As a particular case of the bicrossproduct construction (with trivial coactions) in braided monoidal
categories, \cite{ZC}, we study the Drinfel'd double $D(H)$ of $H$ in $\C$. We obtain that
$D(H)=(H^{op})^*\bowtie H$ in $\C$ is a bicrossproduct Hopf algebra for finite $H$, if $\Phi_{H,H}$
%and $\Phi_{H,H^*}$ are
is symmetric. %We prove the braided version of \cite[Proposition 4.6]{Rad2}
Equivalent conditions for when $D(H)$ is (co)commutative are given. %Generalizing
%\cite[Proposition 2.4]{Rad3} to the braided case, we prove that ${}_{D(H)}\C\iso
%{}_H\YD(\C)^{H^{op}}$ as braided monoidal categories. Beside the latter, we study other versions
%of braided monoidal categories of Yetter-Drinfel'd modules and relate them with the categories of
%modules over the Drinfel'd double. We obtain
We prove that the category of modules over $D(H)$ in $\C$ is isomorphic to that of Yetter-Drinfel'd modules
over $H$ in $\C$ if $H$ is transparent. In particular, we get that the two diagrams
%the following two commutative diagrams of mutually isomorphic braided monoidal categories:
$$
\scalebox{0.84}{
\bfig
\putmorphism(-300,-160)(1,0)[{}_H\YD(\C)^{H^{op}}`{}^H\YD(\C)_{H^{cop}}`]{1000}1b
\putmorphism(650,330)(0,-1)[``]{480}1r
\putmorphism(-430,330)(1,0)[{}_{D(H)}\C`{}_H ^H\YD(\C)`]{1080}1a
\putmorphism(-350,330)(0,-1)[``]{460}1l
\putmorphism(-460,-160)(2,-1)[` `]{640}1b
\putmorphism(-570,-150)(2,-1)[`\YD(\C)^{H^{op}}_{H^{cop}}`]{700}0b
\putmorphism(330,-340)(2,1)[``]{90}1b
\put(110,60){\fbox{1}}
\efig}
\qquad\textnormal{ and}\qquad
\scalebox{0.84}{
\bfig
\putmorphism(-300,-160)(1,0)[{}_{H^{cop}}\YD(\C)^H`{}^{H^{op}}\YD(\C)_H`]{1000}1b
\putmorphism(620,330)(0,-1)[``]{480}1r
\putmorphism(-430,330)(1,0)[\C_{D(H)}`\YD(\C)_H ^H`]{1080}1a
\putmorphism(-380,330)(0,-1)[``]{460}1l
\putmorphism(-460,-160)(2,-1)[` `]{640}1b
\putmorphism(-570,-150)(2,-1)[` {}_{H^{cop}}^{H^{op}}\YD(\C)`]{700}0b
\putmorphism(330,-340)(2,1)[``]{90}1b
\put(80,60){\fbox{2}}
\efig}
$$
commute as arrows of mutually isomorphic braided monoidal categories. Our goal in this paper is not to
prove that all the above Yetter-Drinfel'd categories are braided monoidally isomorphic,
%as braided monoidal categories,
as it was proved in \cite[Corollary 3.5.5]{Besp} under the previously mentioned suppositions. Rather,
we set up a different approach and investigate how far we can get in the study of the above categories.

Bespalov proved in \cite[Proposition 3.6.1]{Besp} that the category of left-left (resp. right-right)
%one-sided
Yetter-Drinfel'd modules in $\C$ is braided monoidally isomorphic to a subcategory of the center of the
category of left $H$-modules (resp. right $H$-comodules). We differentiate the left and the right center
category and observe that the mentioned category isomorphism can be extended to the categories
in the rectangular diagrams $\langle 1\rangle$ and $\langle 2\rangle$ above yielding
%are braided monoidally isomorphic to the corresponding subcategories of  of the categories of $H$-(co)modules,
%obtaining that these polarize into two
two polarized groups of mutually isomorphic braided monoidal categories:
$$
\scalebox{0.84}{
\bfig
\putmorphism(-400,-130)(1,0)[\Z^{\C}_r({}_H\C)`\Z^{\C}_l({}^H\C)`]{1000}1b
\putmorphism(630,330)(0,-1)[``]{430}1r
\putmorphism(-430,330)(1,0)[\Z^{\C}_l({}_H\C)`\Z^{\C}_r({}^H\C)`]{1030}1a
\putmorphism(-380,320)(0,-1)[``]{420}1l
\efig}
\qquad\quad\textnormal{ and}\qquad
\scalebox{0.84}{
\bfig
\putmorphism(-400,-140)(1,0)[\Z^{\C}_r(\C^H)`\Z^{\C}_l(\C_H).`]{1020}1b
\putmorphism(660,330)(0,-1)[``]{450}1r
\putmorphism(-380,330)(1,0)[\Z^{\C}_r(\C_H)`\Z^{\C}_l(\C^H)`]{1000}1a
\putmorphism(-350,330)(0,-1)[``]{440}1l
\efig}
$$
%(Here we differ the left and the right braided center categories.)
As for the relation between the centers $\Z_1$ and $\Z_2$ in the notation of
%the centers $\Z_1$ and $\Z_2$ due to
M\"uger, we obtain in particular that if $H\in\Z_2(\C)$, %$H\in\Z_2({}_{D(H)}\C)$,
then $%{}_H ^H\YD(\C)\iso
{}_{D(H)}\C \hookrightarrow \Z_{1, l}({}_H\C)$ (and similarly $%\YD(\C)_H ^H\iso
\C_{D(H)} \hookrightarrow \Z_{1, r}(\C_H)$). %This illustrates a

%For the whole center category of a braided, rigid and cocomplete category $\C$ Majid proved
%$\Z_{1, l}(\C)\iso\C_C$ in \cite{Maj1a}, where $C$ is the coend Hopf algebra in $\C$. %When $\C$
%For a quasitriangular Hopf algebra $H\in\C$ \cite[Definition 1.3]{Maj2} such that
%$H\in\Z_2(\C_H)$ we prove that $\Z_{1, l}(\C_H)\iso\C_{H\ltimes\Aut(\C_H)}$ as monoidal categories,
%where $H\ltimes\Aut(\C_H)$ is the bosonization of the braided Hopf algebra $\Aut(\C_H)$ in $\C_H$.
%When $\C=Vec$ and $H$ is a finite-dimensional quasitriangular Hopf algebra, this yields
%$\Z_l(\M_H)\iso\M_{D'(H)}$.
For the whole center category of a braided, rigid and cocomplete category $\C$ Majid proved
$\Z_{1, l}(\C)\iso\C_{\Aut(\C)}$ in \cite{Maj1a}, where $\Aut(\C)$ is the coend Hopf algebra in $\C$.
%When $\C$
%is the category of modules over a finite-dimensional quasitriangular Hopf algebra $H$ this
%yields $\Z_{1, l}(\M_H)\iso\M_{D'(H)}$. %, where $\crta{D(H)}$ is the Drinfel'd double with the
%underlying vector space $H\ot H^*$.
%For general $\C$ and a Hopf algebra $H\in\C$ quasitriangular in the sense of Majid,
%not the whole category $\C_H$ is braided, so we can not discuss the above result for $\Z_l(\C_H)$.
%This difficultly is overcome in the construction by Brugui\`ere and Virelizier in \cite{Al}
%which relies on monads. Using a different definition of the quasitriangular structure and
%the Drinfel'd double $\Dd_H$, whose underlying object is $H\ot {}^*H\ot C$, they prove the isomorphism
%$\Z_l(\C_H)\iso\C_{\Dd_H}$.
For %$\C$ as above and
a quasitriangular Hopf algebra $H\in\C$ \cite[Definition 1.3]{Maj2} such that
%$\Phi_{H, H}$ and $\Phi_{H, M}$ are symmetric for all $M\in\C_H$,
$H\in\Z_2(\C)$ we obtain $\Z_{1, l}(\C_H)\iso\C_{H\ltimes\Aut(\C_H)}$ as monoidal categories,
where $H\ltimes\Aut(\C_H)$ is the bosonization of the braided Hopf algebra $\Aut(\C_H)$
%which is the coend
in $\C_H$. When $\C=Vec$ and $H$ is a finite-dimensional quasitriangular Hopf algebra, this
recovers the known isomorphism $\Z_l(\M_H)\iso\M_{D'(H)}$. We point out that a similar result to
ours was proved in \cite{Al} where the authors work with Hopf monads and construct a Drinfel'd
double in a fully non-braided setting.
%\par\medskip

At the end we present a family of transparent Hopf algebras in braided monoidal categories
which support our constructions.
\par\medskip

The paper is organized as follows. In \seref{prel} we present preliminaries on some structures
in any braided monoidal category $\C$. In the next section we study the braided monoidal category
of left-right Yetter-Drinfel'd modules ${}_H\YD(\C)^{H^{op}}$ (assuming that $H$ is transparent).
%$\Phi_{H,M}$ is symmetric for all Yetter-Drinfel'd modules $M$).
We point out that the categories ${}_H^H\YD(\C)$ and $\YD(\C)_H^H$ are  braided monoidal without any
symmetricity conditions on the braiding. \seref{bicruz} recalls the bicrossproduct construction
(with trivial coactions) in $\C$. We use it to study the Drinfel'd double $D(H)=(H^{op})^*\bowtie H$
in $\C$ for a finite $H$, when $\Phi_{H,H}$ is symmetric. \seref{YD-DH} is devoted to the braided monoidal
isomorphism ${}_{D(H)}\C\iso {}_H\YD(\C)^{H^{op}}$.
In \seref{other-cases} we compare different versions of the braided Yetter-Drinfel'd categories
in $\C$, connecting them with the categories of left and right modules over the Drinfel'd double
in $\C$. In the penultimate section we deal with the center construction and relate it to the Yetter-Drinfel'd
categories. The last section presents some examples.
%We point out the existing results from \cite{Maj1a} and \cite{Al} on the center category.
\par\bigskip

{\bf Acknowledgements.} This work has partially been developed in the Mathematical Institute of
the Serbian Academy of Sciences and Arts in Belgrade (Serbia). The author wishes to thank to
Facultad de Ciencias de la Universidad de la Rep\'ublica in Montevideo for their worm
hospitality and provision of the necessary facilities.
%Special thanks to the library of the faculty for the collaboration.
My gratitude to Yuri Bespalov for clarifying me his proof of \cite[Proposition 3.6.1]{Besp},
and to Alain Brugui\`ere for the discussions on the construction of the Drinfel'd double via monads.

\section{Preliminaries} \selabel{prel}

We assume the reader is familiar with %the general category theory and
the theory of braided monoidal categories as well as with the notation of braided diagrams.
For the references we recommend \cite{K} and \cite{Besp}.
We recall that a Hopf algebra in a braided monoidal category $\C$ was introduced by Majid in \cite{Maj4}.
In the same paper it was proved that the categories of modules and comodules over a bialgebra in $\C$
are monoidal. We only outline some basic conventions. In view of Mac Lane's Coherence Theorem we
will assume that our braided monoidal category $\C$ is strict. Our braided diagrams
are read from top to bottom, the braiding $\Phi:X\ot Y\to Y\ot X$ and its inverse in $\C$ we denote by:
$$\Phi_{X, Y}=
\gbeg{1}{3}
\got{1}{X} \got{1}{Y} \gnl
\gbr \gnl
\gob{1}{Y} \gob{1}{X}
\gend
\quad\qquad\textnormal{and}\qquad
\Phi_{Y, X}^{-1}=
\gbeg{1}{3}
\got{1}{Y} \got{1}{X} \gnl
\gibr \gnl
\gob{1}{X} \gob{1}{Y.}
\gend
$$
%We only outline that f
For an algebra $A\in\C$ and a coalgebra $C\in\C$ %in a braided monoidal category $\C$
the multiplication in the opposite algebra $A^{op}$ of $A$ and the comultiplication in the co-opposite
coalgebra $C^{cop}$ of $C$ we denote by:
$$\nabla_{A^{op}}=
\scalebox{0.86}{
\gbeg{3}{4}
\got{1}{A} \got{1}{A} \gnl
\gbr \gnl
\gmu \gnl
\gob{2}{A}
\gend} \qquad\textnormal{and}\qquad
\Delta_{C^{cop}}=
\scalebox{0.86}{
\gbeg{4}{4}
\got{2}{C} \gnl
\gcmu \gnl
\gibr \gnl
\gob{1}{C} \gob{1}{C}
\gend}
$$
respectively. The antipode $S$ of a Hopf algebra $H$ in $\C$ is a bialgebra map
$S: H\to H^{op, cop}$. Its compatibility with multiplication
and comultiplication is written as:
$$
\gbeg{2}{5}
\got{1}{H} \got{1}{H} \gnl
\gmu \gnl
\gvac{1} \hspace{-0,22cm} \gnot{S} \gmp \gnl \gnl
\gvac{1} \gcn{1}{1}{1}{1} \gnl
\gob{3}{H}
\gend=
\gbeg{2}{5}
\got{1}{H} \got{1}{H} \gnl
\gnot{S} \gmp \gnl \gnot{S} \gmp \gnl \gnl
\gbr \gnl
\gmu \gnl
\gob{2}{H}
\gend \qquad\textnormal{and}\qquad
\gbeg{2}{5}
\got{2}{H} \gnl
\gcmu \gnl
\gnot{S} \gmp \gnl \gnot{S} \gmp \gnl \gnl
\gcl{1} \gcl{1} \gnl
\gob{1}{H} \gob{1}{H}
\gend=
\gbeg{2}{5}
\got{2}{H} \gnl
\gvac{1} \hspace{-0,38cm} \gnot{S} \gmp \gnl \gnl
\gvac{1} \hspace{-0,22cm} \gcmu \gnl
\gvac{1} \gibr \gnl
\gvac{1} \gob{1}{H} \gob{1}{H}
\gend
$$
respectively. Moreover, $S$ is the antipode for $H^{op, cop}$.
Note that for a bialgebra $B\in\C$, neither $B^{op}$ nor $B^{cop}$ is a bialgebra, unless the
braiding $\Phi$ fulfills $\Phi_{B, B}=\Phi_{B, B}^{-1}$. \label{op,cop-bialg}
\par\medskip

We recall some basic facts.

\begin{prel}\inlabel{inner-hom}
A monoidal category $\C$ is called {\em right closed} if the functor $-\otimes M:
\C \rightarrow \C$ has a right adjoint, denoted by $[M,-]$, for all $M \in \C$. For $N\in\C$,
the object $[M, N]$ is called {\em inner hom-object}. The counit of the adjunction evaluated
at $N$ is denoted by $\ev_{M,N}:[M,N]\otimes M \rightarrow N$. It satisfies
the following universal property:
{\it for any morphism $f:T \otimes M \rightarrow N$ there is a unique morphism
$g:T \rightarrow [M,N]$ such that $f=\ev_{M,N}(g \otimes M)$}. If $f: N\to N'$ is a morphism in $\C$, then
$[M, f]:[M, N]\to [M, N']$ is the unique morphism such that $\ev_{M,N'}([M, f]\ot M)=f\hspace{0,1cm}\ev_{M,N}$.
The unit of the adjunction $\alfa: N\to [M, N\ot M]$ is induced by $\ev_{M,N\ot M}(\alpha\ot M)=id_{N\ot M}$.
A monoidal category $\C$ is called {\em left closed} if the functor $M\otimes -: \C \rightarrow \C$
has a right adjoint $\{M,-\}$ for
all $M\in\C$. The counit of this adjunction evaluated at $N\in\C$ is denoted by $\crta\ev_{M,N}:M\otimes \{M,N\}\rightarrow N$ and the unit by $\tilde\alpha: N\to \{M, M\ot N\}$. It obeys $\crta\ev_{M,M\ot N}(M\ot\tilde\alpha)=id_{M\ot N}$. When $\C$ is braided,
there is a natural equivalence of functors $[M, -]\iso \{M, -\}$ and $\C$ is right closed if and only
if it is left closed. Throughout we will write $[-, -]$ for both types of inner hom-bifunctors, the difference will be clear from the context.
The object $[M, M]$ is an algebra for all $M \in \C$.
\end{prel}

\begin{prel} \inlabel{r.adj. dual}
Let $P$ be an object in $\C$. An object $P^*\in\C$ together with a morphism
$e_P:P^*\ot P \to I$ is called a {\em left dual object} for $P$ if there exists a morphism
$d_P:I\to P\ot P^*$ in $\C$ such that $(P \otimes e_P)(d_P \otimes P)=id_P$ and
$(e_P \otimes P^*)(P^* \otimes d_P)=id_{P^*}$. The morphisms $e_P$ and $d_P$ are called
{\em evaluation} and {\em dual basis}, respectively. In braided diagrams the evaluation
$e_P$ and dual basis $d_P$ are denoted by: \vspace{-0,4cm}
$$e_P=
\gbeg{3}{1}
\got{1}{\hspace{0,12cm}P^*} \got{1}{\hspace{0,1cm}P} \gnl
\gev \gnl
\gob{1}{}
\gend
\qquad\textnormal{and}\qquad
d_P=
\gbeg{3}{3}
\got{1}{} \gnl
\gdb \gnl
\gob{1}{P} \gob{2}{\hspace{-0,2cm}P^*}
\gend$$
and the two identities they satisfy by:
\vspace{-0,7cm}
\begin{center} \hspace{-0,4cm}
\begin{tabular}{p{4.4cm}p{1.6cm}p{4.8cm}}
\begin{equation} \eqlabel{dbev=id}
\scalebox{0.9}[0.9]{
\gbeg{3}{4}
\got{1}{} \got{3}{P} \gnl
\gdb \gcl{1} \gnl
\gcl{1} \gev \gnl
\gob{1}{P}
\gend}=id_P
\end{equation} & &
%\qquad\textnormal{and}\qquad &  &
\begin{equation} \eqlabel{evdb=id}
\scalebox{0.9}[0.9]{
\gbeg{3}{4}
\got{1}{P^*} \got{1}{} \gnl
\gcl{1} \gdb \gnl
\gev \gcl{1} \gnl
\gob{1}{} \gob{4}{\hspace{-0,1cm}P^*}
\gend}=\id_{P^*}.
\end{equation}
\end{tabular}
\end{center} \vspace{-0,7cm}
Symmetrically, one defines a right dual object ${}^*P$ for $P$ with morphisms $e'_P:P\ot {}^*P \to I$
and $d'_P:I\to {}^*P\ot P$. Left and right dual objects are unique up to isomorphism. In a braided
monoidal category the left and the right dual for $P$ coincide. The corresponding evaluation and dual basis
morphisms are related via:
\vspace{-0,7cm}
\begin{center} \hspace{-0,4cm}
\begin{tabular}{p{4.8cm}p{1cm}p{4.8cm}}
\begin{equation}\eqlabel{crta ev}
e_P'=e_P\Phi_{P^*, P}
\end{equation} & &
\begin{equation}\eqlabel{crta d}
d_P'=\Phi_{P, P^*}^{-1} d_P
\end{equation}
\end{tabular}
\end{center} \vspace{-0,7cm}
%where $\Phi$ is the braiding for $\C$. For the latter
see e.g. \cite[Prop. 2.13, b)]{Tak2} (we take here the opposite sign of the first power of the braiding).
\end{prel}

\begin{prel}
An object $P\in\C$ is called {\em right finite}, if $[P,I]$ and $[P,P]$ exist
%$[P,I]$ is $[P,P]$-coflat
and the morphism $\db: P \ot [P,I] \to [P,P]$, called the {\em dual basis morphism} as well,
defined via the universal property of $[P,P]$ by $\ev_{P,P}(\db\ot P)=P\ot \ev_{P,I}$ is an
isomorphism. One may easily prove that {\it if $P$ is right finite, then $([P,I], e_P=\ev)$ is its left
dual}. The dual basis morphism is $d_P=db^{-1}\eta_{[P,P]}$, where $\eta_{[P,P]}$ is the unit for the
algebra $[P,P]$. A similar claim holds for a left finite object, which is defined similarly  as
a right finite object. In a braided monoidal category an object is
left finite if and only if it is right finite. If $P$ is a finite object, then so is $P^*$ and there is a natural isomorphism
$P\iso P^{**}$.
\end{prel}

\begin{prel} \inlabel{alg-coalg}
In the following we collect some facts about duality of Hopf algebras from
\cite[2.5, 2.14 and 2.16]{Tak2}.
%\begin{prop} %\prlabel{alg<->coalg}
Let $\C$ be a closed braided monoidal category.

(i) If $H$ is a coalgebra in $\C$, then $H^*:=[H, I]$ is an algebra.

(ii) If $H$ is a finite algebra in $\C$, then $H^*$ is a coalgebra.

(iii) If $H$ is a finite Hopf algebra in $\C$, then so is $H^*$.

%(iv) A finite algebra $A$ in $\C$ is commutative if and only if $A^*$ is a cocommutative coalgebra.
%\end{prop}

We give here the structure morphisms. The finiteness condition in ii) and  iii)
is needed in order to be able to consider $H^*\ot H^*\iso(H\ot H)^*$, which allows to define a
comultiplication on $H^*$ using the universal property of $[H\ot H, I]$.
The multiplication, comultiplication, antipode $S^*$, unit and counit of $H^*$ are given by:
\vspace{-0,6cm} \begin{center} \hspace{-0,4cm}
\begin{tabular}{p{3.8cm}p{1cm}p{2,4cm}p{2cm}p{4cm}}
\begin{eqnarray} \label{multH*}
\scalebox{0.78}[0.78]{ \gbeg{4}{3}
\got{1}{H^*} \got{3}{H^*} \got{1}{\hspace{-0,66cm} H} \gnl \gwmu{3} \gcl{1} \gnl
\gvac{1} \gcl{1} \gcn{2}{1}{3}{1} \gnl \gvac{1} \gev \gvac{1} \gnl \gvac{2} \gob{2}{}
\gend} = \scalebox{0.78}[0.78]{ \gbeg{4}{3} \got{1}{H^*} \got{1}{H^*} \got{3}{\hspace{-0,38cm}H} \gnl
\gcl{2} \gcl{1} \gcmu \gnl \gvac{1} \gbr \gcl{1} \gnl \gev \gev \gnl \gvac{2} \gob{1}{} \gend}
\end{eqnarray} & & \begin{eqnarray} \label{codiagH*}
\scalebox{0.78}[0.78]{ \gbeg{4}{3}
\got{2}{H^*} \got{1}{H}\got{1}{H} \gnl \gcmu \gcl{1} \gcl{2} \gnl \gcl{1} \gbr \gnl \gev \gev \gnl
\gvac{2} \gob{1}{} \gend} = \scalebox{0.78}[0.78]{ \gbeg{4}{3} \got{1}{H^*} \got{1}{H} \got{3}{H} \gnl
\gcl{1} \gwmu{3} \gnl \gcn{2}{1}{1}{3} \gcl{1} \gnl \gvac{1} \gev \gvac{1} \gnl \gvac{2} \gob{2}{} \gend}
\end{eqnarray} & & \begin{eqnarray} \label{antipodeH*}
\scalebox{0.78}[0.78]{ \gbeg{2}{3}
\got{1}{H^*} \got{1}{H} \gnl \gnot{\hspace{0,12cm}S^*} \gmp \gnl \gcl{1} \gnl \gev \gnl \gob{1}{}
\gend} = \scalebox{0.78}[0.78]{ \gbeg{1}{3} \got{1}{H^*} \got{1}{H} \gnl \gcl{1} \gnot{S} \gmp \gnl \gnl
\gev \gnl \gvac{1} \gob{2}{} \gend} \end{eqnarray} \end{tabular} \end{center} \vspace{-1,6cm}
\begin{center} \hspace{-3,6cm} \begin{tabular}{p{3.44cm}p{0,44cm}p{7cm}} \begin{eqnarray} \label{unitH*}
\scalebox{0.78}[0.78]{ \gbeg{3}{3} \got{1}{} \got{1}{H} \gnl \gu{1} \gcl{1} \gnl \gev \gnl
\gvac{1} \gob{2}{} \gend} = \scalebox{0.78}[0.78]{ \gbeg{1}{3} \got{1}{H} \gnl \gcl{1} \gnl \gcu{1} \gnl
\gob{2}{} \gend} \end{eqnarray} & & \begin{eqnarray} \label{counitH*}
\scalebox{0.78}[0.78]{ \gbeg{1}{3}
\got{1}{H^*} \gnl \gcl{1} \gnl \gcu{1} \gnl \gob{1}{} \gend} = \scalebox{0.78}[0.78]{ \gbeg{4}{3}
\got{1}{H^*} \got{1}{} \gnl \gcl{1} \gu{1} \gnl \gev \gnl \gvac{1} \gob{2}{} \gend} \end{eqnarray}
\end{tabular} \end{center} \vspace{-0,8cm} respectively (one uses the universal property of $[H, I]$).
It is easy to see that a finite algebra $A$ in $\C$ is commutative if and only if $A^*$ is a
cocommutative coalgebra.
\end{prel}

%\vspace{0,4cm}

For an algebra $A\in\C$ and a coalgebra $C\in\C$ we denote by ${}_A\C$ and $\C^C$ the categories of
left $A$-modules and right $C$-comodules, respectively. The proof of the following proposition is not
difficult. The first statement is proved in \cite[Proposition 2.7]{Tak2}.

\begin{prop} \prlabel{Hmod-H*comod}
Let $H\in \C$ be a finite coalgebra. If $M\in\C^H$, then $M\in {}_{H^*}\C$ with the structure morphism given
in (\ref{H*mod-S}). If $N\in {}_{H^*}\C$, then $N\in\C^H$ with the structure morphism given in (\ref{Hcomod}).
These assignments make the categories $\C^H$ and ${}_{H^*}\C$ isomorphic. \vspace{-1cm}
\begin{center}
\begin{tabular}{p{6.8cm}p{1cm}p{4.8cm}}
\begin{eqnarray} \label{H*mod-S}
\scalebox{0.9}[0.9]{
\gbeg{2}{3}
\got{1}{H^*} \got{1}{\hspace{0,1cm}M} \gnl
\glm \gnl
\gvac{1} \gob{1}{M}
\gend} = \scalebox{0.9}[0.9]{
\gbeg{2}{5}
\got{1}{H^*} \got{1}{M} \gnl
\gcl{1} \grcm \gnl
\gbr \gcl{1} \gnl
\gcl{1} \gev \gnl
\gob{1}{M} \gvac{1} \gob{1}{}
\gend}
\end{eqnarray} & &
\begin{eqnarray}\label{Hcomod}
\scalebox{0.9}[0.9]{
\gbeg{2}{3}
\got{1}{N} \gnl
\grcm \gnl
\gob{1}{N} \gob{1}{H}
\gend} = \scalebox{0.9}[0.9]{
\gbeg{7}{5}
\got{1}{} \got{1}{} \got{1}{N} \gnl
\gdb \gcl{1} \gnl
\gcn{1}{1}{1}{3} \glm \gnl
\gvac{1} \gbr \gnl
\gvac{1} \gob{1}{N} \gob{1}{H}
\gend}
\end{eqnarray}
\end{tabular}
\end{center}
\end{prop}

%%%%%%%%%%%%%%%%%%%%%%% ASSUMPTIONS  %%%%%%%%%%%%%%%%%%%%%%%%%%%%%%%%%%

Throughout the paper $\C$ will be a braided monoidal category with braiding $\Phi$ and $H\in\C$ a Hopf
algebra having a bijective antipode. %we will assume that

%%%%%%%%%%%%%%%%%%%%%%%%%%%%%%%%%%%%%%%%%%%%%%%%%%%%%%%%%%%%%%%

\section{Some braided monoidal categories of Yetter-Drinfel'd modules} \selabel{YD-mods}
\setcounter{equation}{0}

A left $H$-module and left $H$-comodule $N\in\C$ and a right $H$-module and right $H$-comodule
$L\in\C$ are called respectively {\em left-left} and {\em right-right Yetter-Drinfel'd modules} over $H$ in $\C$
if they obey the compatibility conditions:
\begin{center}
\begin{tabular}{p{6cm}p{0cm}p{7.4cm}}
\begin{equation} \eqlabel{left YD}
%\textcolor{rojo}{
%\gbeg{4}{8}
%\got{2}{H} \got{1}{N} \gnl
%\gcmu \gcl{2} \gnl
%\gbr \gnl
%\gcl{1} \glm \gnl
%\gcl{1} \glcm \gnl
%\gibr \gcl{1} \gnl
%\gmu \gcl{1} \gnl
%\gob{2}{H} \gob{1}{N}
%\gend}=
\gbeg{4}{8}
\got{2}{H} \got{1}{N} \gnl
\gcmu \gcl{1} \gnl
\gcl{1} \gbr \gnl
\glm \gcl{2} \gnl
\glcm \gnl
\gcl{1} \gbr \gnl
\gmu \gcl{1} \gnl
\gob{2}{H} \gob{1}{N}
\gend=
\gbeg{4}{5}
\got{2}{H} \got{3}{N} \gnl
\gcmu \glcm \gnl
\gcl{1} \gbr \gcl{1} \gnl
\gmu \glm \gnl
\gob{2}{H} \gob{3}{N}
\gend
\quad
\end{equation}
& &
\begin{equation} \eqlabel{right YD} \textnormal{and}\qquad\quad
\gbeg{3}{8}
\got{1}{L} \got{2}{H} \gnl
\gcl{1} \gcmu \gnl
\gbr \gcl{1} \gnl
\gcl{1} \grm \gnl
\gcl{1} \grcm \gnl
\gbr \gcl{1} \gnl
\gcl{1} \gmu \gnl
\gob{1}{L} \gob{2}{H}
\gend=
\gbeg{4}{5}
\got{1}{L} \got{4}{H} \gnl
\grcm \gcmu \gnl
\gcl{1} \gbr \gcl{1} \gnl
\grm \gmu \gnl
\gob{1}{L} \gob{4}{H}
\gend
\end{equation}
\end{tabular}
\end{center}
respectively. A {\em left-right Yetter-Drinfel'd module} over $H$ is a left $H$-module and  right
$H$-comodule $M\in\C$ whose $H$-structures are related via the relation:
\begin{eqnarray} \label{YD-mix}
%\gbeg{3}{8}
%\got{2}{H} \got{1}{M} \gnl
%\gcmu \gcl{1} \gnl
%\gcn{1}{1}{1}{3} \glm \gnl
%\gvac{1} \gibr \gnl
%\gcn{1}{1}{3}{1} \gvac{1} \gcl{2} \gnl
%\grcm \gnl
%\gcl{1} \gmu \gnl
%\gob{1}{M} \gob{2}{H}
%\gend=
\gbeg{4}{9}
\got{3}{H} \got{1}{M} \gnl
\gwcm{3} \gcl{1} \gnl
\gcn{1}{2}{1}{3} \gvac{1} \glm \gnl
\gvac{1} \gcn{1}{1}{5}{3} \gnl
\gvac{1} \gibr \gnl
\gcn{1}{1}{3}{1} \gcn{1}{1}{3}{5} \gnl
\grcm \gcn{1}{1}{3}{3} \gnl
\gcl{1} \gwmu{3} \gnl
\gob{1}{M} \gob{3}{H}
\gend=
\gbeg{3}{7}
\got{2}{H} \got{1}{M} \gnl
\gcn{1}{1}{2}{2} \gcn{1}{1}{3}{3} \gnl
\gcmu \grcm \gnl
\gcl{1} \gbr \gcl{1} \gnl
\glm \gmu \gnl
\gcn{1}{1}{3}{3} \gcn{1}{1}{4}{4} \gnl
\gob{3}{M} \gob{1}{\hspace{-0,26cm}H.}
\gend
\end{eqnarray}
In all the cases we will shorten the term ``Yetter-Drinfel'd module'' to {\em YD-module}.
The categories of left-left YD-modules and left $H$-linear and left $H$-colinear morphisms in $\C$
(which we denote by ${}_H ^H\YD(\C)$) and that of right-right YD-modules and right $H$-linear
and right $H$-colinear morphisms in $\C$ (denoted by $\YD(\C)_H^H$) respectively, are known to be
braided monoidal categories with braidings:
\begin{equation} \eqlabel{braid-LR}
\Phi_{X,Y}^L=
\gbeg{3}{5}
\got{1}{} \got{1}{X} \got{1}{Y} \gnl
\glcm \gcl{1} \gnl
\gcl{1} \gbr \gnl
\glm \gcl{1} \gnl
\gob{1}{} \gob{1}{Y} \gob{1}{X}
\gend
\quad\textnormal{and}\quad
\Phi_{W,Z}^R=
\gbeg{3}{5}
\got{1}{W} \got{1}{Z} \gnl
\gcl{1} \grcm \gnl
\gbr \gcl{1} \gnl
\gcl{1} \grm \gnl
\gob{1}{Z} \gob{1}{W}
\gend
\end{equation}
for objects $X, Y\in{}_H ^H\YD(\C)$ and $W, Z\in\YD(\C)_H^H$ respectively,
(see e.g. \cite{Besp}). However, in order that the category of left-right
YD-modules be braided monoidal, some symmetricity conditions on the braiding in $\C$
should be assumed, as we will see further below. Like in \cite[Thm. 3.4.3]{Besp} Bespalov
has that the category ${}_H\YD(\C)^{H^{op}}$ %${}_{H^{cop}}\YD(\C)^H$
is braided monoidal, but there he considers the tensor product of two left-right YD-modules a
%left $H^{cop}$-module via the codiagonal structure in the category $\crta\C$,
%whereas the $H$-comodule structure he considers in $\C$
right $H^{op}$-comodule via the codiagonal structure in the category $\crta\C$,
whereas the $H$-module structure he considers in $\C$
(as in \cite[Lemma 3.3.2]{Besp}).
Thus for two objects $M, N$ of this category, the object $M\ot N$ has the $H$-comodule structure:
$$\ro_{M\ot N}=(M\ot N\ot\nabla_{H^{op}})(M\ot\Phi_{N, H}^{-1}\ot H)(\ro_M\ot\ro_N).$$
Bespalov considers $\nabla_{H^{op}}=\nabla\Phi_{H,H}^{-1}$ (instead, we regard here the positive
sign of the braiding) in order that $H^{op}$ be a bialgebra in $\crta\C$. In the present paper
we prefer to consider {\em all} the structures in $\C$.
Accordingly, we will have that the categories ${}_H\YD(\C)^{H^{op}}$ and ${}_{H^{cop}}\YD(\C)^H$
are braided monoidal if the braiding $\Phi$ in $\C$ fulfills $\Phi_{H, X}=\Phi_{X, H}^{-1}$ for
every corresponding YD-module $X\in\C$. We will say that {\em $\Phi_{H, X}$ is symmetric}.
As a matter of fact, if $\Phi_{H, H}$ and $\Phi_{H, X}$ are symmetric (indeed $H$ itself is a YD-module
over itself), then the upper structure coincides with the usual codiagonal comodule structure on
$M\ot N$ in $\C$. Nevertheless, we will prove explicitly the claims by our approach as this is
the general setting of our work and we will prove also other results in this manner.
%Let us set up the notation for general YD-categories. Lower index $H^{op}$ denotes that the objects
%are $H^{op}$-modules, whereas lower index $H^{cop}$ denotes that the tensor product of two objects
%has an $H$-module structure given via the codiagonal structure of the coalgebra $H^{cop}$, that is
%$\Delta^{op}$, where $\Delta$ denotes the codiagonal of $H$. Upper index $H^{cop}$ stands for the
%objects which are $H^{cop}$-comodules, whereas upper index $H^{op}$ denotes that the tensor product
%of two objects is an $H^{op}$-module. When in the lower and the upper index there are two different
%Hopf algebras obtained from $H$ (e.g. $H^{op}$ and $H^{cop}$),
%The lower and the upper indexes for the Hopf algebra refer to the (co)module structures of the objects,
%whereas in the YD-compatibility condition the multiplication and the comultiplication are taken to be
%the ones from the Hopf algebra $H$, if not otherwise specified. When the Hopf algebras in the lower
%and upper indexes coincide, as we saw in the above Lemma, we consider the mentioned operations to
%be those of the Hopf algebra in question.
 %In the notation of the two categories of left-right YD-modules the indexes $H^{op}$ and $H^{cop}$
%refer to the monoidal structures of right comodules and left modules, respectively. %\par\medskip
%In this paper we are going to study only those monoidal YD-categories which have
%in their notation either $H$ or $H^{cop}$ in the lower indexes, and $H$ or $H^{op}$ in the upper ones.

\vspace{0,2cm}

Before proving that the category ${}_H\YD(\C)^{H^{op}}$ is braided monoidal, we will note some important
facts. Observe that:
\begin{eqnarray} \label{antipode2}
\gbeg{2}{6}
\got{2}{H} \gnl
\gcmu \gnl
\gmp{-} \gcl{1} \gnl
\gibr \gnl
\gmu \gnl
\gob{2}{H}
\gend=
\gbeg{2}{6}
\got{1}{H} \gnl
\gcl{1} \gnl
\gcu{1} \gnl
\gu{1} \gnl
\gcl{1} \gnl
\gob{1}{H}
\gend
\end{eqnarray}
since:
$$
\gbeg{2}{7}
\got{2}{H} \gnl
\gcmu \gnl
\gmp{-} \gcl{1} \gnl
\gibr \gnl
\gmu \gnl
\gvac{1} \hspace{-0,22cm} \gmp{+} \gnl
\gob{3}{H}
\gend=
\gbeg{2}{8}
\got{2}{H} \gnl
\gcmu \gnl
\gmp{-} \gcl{1} \gnl
\gibr \gnl
\gmp{+} \gmp{+} \gnl
\gbr \gnl
\gmu \gnl
\gob{2}{H}
\gend=
\gbeg{2}{5}
\got{2}{H} \gnl
\gcmu \gnl
\gcl{1} \gmp{+} \gnl
\gmu \gnl
\gob{2}{H}
\gend=
\gbeg{2}{6}
\got{1}{H} \gnl
\gcl{1} \gnl
\gcu{1} \gnl
\gu{1} \gnl
\gcl{1} \gnl
\gob{1}{H}
\gend=
\gbeg{2}{7}
\got{1}{H} \gnl
\gcl{1} \gnl
\gcu{1} \gnl
\gu{1} \gnl
\gmp{+} \gnl
\gcl{1} \gnl
\gob{1}{H.}
\gend
$$
From this point on we will assume that the antipode of $H$ is bijective (which is fulfilled for example
if $H$ is finite and $\C$ has equalizers, \cite[Theorem 4.1]{Tak2}).
The sign ``$+$'' stands for the antipode whereas ``$-$'' stands for the inverse of the antipode.
Furthermore, we have that the condition (\ref{YD-mix}) is equivalent to:
\begin{eqnarray} \label{YD-S}
\gbeg{6}{9}
\got{1}{H} \got{5}{M} \gnl
%\gcn{1}{1}{1}{3} \gvac{2} \gcl{2} \gnl
%\gvac{1} \gcn{1}{1}{1}{3} \gnl
\gcn{2}{2}{1}{5} \gvac{1} \gcl{2} \gnl \gnl
\gvac{2} \glm \gnl
\gvac{3} \gcl{1} \gnl
\gvac{3} \grcm \gnl
\gvac{3} \gcl{2} \gcn{2}{2}{1}{5} \gnl
\gob{7}{M} \gob{1}{\hspace{-0,64cm}H}
\gend=
\gbeg{6}{10}
\gvac{1} \got{2}{H} \got{3}{M} \gnl
\gvac{1} \gcmu \gvac{1} \gcl{2} \gnl
\gvac{1} \gcl{1} \gcn{1}{1}{1}{2} \gnl
\gvac{1} \gmp{-} \gcmu \grcm \gnl
\gvac{1} \gcl{1} \gcl{1} \gbr \gcl{1} \gnl
\gvac{1} \gcn{1}{1}{1}{3} \glm \gmu \gnl
\gvac{1} \gvac{1} \gibr \gcn{1}{1}{2}{1} \gnl
\gvac{1} \gvac{1} \gcl{2} \gibr \gnl
\gvac{1} \gvac{2} \gmu \gnl
\gvac{1} \gob{3}{M} \gob{1}{\hspace{-0,26cm}H}
\gend
\end{eqnarray}
To prove this assume that (\ref{YD-mix}) holds. Then:
$$
\gbeg{6}{10}
\got{2}{H} \got{3}{M} \gnl
\gcmu \gvac{1} \gcl{2} \gnl
\gcl{1} \gcn{1}{1}{1}{2} \gnl
\gmp{-} \gcmu \grcm \gnl
\gcl{1} \gcl{1} \gbr \gcl{1} \gnl
\gcn{1}{1}{1}{3} \glm \gmu \gnl
\gvac{1} \gibr \gcn{1}{1}{2}{1} \gnl
\gvac{1} \gcl{2} \gibr \gnl
\gvac{2} \gmu \gnl
\gob{3}{M} \gob{1}{\hspace{-0,26cm}H}
\gend\stackrel{(\ref{YD-mix})}{=}
\gbeg{6}{13}
\got{3}{H} \got{3}{M} \gnl
\gwcm{3} \gvac{1} \gcl{2} \gnl
\gmp{-} \gwcm{3} \gcl{1} \gnl
\gcl{6} \gcn{1}{2}{1}{3} \gvac{1} \glm \gnl
\gvac{1} \gvac{1} \gcn{1}{1}{5}{3} \gnl
\gvac{1} \gvac{1} \gibr \gnl
\gvac{1} \gcn{1}{1}{3}{1} \gcn{1}{1}{3}{5} \gnl
\gvac{1} \grcm \gcn{1}{1}{3}{3} \gnl
\gvac{1} \gcl{1} \gwmu{3} \gnl
\gibr \gcn{1}{1}{3}{1} \gnl
\gcl{2} \gibr \gnl
\gvac{1} \gmu \gnl
\gob{1}{M} \gob{2}{H}
\gend\stackrel{coass.}{\stackrel{ass.}{\stackrel{nat.}{=}}}
\gbeg{5}{13}
\got{5}{H} \got{1}{\hspace{-0,6cm}M} \gnl
\gvac{1} \gwcm{3} \gcl{1} \gnl
\gwcm{3} \glm \gnl
\gmp{-} \gcn{1}{1}{3}{5} \gvac{2} \gcl{1} \gnl
\gcl{1} \gvac{2} \gibr \gnl
\gcn{1}{1}{1}{3} \gvac{1} \gcn{1}{1}{3}{1} \gvac{1} \gcl{3} \gnl
\gvac{1} \gcl{1} \grcm \gnl
\gvac{1} \gibr \gcl{1} \gnl
\gvac{1} \gcl{3} \gibr \gcl{1} \gnl
\gvac{2} \gcl{1} \gibr \gnl
\gvac{1} \gcl{2} \gcn{1}{1}{1}{2} \gmu \gnl
\gvac{3} \hspace{-0,22cm} \gmu \gnl
\gob{4}{M} \gob{1}{\hspace{-0,24cm}H}
\gend\stackrel{nat.}{\stackrel{(\ref{antipode2})}{\stackrel{unit}{\stackrel{counit}{=}}}}
\gbeg{7}{9}
\got{1}{H} \got{5}{M} \gnl
\gcn{2}{2}{1}{5} \gvac{1} \gcl{2} \gnl \gnl
\gvac{2} \glm \gnl
\gvac{3} \gcl{1} \gnl
\gvac{3} \grcm \gnl
\gvac{3} \gcl{2} \gcn{2}{2}{1}{5} \gnl
\gob{7}{M} \gob{1}{\hspace{-0,64cm}H.}
\gend
$$
Conversely, (\ref{YD-S}) implies:
$$
\gbeg{4}{9}
\got{3}{H} \got{1}{M} \gnl
\gwcm{3} \gcl{1} \gnl
\gcn{1}{2}{1}{3} \gvac{1} \glm \gnl
\gvac{1} \gcn{1}{1}{5}{3} \gnl
\gvac{1} \gibr \gnl
\gcn{1}{1}{3}{1} \gcn{1}{1}{3}{5} \gnl
\grcm \gcn{1}{1}{3}{3} \gnl
\gcl{1} \gwmu{3} \gnl
\gob{1}{M} \gob{3}{H}
\gend\stackrel{nat.}{=}
\gbeg{4}{8}
\got{2}{H} \got{1}{M} \gnl
\gcmu \gcl{1} \gnl
\gcn{1}{2}{1}{3} \glm \gnl
\gvac{2}\grcm \gnl
\gvac{1} \gibr \gcl{1} \gnl
\gvac{1} \gcl{2} \gibr \gnl
\gvac{2} \gmu \gnl
\gob{3}{M} \gob{1}{\hspace{-0,26cm}H}
\gend\stackrel{(\ref{YD-S})}{=}
\gbeg{6}{13}
\got{2}{H} \got{5}{M} \gnl
\gcmu \gvac{2} \gcl{4} \gnl
\gcl{6} \gcn{1}{1}{1}{2} \gnl
\gvac{1} \gcmu \gnl
\gvac{1} \gcl{1} \gcn{1}{1}{1}{2} \gnl
\gvac{1} \gmp{-} \gcmu \grcm \gnl
\gvac{1} \gcl{1} \gcl{1} \gbr \gcl{1} \gnl
\gvac{1} \gcn{1}{1}{1}{3} \glm \gmu \gnl
\gcn{1}{1}{1}{3} \gvac{1} \gibr \gcn{1}{1}{2}{1} \gnl
\gvac{1} \gibr \gibr \gnl
\gvac{1} \gcl{2} \gcn{1}{1}{1}{2} \gmu \gnl
\gvac{3} \hspace{-0,22cm}\gmu \gnl
\gob{4}{M} \gob{1}{\hspace{-0,32cm}H}
\gend\stackrel{coass.}{\stackrel{ass.}{=}}
\gbeg{6}{12}
\gvac{1} \got{3}{H} \got{2}{M} \gnl
\gvac{1} \gwcm{3} \gcn{1}{1}{2}{2} \gnl
\gvac{1} \hspace{-0,34cm} \gcmu \gcmu \grcm \gnl
\gvac{1} \gcl{2} \gmp{-} \gcl{1} \gbr \gcl{1} \gnl
\gvac{1} \gvac{1} \gcn{1}{1}{1}{3} \glm \gmu \gnl
\gvac{1} \gcn{1}{1}{1}{3} \gvac{1} \gibr \gcn{1}{1}{2}{1} \gnl
\gvac{1} \gvac{1} \gibr \gibr \gnl
\gvac{1} \gvac{1} \gcl{4} \gibr \gcl{1} \gnl
\gvac{1} \gvac{2} \gcn{1}{2}{1}{2} \gibr \gnl
\gvac{1} \gvac{3} \gmu \gnl
\gvac{1} \gvac{3} \hspace{-0,34cm} \gmu \gnl
\gvac{1} \gob{4}{M} \gob{1}{\hspace{-0,32cm}H}
\gend\stackrel{nat.}{\stackrel{(\ref{antipode2})}{\stackrel{unit}{\stackrel{counit}{=}}}}
\gbeg{3}{7}
\got{2}{H} \got{1}{M} \gnl
\gcn{1}{1}{2}{2} \gcn{1}{1}{3}{3} \gnl
\gcmu \grcm \gnl
\gcl{1} \gbr \gcl{1} \gnl
\glm \gmu \gnl
\gcn{1}{1}{3}{3} \gcn{1}{1}{4}{4} \gnl
\gob{3}{M} \gob{1}{\hspace{-0,26cm}H.}
\gend
$$

\begin{rem}
If $\Phi_{H, H}$ is symmetric, (\ref{antipode2}) can be considered with $\Phi_{H,H}$ instead of
$\Phi_{H,H}^{-1}$; then one proves that \vspace{-0,4cm}
\begin{center}
\hspace{-1,4cm} %1,56/ ,558 is too much, ,556 is too little
\begin{tabular}{p{6.1cm}p{0.9cm}p{8.7cm}} %9.2  \hspace{-1,5cm}
\begin{equation} \label{YD-mix*}
\gbeg{4}{9}
\got{3}{H} \got{1}{M} \gnl
\gwcm{3} \gcl{1} \gnl
\gcn{1}{2}{1}{3} \gvac{1} \glm \gnl
\gvac{1} \gcn{1}{1}{5}{3} \gnl
\gvac{1} \gbr \gnl
\gcn{1}{1}{3}{1} \gcn{1}{1}{3}{5} \gnl
\grcm \gcn{1}{1}{3}{3} \gnl
\gcl{1} \gwmu{3} \gnl
\gob{1}{M} \gob{3}{H}
\gend=
\gbeg{3}{7}
\got{2}{H} \got{1}{M} \gnl
\gcn{1}{1}{2}{2} \gcn{1}{1}{3}{3} \gnl
\gcmu \grcm \gnl
\gcl{1} \gbr \gcl{1} \gnl
\glm \gmu \gnl
\gcn{1}{1}{3}{3} \gcn{1}{1}{4}{4} \gnl
\gob{3}{M} \gob{1}{\hspace{-0,26cm}H}
\gend
\end{equation} & &
\vspace{-0,96cm}
\begin{eqnarray} \label{YD-S*}
\textnormal{\hspace{-1,2cm} is equivalent to}\hspace{0,2cm}
\gbeg{6}{9}
\got{1}{H} \got{5}{M} \gnl
\gcn{2}{2}{1}{5} \gvac{1} \gcl{2} \gnl \gnl
\gvac{2} \glm \gnl
\gvac{3} \gcl{1} \gnl
\gvac{3} \grcm \gnl
\gvac{3} \gcl{2} \gcn{2}{2}{1}{5} \gnl
\gob{7}{M} \gob{1}{\hspace{-0,64cm}H}
\gend=
\gbeg{7}{10}
\gvac{1} \got{2}{H} \got{3}{M} \gnl
\gvac{1} \gcmu \gvac{1} \gcl{2} \gnl
\gvac{1} \gcl{1} \gcn{1}{1}{1}{2} \gnl
\gvac{1} \gmp{-} \gcmu \grcm \gnl
\gvac{1} \gcl{1} \gcl{1} \gbr \gcl{1} \gnl
\gvac{1} \gcn{1}{1}{1}{3} \glm \gmu \gnl
\gvac{1} \gvac{1} \gbr \gcn{1}{1}{2}{1} \gnl
\gvac{1} \gvac{1} \gcl{2} \gbr \gnl
\gvac{1} \gvac{2} \gmu \gnl
\gvac{1} \gob{3}{M} \gob{1}{\hspace{-0,26cm}H}
\gend
\end{eqnarray}
\end{tabular}
\end{center} \vspace{-0,5cm}
(versions of the relations (\ref{YD-mix}) and (\ref{YD-S})).
%with changed signs of the first power of $\Phi$).
\end{rem}

It is important to note that $H$ itself is a YD-module over itself with suitable structures. For example,
it is a left-right YD-module with the regular action and the adjoint coaction:
$$
\scalebox{0.88}{\gbeg{3}{5}
\got{1}{H} \gnl
\gcl{1} \gnl
\grcm \gnl
\gcl{1} \gcn{1}{1}{1}{3} \gnl
\gob{1}{H} \gob{3}{H}
\gend}=\scalebox{0.88}{
\gbeg{6}{7}
\got{3}{H} \gnl
\gwcm{3} \gnl
\gmp{-} \gwcm{3} \gnl
\gibr \gcn{1}{1}{3}{1} \gnl
\gcl{1} \gibr \gnl
\gcl{1} \gmu \gnl
\gob{1}{H} \gob{2}{H.}
\gend}
$$
For the other versions of a YD-module (see \seref{other-cases}) $H$ can be equipped with similar
structures - regular (co)actions and adjoint (co)actions.
\par\medskip

The last convention before the promissed proof is that throughout, by abuse of notation, we will
write {\em $\Phi_{H, M}$
is symmetric for all $M\in {}_H\YD(\C)^{H^{op}}$}, and similarly for other versions of the
YD-categories, when strictly speaking we should say {\em for all $M\in\C$}. Indeed, via the
forgetful functor $\U: {}_H\YD(\C)^{H^{op}} \to\C$ every $M\in {}_H\YD(\C)^{H^{op}}$ is an object
in $\C$, and every $N\in\C$ can be equipped with trivial $H$-(co)module structures to form a YD-module.
%We will write as indicated for simplicity reasons, abusing notation.

\begin{prop} \prlabel{prva cat}
Assume that %$\Phi_{H, X}$ for $X\in\C$ is symmetric.
%$\Phi$ is symmetric on $H\ot H$ and $H\ot M$
$\Phi_{H, M}$ is symmetric for every left-right YD-module $M$ over $H$ in $\C$. The category
${}_H\YD(\C)^{H^{op}}$ is braided monoidal with braiding and its inverse given by:
$$
\Phi^*_{M, N}=
\gbeg{4}{6}
\gvac{1} \got{1}{M} \got{1}{N} \gnl
\gvac{1} \gbr \gnl
\gcn{1}{1}{3}{2} \gcn{1}{1}{3}{4} \gnl
\gvac{1} \hspace{-0,36cm} \grcm \gcn{1}{1}{1}{1} \gnl
\gcn{1}{1}{3}{3} \gvac{1} \glm \gnl
\gob{3}{N} \gob{1}{M}
\gend
\quad\textnormal{and}\quad
(\Phi^*_{M, N})^{-1}=
\gbeg{4}{7}
\got{2}{N} \got{2}{M} \gnl
\gvac{1} \hspace{-0,34cm} \grcm \gcn{1}{1}{1}{1} \gnl
\gvac{1} \gcl{1} \gmp{+} \gcl{1} \gnl
\gcn{1}{1}{3}{3} \gvac{1} \glm \gnl
\gcn{1}{1}{3}{4} \gcn{1}{1}{5}{4} \gnl
\gvac{2} \hspace{-0,34cm} \gibr \gnl
\gvac{2} \gob{1}{M} \gob{1}{N}
\gend
$$
for $M, N\in {}_H\YD(\C)^{H^{op}}$.
\end{prop}

\begin{proof}
Because of the symmetricity assumption on $\Phi$ we will consider the
YD-compatibility condition from the above Remark. Let $M$ and $N$ be two
left-right YD-modules over $H$. We consider their tensor product as a left
$H$-module and right $H^{op}$-comodule with the (co)diagonal structures.
We now prove that the YD-compatibility of these $H$-structures holds for $M\ot N$:
$$
\gbeg{4}{9}
\got{3}{H} \got{1}{M\ot N} \gnl
\gwcm{3} \gcl{1} \gnl
\gcn{1}{2}{1}{3} \gvac{1} \glm \gnl
\gvac{1} \gcn{1}{1}{5}{3} \gnl
\gvac{1} \gbr \gnl
\gcn{1}{1}{3}{1} \gcn{1}{1}{3}{5} \gnl
\grcm \gcn{1}{1}{3}{3} \gnl
\gcl{1} \gwmu{3} \gnl
\gob{1}{M\ot N} \gob{3}{H}
\gend=
\gbeg{7}{13}
\gvac{1} \got{3}{H} \got{2}{M} \got{1}{\hspace{-0,34cm}N} \gnl
\gvac{1} \gwcm{3} \gcn{1}{2}{2}{2} \gcn{1}{3}{2}{2} \gnl
\gvac{1} \gcl{2} \gvac{1} \hspace{-0,34cm} \gcmu \gnl
\gvac{3} \gcl{1} \gbr \gnl
\gvac{1} \gcn{2}{1}{2}{5} \glm \glm \gnl
\gvac{3} \gbr \gcn{1}{1}{3}{1} \gnl
\gvac{2} \gcn{1}{1}{3}{1} \gvac{1} \gbr \gnl
\gvac{2} \grcm \grcm \gcn{1}{1}{-1}{1} \gnl
\gvac{2} \gcl{4} \gbr \gcl{1} \gcl{2} \gnl
\gvac{3} \gcl{3} \gbr \gnl
\gvac{4} \gmu \gcn{1}{1}{1}{0} \gnl
\gvac{5} \hspace{-0,34cm} \gmu \gnl
\gvac{2} \gob{2}{M} \gob{1}{\hspace{-0,32cm}N} \gob{2}{H}
\gend\stackrel{coass.}{\stackrel{ass.}{=}}
\gbeg{7}{13}
\gvac{1} \got{3}{H} \got{2}{M} \got{1}{\hspace{-0,34cm}N} \gnl
\gvac{1} \gwcm{3} \gcn{1}{2}{2}{2} \gcn{1}{3}{2}{2} \gnl
\gvac{1} \hspace{-0,34cm} \gcmu \gcn{1}{1}{2}{3} \gnl
\gvac{1} \gcl{1} \gcn{1}{1}{1}{3} \gvac{1} \gbr \gnl
\gvac{1} \gcn{2}{1}{1}{5} \glm \glm \gnl
\gvac{3} \gbr \gcn{1}{1}{3}{1} \gnl
\gvac{2} \gcn{1}{1}{3}{1} \gvac{1} \gbr \gnl
\gvac{2} \grcm \grcm \gcn{1}{1}{-1}{1} \gnl
\gvac{2} \gcl{4} \gbr \gcl{1} \gcl{2} \gnl
\gvac{3} \gcl{3} \gbr \gnl
\gvac{4} \gcn{1}{1}{1}{2} \gmu \gnl
\gvac{5} \hspace{-0,22cm} \gmu \gnl
\gvac{2} \gob{2}{M} \gob{1}{\hspace{-0,32cm}N} \gob{2}{H}
\gend\stackrel{nat.}{=}
\gbeg{7}{13}
\got{3}{H} \got{2}{M} \got{1}{\hspace{-0,34cm}N} \gnl
\gwcm{3} \gcn{1}{2}{2}{2} \gcn{1}{3}{2}{2} \gnl
\hspace{-0,34cm} \gcmu \gcn{1}{1}{2}{3} \gnl
\gcl{1} \gcn{1}{1}{1}{3} \gvac{1} \gbr \gnl
\gcn{2}{1}{1}{5} \glm \glm \gnl
\gvac{2} \gbr \gcn{1}{1}{3}{1} \gnl
\gvac{1} \gcn{1}{1}{3}{1} \gvac{1} \gcl{2} \grcm \gnl
\gvac{1} \grcm \gcl{1} \gcl{1} \gcl{1} \gnl
\gvac{1} \gcl{4} \gmu \gcn{1}{1}{1}{0} \gcn{1}{2}{1}{0} \gnl
\gvac{3} \hspace{-0,34cm} \gbr \gnl
\gvac{3} \gcn{1}{2}{1}{1} \gbr \gnl
\gvac{4} \gmu \gnl
\gvac{1} \gob{2}{M} \gob{1}{N} \gob{2}{H}
\gend
$$

$$\stackrel{M (\ref{YD-mix*})}{\stackrel{nat.}{{=}}}
\gbeg{7}{11}
\got{4}{H} \got{1}{M} \got{1}{N} \gnl
\gwcm{4} \gcl{1} \gcn{1}{2}{1}{1} \gnl
\gcn{1}{1}{1}{3} \gvac{2} \gbr \gnl
\gwcm{3} \gcl{1} \glm \gnl
\gcn{1}{1}{1}{1} \gcn{1}{1}{3}{3} \gvac{1} \grcm \gcl{2} \gnl
\gcn{2}{1}{1}{3} \gbr \gcl{1} \gnl
\gvac{1} \glm \gmu \gcn{1}{1}{1}{0} \gnl
\gvac{2} \gcl{3} \gvac{1} \hspace{-0,34cm} \gbr \gnl
\gvac{5} \hspace{-0,4cm} \grcm \gcn{1}{1}{-1}{1} \gnl
\gvac{5} \gcl{1} \gmu \gnl
\gvac{3} \gob{2}{M} \gob{1}{N} \gob{2}{H}
\gend\stackrel{coass.}{\stackrel{ass.}{\stackrel{nat.}{=}}}
\gbeg{6}{14}
\got{2}{H} \got{1}{M} \got{5}{N} \gnl
\gcmu \gcl{1} \gcn{1}{4}{5}{5} \gnl
\gcl{3} \gbr \gnl
\gvac{1} \gcl{1} \gcn{1}{1}{1}{4} \gnl
\gvac{1} \grcm \gcmu \gnl
\glm \gibr \glm \gnl
\gvac{1} \gcl{7} \gcn{1}{2}{1}{3} \gcn{1}{1}{1}{3} \gvac{1} \gcl{1} \gnl
\gvac{4} \gbr \gnl
\gvac{3} \gbr \gcl{3} \gnl
\gcn{1}{1}{7}{5} \gvac{3} \gcl{2} \gnl
\gvac{2} \grcm \gnl
\gvac{2} \gcl{2} \gmu \gcn{1}{1}{1}{0} \gnl
\gvac{4} \hspace{-0,34cm} \gmu  \gnl
\gvac{1} \gob{2}{M} \gob{1}{\hspace{-0,34cm}N} \gob{2}{H}
\gend\stackrel{\Phi_{H,H}}{\stackrel{nat.}{=}}
\gbeg{7}{15}
\got{2}{H} \got{1}{M} \got{5}{N} \gnl
\gcmu \gcl{1} \gcn{1}{4}{5}{5} \gnl
\gcl{3} \gbr \gnl
\gvac{1} \gcl{1} \gcn{1}{1}{1}{4} \gnl
\gvac{1} \grcm \gcmu \gnl
\glm \gcl{5} \gcl{1} \glm \gnl
\gvac{1} \gcl{8} \gvac{1} \gcn{1}{1}{1}{3} \gvac{1} \gcl{1} \gnl
\gvac{4} \gbr \gnl
\gvac{3} \gcn{1}{1}{3}{1} \gvac{1} \gcl{2} \gnl
\gvac{3} \grcm \gnl
\gvac{2} \gbr \gmu \gnl
\gvac{2} \gcl{3} \gcl{1} \gcn{1}{1}{2}{1} \gnl
\gvac{3} \gbr \gnl
\gvac{3} \gmu \gnl
\gvac{1} \gob{1}{M} \gob{1}{N} \gob{2}{H}
\gend
$$

$$
\stackrel{N (\ref{YD-mix*})}{\stackrel{nat.}{{=}}}
\gbeg{8}{12}
\got{2}{H} \got{1}{M} \got{5}{N} \gnl
\gcmu \gcl{1} \gcn{1}{3}{5}{5} \gnl
\gcl{3} \gbr \gnl
\gvac{1} \gcl{1} \gcn{1}{1}{1}{4} \gnl
\gvac{1} \grcm \gcmu \grcm \gnl
\glm \gcl{2} \gcl{1} \gbr \gcl{1} \gnl
\gvac{1} \gcl{5} \gvac{1} \glm \gmu \gnl
\gvac{2} \gcn{1}{1}{1}{3} \gvac{1} \gcl{1} \gcn{1}{1}{2}{1} \gnl
\gvac{3} \gbr \gcl{1} \gnl
\gvac{3} \gcl{2} \gbr \gnl
\gvac{4} \gmu \gnl
\gvac{1} \gob{1}{M} \gvac{1} \gob{1}{N} \gob{2}{H}
\gend\stackrel{\Phi_{H,H}}{\stackrel{nat.}{=}}
\gbeg{8}{11}
\got{2}{H} \got{1}{M} \got{5}{N} \gnl
\gcmu \gcl{1} \gcn{1}{3}{5}{5} \gnl
\gcl{3} \gbr \gnl
\gvac{1} \gcl{1} \gcn{1}{1}{1}{4} \gnl
\gvac{1} \grcm \gcmu \grcm \gnl
\glm \gibr \gbr \gcl{1} \gnl
\gvac{1} \gcl{4} \gcl{1} \gbr \gmu \gnl
\gvac{2} \glm \gcl{1} \gcn{1}{1}{2}{1} \gnl
\gvac{3} \gcl{2} \gbr \gnl
\gvac{4} \gmu \gnl
\gvac{1} \gob{1}{M} \gvac{1} \gob{1}{N} \gob{2}{H}
\gend\stackrel{\Phi_{H,H}}{\stackrel{nat.}{=}}
\gbeg{8}{12}
\got{2}{H} \got{1}{M} \got{5}{N} \gnl
\gcmu \gcl{1} \gcn{1}{3}{5}{5} \gnl
\gcl{3} \gbr \gnl
\gvac{1} \gcl{1} \gcn{1}{1}{1}{4} \gnl
\gvac{1} \grcm \gcmu \grcm \gnl
\glm \gibr \gcl{1} \gcl{2} \gcl{3} \gnl
\gvac{1} \gcl{5} \gcl{3} \gibr \gnl
\gvac{3} \gcl{1} \gbr \gnl
\gvac{3} \gbr \gbr \gnl
\gvac{2} \glm \gmu \gcn{1}{1}{1}{0} \gnl
\gvac{3} \gcl{1} \gvac{1} \hspace{-0,34cm} \gmu \gnl
\gvac{1} \gob{2}{M} \gob{2}{N} \gob{2}{H}
\gend
$$

$$
\stackrel{nat.}{=}
\gbeg{8}{8}
\got{3}{H} \got{2}{M} \gvac{1} \got{1}{\hspace{-0,34cm}N} \gnl
\gwcm{3} \gvac{1} \hspace{-0,34cm} \grcm \grcm \gnl
\gcn{1}{2}{2}{2} \gvac{1} \gcmu \gcl{1} \gbr \gcl{1} \gnl
\gvac{2} \gcl{1} \gbr \gcl{1} \gbr \gnl
\gvac{1} \gcn{1}{1}{0}{1} \gbr \gbr \gcl{1} \gcl{1} \gnl
\gvac{1} \glm \glm \gmu \gcn{1}{1}{1}{0} \gnl
\gvac{2} \gcl{1} \gvac{1} \gcl{1} \gvac{1} \hspace{-0,34cm} \gmu \gnl
\gvac{2} \gob{2}{M} \gvac{1} \gob{1}{\hspace{-0,32cm}N} \gob{2}{H}
\gend\stackrel{coass.}{\stackrel{ass.}{\stackrel{nat.}{=}}}
\gbeg{8}{8}
\gvac{1} \got{3}{H} \got{1}{M} \gvac{1} \got{1}{N} \gnl
\gvac{1} \gwcm{3} \grcm \grcm \gnl
\gvac{1} \hspace{-0,22cm} \gcmu \gvac{1} \hspace{-0,2cm} \gcl{1} \gcl{1} \gbr \gcl{1} \gnl
\gvac{2} \gcn{1}{1}{0}{1} \gcn{1}{1}{0}{1} \gbr \gcl{1} \gbr \gnl
\gvac{2} \gcl{1} \gbr \gbr \gmu \gnl
\gvac{2} \glm \glm \gcn{1}{1}{1}{2} \gcn{1}{1}{2}{2} \gnl
\gvac{3} \gcl{1} \gvac{1} \gcl{1} \gvac{1} \hspace{-0,34cm} \gmu \gnl
\gvac{3} \gob{2}{M} \gvac{1} \gob{1}{\hspace{-0,32cm}N} \gob{2}{H}
\gend=
\gbeg{3}{7}
\got{2}{H} \got{1}{M\ot N} \gnl
\gcn{1}{1}{2}{2} \gcn{1}{1}{3}{3} \gnl
\gcmu \grcm \gnl
\gcl{1} \gbr \gcl{1} \gnl
\glm \gmu \gnl
\gcn{1}{1}{3}{3} \gcn{1}{1}{4}{4} \gnl
\gob{3}{M\ot N} \gob{1}{\hspace{-0,26cm}H.}
\gend
$$
The check that $\Phi^*$ satisfies the braiding axioms we leave to the reader. We prove here
the $H$-linearity of $\Phi^*$: \vspace{-0,26cm}
$$
\gbeg{4}{10}
\got{2}{H} \got{1}{M} \got{1}{N} \gnl
\gcmu \gcl{1} \gcl{1} \gnl
\gcl{1} \gbr \gcl{1} \gnl
\glm \glm \gnl
\gvac{1} \gcl{1} \gcn{1}{1}{3}{1} \gnl
\gvac{1} \gbr \gnl
\gcn{1}{1}{3}{2} \gcn{1}{1}{3}{4} \gnl
\gvac{1} \hspace{-0,34cm} \grcm \gcn{1}{1}{1}{1} \gnl
\gcn{1}{1}{3}{3} \gvac{1} \glm \gnl
\gvac{1}\gob{1}{N}\gvac{1}\gob{1}{M}
\gend\stackrel{\Phi_{H,M}}{\stackrel{nat.}{=}}
\gbeg{5}{9}
\got{2}{H} \got{1}{M} \got{1}{N} \gnl
\gcmu \gcl{1} \gcl{1} \gnl
\gcl{1} \gibr \gcl{1} \gnl
\glm \glm \gnl
\gvac{1} \gcn{1}{1}{1}{3} \gvac{1} \grcm \gnl
\gvac{2} \gbr \gcl{1} \gnl
\gvac{2} \gcl{1} \gbr \gnl
\gvac{1} \gcn{1}{1}{3}{3} \gvac{1} \glm \gnl
\gvac{2}\gob{1}{N}\gvac{1}\gob{1}{M}
\gend\stackrel{N (\ref{YD-S})}{=}
\gbeg{7}{15}
\got{2}{H} \got{1}{M} \gvac{2} \got{1}{N} \gnl
\gcmu \gcl{1} \gvac{2} \gcl{5} \gnl
\gcl{1} \gibr \gnl
\glm \gcn{1}{1}{1}{2} \gnl
\gvac{1} \gcl{5} \gcmu \gnl
\gvac{2} \gcl{1} \gcn{1}{1}{1}{2} \gnl
\gvac{2} \gmp{-} \gcmu \grcm \gnl
\gvac{2} \gcl{1} \gcl{1} \gbr \gcl{1} \gnl
\gvac{2} \gcn{1}{1}{1}{3} \glm \gmu \gnl
\gvac{1} \gcn{1}{1}{1}{3} \gvac{1} \gibr \gcn{1}{1}{2}{1} \gnl
\gvac{2} \gbr \gibr \gnl
\gvac{2} \gcl{3} \gcn{1}{1}{1}{2} \gmu \gnl
\gvac{4} \hspace{-0,2cm} \gbr \gnl
\gvac{4} \glm \gnl
\gvac{2}\gob{2}{N} \gob{3}{M}
\gend\stackrel{nat.}{\stackrel{coass.}{=}}
\gbeg{8}{14}
\gvac{1} \got{3}{H} \gvac{1} \got{1}{M} \got{1}{N} \gnl
\gvac{1} \gwcm{3} \gvac{1} \gbr \gnl
\gvac{1} \hspace{-0,22cm} \gcmu \gcn{1}{1}{2}{2} \gvac{1} \gcn{1}{1}{2}{1} \gcn{1}{1}{2}{3} \gnl
\gvac{1} \gcl{1} \gmp{-} \gcmu \grcm \gcl{6} \gnl
\gvac{1} \gbr \gcl{1} \gbr \gcl{1} \gnl
\gvac{1} \gcl{1} \gcn{1}{1}{1}{3} \glm \gmu \gnl
\gvac{1} \gcn{1}{1}{1}{3} \gvac{1} \gbr \gcn{1}{1}{2}{1} \gnl
\gvac{2} \gibr \gbr \gnl
\gvac{2} \gcl{5} \gibr \gcn{1}{1}{1}{3} \gnl
\gvac{3} \gmu \gvac{1} \glm \gnl
\gvac{3} \gcn{1}{2}{2}{7} \gvac{3} \gcl{2} \gnl
\gnl
\gvac{6} \glm \gnl
\gvac{2} \gob{1}{N} \gvac{4} \gob{1}{M}
\gend
$$

$$
\stackrel{mod.}{\stackrel{ass.}{=}}
\gbeg{8}{15}
\gvac{1} \got{3}{H} \gvac{1} \got{1}{M} \got{1}{N} \gnl
\gvac{1} \gwcm{3} \gvac{1} \gbr \gnl
\gvac{1} \hspace{-0,22cm} \gcmu \gcn{1}{1}{2}{2} \gvac{1} \gcn{1}{1}{2}{1} \gcn{1}{1}{2}{3} \gnl
\gvac{1} \gcl{1} \gmp{-} \gcmu \grcm \gcl{10} \gnl
\gvac{1} \gbr \gcl{1} \gbr \gcl{1} \gnl
\gvac{1} \gcl{1} \gcn{1}{1}{1}{3} \glm \gmu \gnl
\gvac{1} \gcn{1}{1}{1}{3} \gvac{1} \gbr \gcn{1}{1}{2}{1} \gnl
\gvac{2} \gibr \gbr \gnl
\gvac{2} \gcl{6} \gibr \gcl{1} \gnl
\gvac{3} \gcn{1}{1}{1}{2} \gmu \gnl
\gvac{4} \hspace{-0,22cm} \gmu \gnl
\gvac{4} \gcn{2}{2}{2}{6} \gnl
\gnl
\gvac{7} \hspace{-0,34cm} \glm \gnl
\gvac{3} \gob{1}{N} \gvac{4} \gob{1}{M}
\gend\stackrel{\Phi_{H,N}}{\stackrel{\Phi_{H,H}}{\stackrel{(\ref{antipode2})}{=}}}
\gbeg{5}{9}
\got{2}{H} \got{1}{M} \got{1}{N} \gnl
\gcn{1}{1}{2}{2} \gvac{1} \gbr \gnl
\gcmu \grcm \gcn{1}{1}{-1}{1} \gnl
\gcl{1} \gbr \gcl{1} \gcl{3} \gnl
\glm \gmu \gnl
\gvac{1} \gcl{3} \gcn{1}{1}{2}{3} \gnl
\gvac{3} \glm \gnl
\gvac{4} \gcl{1} \gnl
\gvac{1} \gob{1}{N} \gvac{2} \gob{1}{M}
\gend\stackrel{mod.}{=}
\gbeg{5}{11}
\got{1}{H}\gvac{1} \got{1}{M} \got{1}{N} \gnl
\gcl{4} \gvac{1} \gcl{1} \gcl{1} \gnl
\gvac{2} \gbr \gnl
\gvac{1} \gcn{1}{1}{3}{2} \gcn{1}{1}{3}{4} \gnl
\gvac{2} \hspace{-0,34cm} \grcm \gcn{1}{1}{1}{1} \gnl
\hspace{-0,14cm} \gcmu \gcl{1} \glm \gnl
\gcl{1} \gbr \gvac{1} \gcl{2} \gnl
\gcl{1} \gcl{1} \gcn{1}{1}{1}{3} \gnl
\glm \gvac{1} \glm \gnl
\gvac{1} \gcl{1} \gvac{2} \gcl{1} \gnl
\gvac{1} \gob{1}{N} \gvac{2} \gob{1}{M.}
\gend
$$
The $H^{op}$-colinearity of $\Phi^*$ follows from:
$$
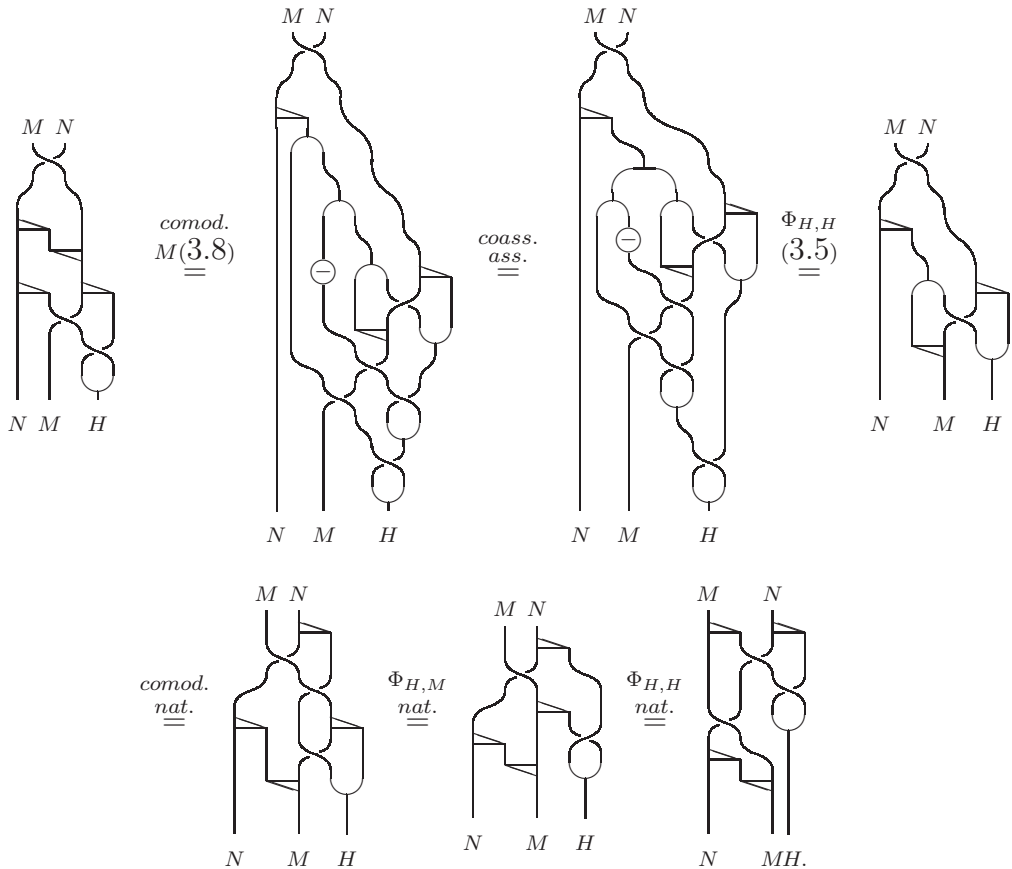
\gbeg{4}{10}
\got{1}{M} \got{1}{N} \gnl
\gbr \gnl
\gcn{1}{1}{1}{0} \gcn{1}{1}{1}{2} \gnl
\hspace{-0,34cm} \grcm \gcn{1}{1}{1}{1} \gnl
\gcn{1}{1}{1}{1} \glm \gnl
\grcm \grcm \gnl
\gcl{3} \gbr \gcl{1} \gnl
\gvac{1} \gcl{2} \gbr \gnl
\gvac{2} \gmu \gnl
\gob{1}{N} \gob{1}{M} \gob{2}{H}
\gend\stackrel{comod.}{\stackrel{M (\ref{YD-S*})}{=}}
\gbeg{7}{17}
\gvac{1} \got{1}{M} \got{1}{N} \gnl
\gvac{1} \gbr \gnl
\gvac{1} \gcn{1}{1}{1}{0} \gcn{1}{1}{1}{2} \gnl
\gvac{1} \hspace{-0,36cm} \grcm \gcn{2}{2}{1}{3} \gnl
\gvac{2} \hspace{-0,34cm} \gcn{1}{12}{0}{0} \hspace{-0,44cm} \gcmu \gnl%\gcn{1}{1}{1}{2} \gnl
\gvac{3} \gcl{6} \gcn{1}{1}{1}{2} \gcn{2}{2}{2}{4} \gnl
\gvac{4} \gcmu \gnl
\gvac{4} \gcl{1} \gcn{1}{1}{1}{2} \gcn{1}{1}{2}{3} \gnl
\gvac{4} \gmp{-} \gcmu \grcm \gnl
\gvac{4} \gcl{1} \gcl{1} \gibr \gcl{1} \gnl
\gvac{4} \gcn{1}{1}{1}{3} \glm \gmu \gnl
\gvac{3} \gcn{1}{1}{1}{3} \gvac{1} \gbr \gcn{1}{1}{2}{1} \gnl
\gvac{4} \gbr \gbr \gnl
\gvac{4} \gcl{3} \gcn{1}{1}{1}{2} \gmu \gnl
\gvac{6} \hspace{-0,2cm} \gbr \gnl
\gvac{6} \gmu \gnl
\gvac{3}\gob{1}{N} \gob{2}{M} \gob{2}{H}
\gend\stackrel{coass.}{\stackrel{ass.}{=}}
\gbeg{7}{17}
\gvac{1} \got{1}{M} \got{1}{N} \gnl
\gvac{1} \gbr \gnl
\gvac{1} \gcn{1}{1}{1}{0} \gcn{1}{1}{1}{2} \gnl
\gvac{1} \hspace{-0,34cm} \grcm \gcn{2}{2}{1}{5} \gnl
\gvac{1} \gcl{12} \gcn{1}{1}{1}{3} \gnl
\gvac{2} \gwcm{3} \gcn{1}{1}{1}{2} \gnl
\gvac{2} \hspace{-0,34cm} \gcmu \gcmu \grcm \gnl
\gvac{2} \gcl{1} \gmp{-} \gcl{1} \gibr \gcl{1} \gnl
\gvac{2} \gcl{1} \gcn{1}{1}{1}{3} \glm \gmu \gnl
\gvac{2} \gcn{1}{1}{1}{3} \gvac{1} \gbr \gcn{1}{1}{2}{1} \gnl
\gvac{3} \gbr \gcl{1} \gcl{4} \gnl
\gvac{3} \gcl{5} \gbr \gnl
\gvac{4} \gmu \gnl
\gvac{4} \gcn{1}{1}{2}{3} \gnl
\gvac{5} \gbr \gnl
\gvac{5} \gmu \gnl
\gvac{1} \gob{2}{N} \gob{1}{M} \gvac{1} \gob{2}{H}
\gend\stackrel{\Phi_{H,H}}{\stackrel{(\ref{antipode2})}{=}}
\gbeg{4}{10}
\gvac{1} \got{1}{M} \got{1}{N} \gnl
\gvac{1} \gbr \gnl
\gvac{1} \gcn{1}{1}{1}{0} \gcn{1}{1}{1}{2} \gnl
\gvac{1} \hspace{-0,34cm} \grcm \gcn{1}{2}{1}{3} \gnl
\gvac{1} \gcl{5} \gcn{1}{1}{1}{2} \gnl
\gvac{2} \gcmu \grcm \gnl
\gvac{2} \gcl{1} \gbr \gcl{1} \gnl
\gvac{2} \glm \gmu \gnl
\gvac{3} \gcl{1} \gcn{1}{1}{2}{2} \gnl
\gvac{1} \gob{1}{N} \gvac{1} \gob{1}{M} \gob{2}{H}
\gend
$$

$$
\stackrel{comod.}{\stackrel{nat.}{=}}
\gbeg{5}{9}
\gvac{1} \got{1}{M} \got{1}{N} \gnl
\gvac{1} \gcl{1} \grcm \gnl
\gvac{1} \gbr \gcl{1} \gnl
\gcn{1}{1}{3}{1} \gvac{1} \gbr \gnl
\grcm \gcl{1} \grcm \gnl
\gcl{3} \gcl{1} \gbr \gcl{1} \gnl
\gvac{1} \glm \gmu \gnl
\gvac{2} \gcl{1} \gcn{1}{1}{2}{2} \gnl
\gob{1}{N} \gvac{1} \gob{1}{M} \gob{2}{H}
\gend\stackrel{\Phi_{H,M}}{\stackrel{nat.}{=}}
\gbeg{5}{8}
\gvac{1} \got{1}{M} \got{1}{N} \gnl
\gvac{1} \gcl{1} \grcm \gnl
\gvac{1} \gbr \gcn{1}{1}{1}{3} \gnl
\gcn{1}{1}{3}{1} \gvac{1} \grcm \gcl{1} \gnl
\grcm \gcl{1} \gibr \gnl
\gcl{2} \glm \gmu \gnl
\gvac{2} \gcl{1} \gcn{1}{1}{2}{2} \gnl
\gob{1}{N} \gvac{1} \gob{1}{M} \gob{2}{H}
\gend\stackrel{\Phi_{H,H}}{\stackrel{nat.}{=}}
\gbeg{5}{9}
\got{1}{M} \gvac{1} \got{1}{N} \gnl
\grcm \grcm \gnl
\gcl{2} \gbr \gcl{1} \gnl
\gvac{1} \gcl{1} \gbr \gnl
\gbr \gmu \gnl
\grcm \gcn{1}{1}{-1}{1} \gcn{1}{3}{0}{0} \gnl
\gcl{2} \glm \gnl
\gvac{2} \gcl{1} \gnl
\gob{1}{N} \gvac{1} \gob{1}{M} \gob{1}{\hspace{-0,26cm}H.}
\gend
$$
The proof that the inverse of $\Phi^*$ is given as in the announcement of the claim is straightforward.
\qed\end{proof}

\begin{rem} \rmlabel{Phi-mix-changed}
Note that because of the assumption that $\Phi_{H,M}$ is symmetric, instead of
$\Phi^{1+}:=\Phi^*$ in \prref{prva cat} we can also consider the braiding:
$$
\Phi^{1-}_{M, N}=
\gbeg{4}{6}
\gvac{1} \got{1}{M} \got{1}{N} \gnl
\gvac{1} \gibr \gnl
\gcn{1}{1}{3}{2} \gcn{1}{1}{3}{4} \gnl
\gvac{1} \hspace{-0,34cm} \grcm \gcn{1}{1}{1}{1} \gnl
\gcn{1}{1}{3}{3} \gvac{1} \glm \gnl
\gob{3}{N} \gob{1}{M.}
\gend
$$
\end{rem}

\begin{rem} \rmlabel{lr-Hcop-cat}
With the same conditions as in \prref{prva cat} one has that the category ${}_{H^{cop}}\YD(\C)^H$
is braided monoidal with braiding and its inverse given by:
$$
\Phi^{2+}_{M, N}=
\gbeg{4}{6}
\got{1}{M} \got{3}{N} \gnl
\grcm \gcn{1}{1}{1}{1} \gnl
\gcn{1}{1}{1}{1} \glm \gnl
\gcn{1}{1}{1}{2} \gcn{1}{1}{3}{2} \gnl
\gvac{1} \hspace{-0,34cm} \gbr \gnl
\gob{3}{N} \gob{1}{\hspace{-0,6cm}M}
\gend\quad\textnormal{and}\quad
(\Phi^{2+}_{M, N})^{-1}=
\gbeg{4}{7}
\got{2}{N} \got{1}{\hspace{-0,3cm}M} \gnl
\gvac{1} \hspace{-0,34cm} \gibr \gnl
\gcn{1}{1}{3}{2} \gcn{1}{1}{3}{4} \gnl
\gvac{1} \hspace{-0,34cm} \grcm \gcn{1}{1}{1}{1} \gnl
\gvac{1} \gcl{1} \gmp{+} \gcl{1} \gnl
\gcn{1}{1}{3}{3} \gvac{1} \glm \gnl
\gvac{1} \gob{1}{M} \gob{3}{N}
\gend
$$
for $M, N\in {}_{H^{cop}}\YD(\C)^H$. %As in \rmref{Phi-mix-changed}, instead of $\Phi'$ we
%can also consider the braiding:$$\crta{\Phi'}_{M, N}=\gbeg{4}{6}\got{1}{M} \got{3}{N} \gnl
%\grcm \gcn{1}{1}{1}{1} \gnl\gcn{1}{1}{1}{1} \glm \gnl\gcn{1}{1}{1}{2} \gcn{1}{1}{3}{2} \gnl
%\gvac{1} \hspace{-0,34cm} \gbr \gnl\gob{3}{N} \gob{1}{\hspace{-0,6cm}M.} \gend$$
Analogously as in \rmref{Phi-mix-changed}, the braiding $\Phi^{2+}$ can be taken in the form
$\Phi^{2-}$. Note that ${}_{H^{cop}}\YD(\C)^H$ is not braided by $\Phi^{1\pm}$, since $\Phi^{1\pm}$
is not left $H^{cop}$-linear even if $\C\w=\w Vec$, the category of vector spaces. Thus the identity
functor $\Id\w: {}_H\YD(\C)^{H^{op}} \to {}_{H^{cop}}\YD(\C)^H$ is not an isomorphism of braided monoidal
categories although it is monoidal.
\end{rem}

\section{Bicrossproducts in braided monoidal categories} \selabel{bicruz}

\setcounter{equation}{0}

Bicrossproducts in braided monoidal categories (also called cross product bialgebras)
were treated in \cite{ZC, BD}.
%\textcolor{rojo}{compare general case of the Chinese and the Hopf data of Bespalov-Drabant!}.
We recall here bicrossproducts with trivial coactions. Let $B$ and $H$ be bialgebras
in $\C$, where $B$ is a left $H$-module coalgebra and $H$ is a right $B$-module coalgebra.
Assume further that the following conditions are fulfilled:
$$
\gbeg{5}{8}
\got{2}{H} \got{2}{B} \got{1}{B} \gnl
\gcmu \gcmu \gcl{4} \gnl
\gcl{1} \gbr \gcl{1} \gnl
\glm \grm \gnl
\gvac{1} \gcn{2}{2}{1}{3} \gcn{1}{1}{-1}{1} \gnl
\gvac{3} \glm \gnl
\gvac{2} \gwmu{3} \gnl
\gvac{3} \gob{1}{B}
\gend=
\gbeg{4}{7}
\got{1}{H} \got{3}{B} \got{1}{B} \gnl
\gcn{2}{2}{1}{3} \gwmu{3} \gnl
\gvac{3} \gcl{2} \gnl
\gcn{1}{1}{3}{5} \gnl
\gvac{2} \glm \gnl
\gvac{3} \gcl{1} \gnl
\gvac{3} \gob{1}{B}
\gend
\quad ; \quad\quad
\gbeg{5}{8}
\got{1}{H} \got{2}{H} \got{2}{B} \gnl
\gcl{1} \gcmu \gcmu \gnl
\gcl{1} \gcl{1} \gbr \gcl{1} \gnl
\gcl{1} \glm \grm \gnl
\gcl{1} \gcn{1}{1}{3}{1} \gcn{2}{2}{3}{1} \gnl
\grm \gnl
\gwmu{3} \gnl
\gob{3}{H}
\gend=
\gbeg{5}{7}
\got{1}{H} \got{3}{H} \got{1}{B} \gnl
\gwmu{3} \gcn{2}{2}{3}{1} \gnl
\gvac{1} \gcl{2} \gnl
\gvac{1} \gcn{1}{1}{5}{3} \gnl
\gvac{1} \grm \gnl
\gvac{1} \gcl{1} \gnl
\gvac{1} \gob{1}{B}
\gend
$$

$$
\gbeg{4}{6}
\got{2}{H} \got{2}{B} \gnl
\gcmu \gcmu \gnl
\gcl{1} \gbr \gcl{1} \gnl
\glm \grm \gnl
\gvac{1} \gbr \gnl
\gvac{1} \gob{1}{H} \gob{1}{B}
\gend=
\gbeg{4}{6}
\got{2}{H} \got{2}{B} \gnl
\gcmu \gcmu \gnl
\gcl{1} \gbr \gcl{1} \gnl
\grm \glm \gnl
\gcl{1} \gvac{2} \gcl{1} \gnl
\gob{1}{H} \gvac{2} \gob{1}{B}
\gend
\quad \textnormal{and} \quad\quad
\gbeg{2}{4}
\got{1}{H} \got{1}{}  \gnl
\gcl{1} \gu{1} \gnl
\glm \gnl
\gvac{1} \gob{1}{B}
\gend=
\gbeg{1}{4}
\got{1}{H} \gnl
\gcu{1}  \gnl
\gu{1}  \gnl
\gob{1}{B}
\gend \quad;
\quad
\gbeg{2}{4}
\got{1}{} \got{1}{B}  \gnl
\gu{1} \gcl{1} \gnl
\grm \gnl
\gob{1}{H}
\gend=
\gbeg{1}{4}
\got{1}{B} \gnl
\gcu{1}  \gnl
\gu{1}  \gnl
\gob{1}{H}
\gend
$$
Bialgebras $B$ and $H$ described above are called a {\em matched pair of bialgebras in $\C$}.
We define $B\bowtie H$ as the tensor product $B\ot H$ endowed with the codiagonal comultiplication,
usual unit $\eta$ and counit $\Epsilon$ (that is, $\eta_B\ot\eta_H$ and $\Epsilon_B\ot\Epsilon_H$
respectively), and associative multiplication given by:
$$\nabla_{B\bowtie H}=
\gbeg{6}{6}
\got{1}{B} \got{2}{H} \got{2}{B} \got{1}{H} \gnl
\gcl{1} \gcmu \gcmu \gcl{3} \gnl
\gcl{1} \gcl{1} \gbr \gcl{1} \gnl
\gcl{1} \glm \grm \gnl
\gwmu{3} \gwmu{3} \gnl
\gvac{1} \gob{1}{B} \gvac{2} \gob{1}{H.}
\gend
$$
%\quad ; \quad
%\Delta_{B\#H}=
%\gbeg{4}{4}
%\got{2}{B} \got{2}{H} \gnl
%\gcmu \gcmu \gnl
%\gcl{1} \gbr \gcl{1} \gnl
%\gob{1}{B} \gob{1}{H} \gob{1}{B} \gob{1}{H}
%\gend
%\quad ; \quad\quad
%\eta_{B\#H}=\eta_B\ot\eta_H
%\quad ; \quad
%\Epsilon_{B\#H}=\Epsilon_B\ot\Epsilon_H
In \cite[Theorem 1.4]{ZC} it is proved that $B\bowtie H$ is a bialgebra.
%if and only if
%$B$ and $H$ are a matched pair of bialgebras \textcolor{rojo}{where the requirement on the
%module coalgebra structures is weakened by those that the respective actions be compatible
%with the respective comultiplications}.
Moreover, if both $B$ and $H$ are Hopf algebras,
by \cite[Theorem 1.5]{ZC} we know that so is $B\bowtie H$ with the antipode given by:
$$S_{B\bowtie H}=
\gbeg{4}{10}
\gvac{1} \got{1}{B} \got{1}{H} \gnl
\gvac{1} \gbr \gnl
\gcn{1}{1}{3}{2} \gcn{1}{1}{3}{4} \gnl
\gcmu \gcmu \gnl
\gmp{S} \gmp{S} \gmp{+} \gmp{+} \gnl
\gcl{1} \gibr \gcl{1} \gnl
\grm \glm \gnl
\gcn{1}{1}{1}{3} \gcn{1}{1}{5}{3} \gnl
\gvac{1} \gbr \gnl
\gvac{1} \gob{1}{B} \gob{1}{H}
\gend=
\gbeg{4}{9}
\gvac{1} \got{1}{B} \got{1}{H} \gnl
\gvac{1} \gbr \gnl
\gvac{1} \gmp{S} \gmp{+} \gnl
\gvac{1} \gcn{1}{1}{1}{0} \gcn{1}{1}{1}{2} \gnl
\gcmu \gcmu \gnl
\gcl{1} \gbr \gcl{1} \gnl
\glm \grm \gnl
\gvac{1} \gcl{1} \gcl{1} \gnl
\gvac{1} \gob{1}{B} \gob{1}{H.}
\gend
$$
From here it follows:
\begin{equation}\eqlabel{1S}
S_{B\bowtie H}(\eta_B\ot H)=\eta_B\ot S_H.
\end{equation}

\vspace{0,2cm}

For a module $M$ over $B\bowtie H$ in $\C$ we will consider:
\begin{equation}\eqlabel{BHmod-conv}
\gbeg{3}{5}
\got{1}{B\bowtie H} \got{3}{M} \gnl
\gcn{1}{1}{1}{3} \gvac{1} \gcl{1} \gnl
\gvac{1} \glm \gnl
\gvac{2} \gcl{1} \gnl
\gvac{1} \gob{3}{M}
\gend=
\gbeg{3}{5}
\got{1}{B} \got{1}{H} \got{1}{M} \gnl
\gcl{1} \glm \gnl
\gcn{1}{1}{1}{3} \gvac{1} \gcl{1} \gnl
\gvac{1} \glm \gnl
\gvac{1} \gob{3}{M.}
\gend
\end{equation}

\vspace{0,2cm}

\begin{lma} \lelabel{matched pair mod}
Let $B$ and $H$ be a matched pair of bialgebras. An object $M$ is a module over $B\bowtie H$ in $\C$
if and only if it is an $H$- and a $B$-module satisfying the compatibility condition:
\begin{equation}\eqlabel{B tie H -mod}
\gbeg{5}{6}
\got{1}{H} \got{1}{B} \got{3}{M} \gnl
\gcl{1} \gcn{1}{1}{1}{3} \gvac{1} \gcl{1} \gnl
\gcn{1}{2}{1}{5} \gvac{1} \glm \gnl
\gvac{3} \gcl{1} \gnl
\gvac{2} \glm \gnl
\gvac{2} \gob{3}{M}
\gend=
\gbeg{5}{9}
\got{2}{H} \got{2}{B} \got{1}{M} \gnl
\gcmu \gcmu \gcl{7} \gnl
\gcl{1} \gbr \gcl{1} \gnl
\glm \grm \gnl
\gvac{1} \gcl{1} \gcn{1}{1}{1}{3} \gnl
\gvac{1} \gcn{1}{2}{1}{5} \gvac{1} \glm \gnl
\gvac{4} \gcl{1} \gnl
\gvac{3} \glm \gnl
\gvac{3} \gob{3}{M.}
\gend
\end{equation}
\end{lma}

\begin{proof}
An object $M$ is a module over $B\bowtie H$ if and only if:
$$
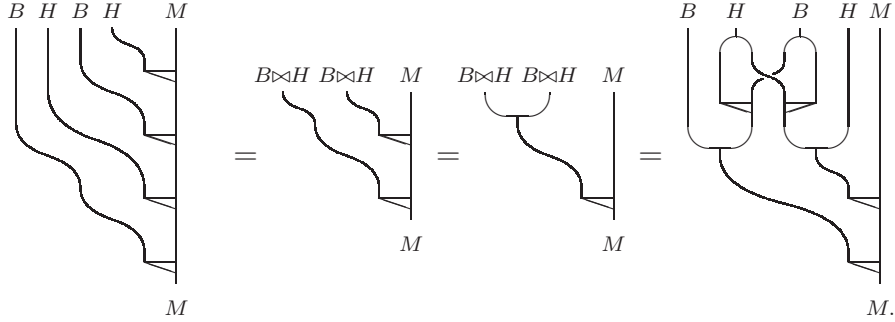
\gbeg{7}{10}
\got{1}{B} \got{1}{H} \got{1}{B} \got{1}{H} \got{3}{M} \gnl
\gcl{3} \gcl{2} \gcl{1} \gcn{1}{1}{1}{3} \gvac{1} \gcl{1} \gnl
\gvac{2} \gcn{1}{2}{1}{5} \gvac{1} \glm \gnl
\gvac{1} \gcn{2}{3}{1}{7} \gvac{2} \gcl{1} \gnl
\gcn{2}{2}{1}{5} \gvac{2} \glm \gnl
\gvac{5} \gcl{1} \gnl
\gvac{2} \gcn{2}{2}{1}{5} \glm \gnl
\gvac{5} \gcl{1} \gnl
\gvac{4} \glm \gnl
\gvac{4} \gob{3}{M}
\gend=
\gbeg{5}{6}
\got{1}{B\bowtie H} \got{3}{B\bowtie H} \got{1}{M} \gnl
\gcn{1}{1}{1}{3} \gvac{1} \gcn{1}{1}{1}{3} \gvac{1} \gcl{1} \gnl
\gvac{1} \gcn{1}{2}{1}{5} \gvac{1} \glm \gnl
\gvac{4} \gcl{1} \gnl
\gvac{3} \glm \gnl
\gvac{3} \gob{3}{M}
\gend=
\gbeg{5}{6}
\got{1}{B\bowtie H} \got{3}{B\bowtie H} \got{1}{M} \gnl
\gwmu{3} \gvac{1} \gcl{3} \gnl
\gvac{1} \gcn{1}{2}{1}{5} \gnl \gnl
\gvac{3} \glm \gnl
\gvac{3} \gob{3}{M}
\gend=
\gbeg{7}{10}
\got{1}{B} \got{2}{H} \got{2}{B} \got{1}{H} \got{1}{M} \gnl
\gcl{1} \gcmu \gcmu \gcl{3} \gcl{7} \gnl
\gcl{1} \gcl{1} \gbr \gcl{1} \gnl
\gcl{1} \glm \grm \gnl
\gwmu{3} \gwmu{3} \gnl
\gvac{1} \gcn{2}{3}{1}{9} \gvac{1} \gcn{1}{1}{1}{3} \gnl
\gvac{5} \glm \gnl \gnl
\gvac{5} \glm \gnl
\gvac{5} \gob{3}{M.}
\gend
$$
Applying this to $\eta_B\ot H\ot B\ot\eta_H$, we obtain \equref{B tie H -mod}. For the converse
observe that the above equality follows from \equref{B tie H -mod} and the $H$- and $B$-module
properties of $M$.
\qed\end{proof}

We now want to consider a particular case of a bicrossproduct - the Drinfel'd double of $H$.
A tedious direct check, which we omit here for practical reasons, shows:

\begin{prop} \prlabel{Drinfelddouble}
Let $H\in\C$ be a finite Hopf algebra with a bijective antipode and the braiding such that
$\Phi_{H,H}$ and $\Phi_{H,H^*}$ are symmetric. Then $B\bowtie H$ is a bicrossproduct with
$B=(H^{op})^*$ and the actions:
$$
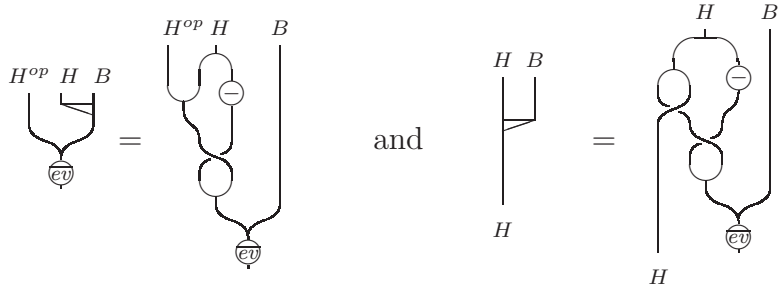
\gbeg{3}{5}
\got{1}{H^{op}} \got{2}{\hspace{-0,2cm}H} \got{1}{\hspace{-0,6cm}B} \gnl
\gcl{1} \glm \gnl
\gcn{1}{1}{1}{3}\gcn{1}{1}{3}{1} \gnl
\gvac{1}  \gmp{\crta\ev} \gnl
\gob{1}{}
\gend=
\gbeg{3}{8}
\got{2}{H^{op}} \got{1}{\hspace{-0,3cm}H} \got{2}{B} \gnl
\gcl{1} \gcmu \gcn{1}{5}{2}{2} \gnl
\gmu \gmp{-} \gnl
\gcn{1}{1}{2}{3} \gvac{1} \gcl{1} \gnl
\gvac{1} \gbr \gnl
\gvac{1} \gmu \gnl
\gvac{1} \gcn{1}{1}{2}{4} \gcn{1}{1}{4}{2} \gnl
\gvac{3} \hspace{-0,34cm} \gmp{\crta\ev} \gnl
\gob{1}{}
\gend
\qquad\qquad\textnormal{and}\qquad
\gbeg{3}{6}
\got{1}{H} \got{1}{B} \gnl
\gcl{1} \gcl{1} \gnl
\grm \gnl
\gcl{2} \gnl
\gob{1}{H}
\gend=
\gbeg{3}{9}
\got{2}{} \got{1}{H} \gvac{1} \got{1}{B} \gnl
\gvac{1} \gwcm{3}  \gcl{5} \gnl %{1}{2}{2}
\gvac{1} \hspace{-0,34cm} \gcmu \gvac{1} \hspace{-0,2cm} \gmp{-} \gnl
\gvac{2} \hspace{-0,36cm} \gibr \gcn{1}{1}{2}{1} \gnl
\gvac{2} \gcl{4} \gbr \gnl
\gvac{3} \gmu \gnl
\gvac{3} \gcn{1}{1}{2}{4} \gcn{1}{1}{4}{2} \gnl
\gvac{5} \hspace{-0,34cm} \gmp{\crta\ev} \gnl
\gvac{2} \gob{2}{H}
\gend
$$
\end{prop}

The bialgebra $B\bowtie H$ is called {\em the Drinfel'd double of $H$} and is denoted by $D(H)$.
Throughout, apart from assuming that our Hopf algebras have a bijective antipode, when we deal with
$D(H)$ we will also assume that $H$ is finite.
As we mentioned before, the antipode of a finite Hopf algebra is bijective if e.g. $\C$ has equalizers.
%\cite[Theorem 4.1]{Tak2}, \textcolor{rojo}{as we will assume in the work that follows}.

Note that $B$ is a
bialgebra since $\Phi_{H,H}$ is symmetric (we commented this before \inref{inner-hom}).
%on page \pageref{op,cop-bialg}).
We only point out that in the proof of the above claim one uses the identity that we next present.
Bearing in mind that $B=(H^{op})^*$, we have:
\begin{equation}\eqlabel{B-codiag}
\scalebox{0.84}[0.84]{ \gbeg{4}{4}
\got{2}{B} \got{1}{H}\got{1}{H} \gnl \gcmu \gcl{1} \gcl{2} \gnl \gcl{1} \gbr \gnl \gev \gev \gnl
\gvac{2} \gob{1}{} \gend} \stackrel{(\ref{codiagH*})}{=} \scalebox{0.84}[0.84]{ \gbeg{4}{4}
\got{1}{\hspace{-0,34cm}B} \got{1}{H}\got{1}{H} \gnl
\gcn{1}{2}{0}{0} \gbr \gnl \gvac{1} \gmu \gnl \hspace{-0,2cm} \gwev{3} \gnl
\gvac{2} \gob{2}{} \gend}
\end{equation}
Composing this from above (in the braided diagram orientation) with $\Phi_{H\ot H, B}$ and
applying $\crta\ev=\ev\Phi$ due to \equref{crta ev}, by naturality we obtain:
\begin{equation}\eqlabel{codiag B}
\scalebox{0.84}[0.84]{
\gbeg{4}{6}
\got{1}{H} \got{1}{H} \got{2}{B} \gnl
\gbr \gcn{1}{2}{2}{2} \gnl
\gmu \gnl
\gcn{1}{1}{2}{4} \gcn{1}{1}{4}{2} \gnl
\gvac{2} \hspace{-0,34cm} \gmp{\crta\ev} \gnl
\gvac{1} \gob{2}{} \gend}
=\scalebox{0.84}[0.84]{ \gbeg{4}{7}
\gvac{1} \got{1}{H} \got{1}{H} \got{1}{B} \gnl
\gvac{1} \gcl{1} \gbr \gnl
\gvac{1} \gbr \gcl{2} \gnl
\gvac{1} \gcn{1}{1}{1}{0} \gcl{1} \gnl
\gcmu \gcl{1} \gcl{2} \gnl \gcl{1} \gbr \gnl \gev \gev \gnl
\gvac{2} \gob{1}{} \gend} \stackrel{\Phi_{H,H^*}}{\stackrel{nat.}{=}}
\scalebox{0.84}[0.84]{
\gbeg{5}{6}
\got{1}{H} \got{1}{H} \got{2}{B} \gnl
\gcl{1} \gcl{1} \gcmu \gnl
\gcl{1} \gbr \gcl{1} \gnl
\gbr \gbr \gnl
\gev \gev \gnl
\gvac{2} \gob{2}{} \gend} =
\scalebox{0.84}[0.84]{
\gbeg{4}{5}
\got{1}{H} \got{1}{H} \got{2}{B} \gnl
\gcl{1} \gcl{1} \gcmu \gnl
\gcl{1} \gbr \gcl{1} \gnl
%\gmp{\crta\ev} \gmp{\crta\ev} \gnl
\glmpt \gnot{\hspace{-0,32cm} \crta \ev} \grmpt \glmpt \gnot{\hspace{-0,32cm} \crta \ev} \grmpt \gnl
\gvac{2} \gob{2}{} \gend} \vspace{-0,12cm}
\end{equation}

As a matter of fact the two symmetricity conditions for $\Phi_{H,H}$ and $\Phi_{H,H^*}$
in \prref{Drinfelddouble} are equivalent (in the next Lemma we add the last condition):
%As noted in  we have the following.

\begin{lma} \cite[Lemma 1.1]{Z} \lelabel{Zhang lemma}
The following conditions are equivalent:
\begin{enumerate}
\item $\Phi_{H,H}, \Phi_{H,H^*}$ and $\Phi_{H^*,H^*}$ are symmetric;
\item $\Phi_{H,H}$ is symmetric;
\item $\Phi_{H^*,H^*}$ is symmetric;
\item $(H\ot\ev)(\Phi_{H^*,H}\ot H)=(\ev\ot H)(H^*\ot\Phi_{H,H})$;
\item $(H^*\ot\ev)(\Phi_{H^*,H^*}\ot H)=(\ev\ot H^*)(H^*\ot\Phi_{H^*,H})$;
\item the conditions 4) and 5) hold true;
\item $\Phi_{H,H^*}$ is symmetric.
\end{enumerate}
\end{lma}

One proves similarly:

\begin{lma} \lelabel{H - H* transp.}
Let $M\in\C$ be any object. Then $\Phi_{H,M}$ is symmetric if and only if $\Phi_{H^*,M}$ is symmetric.
\end{lma}

\begin{rem}
We remark that $(H^{op})^*\iso (H^*)^{cop}$ as coalgebras:
$$\scalebox{0.84}[0.84]{
\gbeg{6}{5}
\got{2}{(H^{op})^*} \got{2}{\hspace{-0,2cm}H^{op}} \got{1}{\hspace{-0,2cm}H^{op}} \gnl
\gcmu \gcl{1} \gcn{1}{1}{2}{1} \gnl
\gcl{1} \gbr \gcl{1} \gnl
\gev \gev \gnl
\gvac{2} \gob{1}{} \gend} =
\scalebox{0.84}[0.84]{
\gbeg{5}{5}
\got{1}{(H^{op})^*} \got{2}{H^{op}} \got{2}{H^{op}} \gnl
\gcl{1} \gwmu{3} \gnl
\gcn{2}{1}{1}{3} \gcl{1} \gnl
\gvac{1} \gev \gvac{1} \gnl
\gvac{2} \gob{2}{} \gend} =
\scalebox{0.84}[0.84]{
\gbeg{4}{5}
\got{1}{\hspace{-0,16cm}H^*} \got{1}{H} \got{1}{H} \gnl
\gcn{1}{2}{0}{0} \gbr \gnl
\gvac{1} \gmu \gnl
\hspace{-0,2cm} \gwev{3} \gnl
\gvac{2} \gob{2}{} \gend} =
\scalebox{0.84}[0.84]{
\gbeg{5}{5}
\gvac{1} \got{1}{\hspace{-0,16cm}H^*} \got{1}{H} \got{1}{H} \gnl
\gcmu \gbr \gnl
\gcl{1} \gbr \gcl{1} \gnl
\gev \gev \gnl
\gvac{2} \gob{1}{} \gend}
 $$

$$\stackrel{nat.}{=}\scalebox{0.84}[0.84]{
\gbeg{5}{5}
\got{1}{\hspace{0,6cm}H^*} \gvac{1} \got{1}{H} \got{1}{H} \gnl
\gcmu \gcl{2} \gcl{2} \gnl
\gibr \gnl
\gcl{1} \gibr \gcl{1} \gnl
\gev \gev \gnl
\gvac{2} \gob{1}{} \gend}
\stackrel{nat.}{=} \scalebox{0.84}[0.84]{
\gbeg{5}{5}
\got{1}{\hspace{0,6cm}H^*} \gvac{1} \got{1}{H} \got{1}{H} \gnl
\gcmu \gcl{2} \gcl{2} \gnl
\gbr \gnl
\gcl{1} \gbr \gcl{1} \gnl
\gev \gev \gnl
\gvac{2} \gob{1}{}
\gend}\stackrel{\Phi_{H^*, H^*}}{=}
\scalebox{0.84}[0.84]{
\gbeg{5}{5}
\got{1}{\hspace{0,6cm}H^*} \gvac{1} \got{1}{H} \got{1}{H} \gnl
\gcmu \gcl{2} \gcl{2} \gnl
\gibr \gnl
\gcl{1} \gbr \gcl{1} \gnl
\gev \gev \gnl
\gvac{2} \gob{1}{} \gend} =
\scalebox{0.84}[0.84]{
\gbeg{6}{5}
\got{2}{\hspace{0,3cm}(H^*)^{cop}} \got{2}{H} \got{1}{H} \gnl
\gcmu \gcn{1}{1}{2}{1} \gcn{1}{1}{3}{1} \gnl
\gcl{1} \gbr \gcl{1} \gnl
\gev \gev \gnl
\gvac{2} \gob{1}{} \gend}
$$
%where at the place $\star$ any of the symmetricities $\Phi_{H,H}, \Phi_{H,H^*}$ or $\Phi_{H^*,H^*}$
%would work.
The claim follows by the universal property of $H^*\ot H^*\iso(H\ot H)^*$.
If $\Phi_{H,H}$ is symmetric, then $(H^{op})^*$ and $(H^*)^{cop}$
are bialgebras %(comment on page \pageref{op,cop-bialg})
and they are isomorphic as Hopf algebras.
\end{rem}

\begin{rem}
There are several ways to construct a Drinfel'd double. In \cite[Prop. 3.6]{BD} one can find a
construction of a matched pair of bialgebras, and hence a bicrossproduct $H\bowtie A$. With
$H\w=\ö(A^{op})^*$ and the pairing $\langle ., .\rangle\w=\ö\ev$ it is given a different construction
than the one in our \prref{Drinfelddouble}.
Taking $A\w=\ö(H^{cop})^*$ and $\langle ., .\rangle\w=\ö\crta\ev$, one obtains a Drinfel'd double of the form
$H\ö\bowtie\ö (H^{cop})^*\iso H\ö\bowtie\ö (H^*)^{op}$. The authors proved that if $A$ and $H$ are Hopf algebras
where the antipode of $A$ is invertible, than $A$ and $H$ are a matched pair of bialgebras if and
only if $\Phi_{A, H}$ is symmetric.
In \cite[Theorem 3.2]{Z} a result similar to our \prref{Drinfelddouble} is proved, but the $H$- and $H^{op}$-actions
are given differently. The quasitriangularity of $D(H)$ we will discus in the next section.
\end{rem}

Developing the right hand-side of the expression \equref{B tie H -mod} applied to the
Drinfel'd double and using the actions given in \prref{Drinfelddouble}, yields:
\begin{equation}\eqlabel{D(H)-action}
\gbeg{5}{9}
\got{2}{H} \got{2}{B} \got{1}{M} \gnl
\gcmu \gcmu \gcl{7} \gnl
\gcl{1} \gbr \gcl{1} \gnl
\glm \grm \gnl
\gvac{1} \gcl{1} \gcn{1}{1}{1}{3} \gnl
\gvac{1} \gcn{1}{2}{1}{5} \gvac{1} \glm \gnl
\gvac{4} \gcl{1} \gnl
\gvac{3} \glm \gnl
\gvac{3} \gob{3}{M}
\gend=
\gbeg{8}{14}
\gvac{4} \got{2}{H} \got{1}{B} \got{1}{M} \gnl
\gvac{4} \gcmu \gcl{4} \gcl{11} \gnl
\gvac{1} \gdb \gvac{1} \hspace{-0,42cm} \gcn{1}{1}{3}{2} \gvac{1} \gmp{-} \gnl %\gvac{1} \hspace{-0,48cm}
\gvac{2} \gibr \gcmu \gcl{1} \gnl
\gvac{2} \gcl{4} \gmu \gbr \gnl
\gvac{3} \gcn{1}{1}{2}{3} \gvac{1} \gcl{1} \gbr \gnl
\gvac{4} \gbr \gcl{1} \gcl{3} \gnl
\gvac{4} \gmu \gcn{1}{1}{1}{0} \gnl
\gvac{2} \gcn{2}{2}{1}{5} \gvac{1} \hspace{-0,34cm} \glmpt \gnot{\hspace{-0,32cm} \crta \ev} \grmptb \gnl
\gvac{6} \gcn{1}{1}{4}{1} \gnl
\gvac{5} \hspace{-0,26cm} \glmp \gnot{\hspace{-0,4cm} B\bowtie H} \grmp \gnl
\gvac{6} \hspace{-0,12cm} \gcn{2}{1}{2}{5} \gnl
\gvac{8} \hspace{-0,12cm} \glm \gnl
\gvac{9} \gob{1}{M.}
\gend
\end{equation}
Taking $M\ö=\ö B\bowtie H$ and applying the above equality to $H\ö\ot B\ö\ot\eta_{B\bowtie H}$, one gets:
\begin{equation} \eqlabel{D(H)-mult}
\gbeg{5}{6}
\got{2}{H} \got{2}{B} \gnl
\gcmu \gcmu \gnl
\gcl{1} \gbr \gcl{1} \gnl
\glm \grm \gnl
\gvac{1} \gcl{1} \gcl{1} \gnl
\gvac{1} \gob{1}{B} \gob{1}{H}
\gend=
\gbeg{8}{11}
\gvac{3} \got{2}{H} \got{1}{B} \gnl
\gvac{3} \gcmu \gcl{4} \gnl
\gdb \gvac{1} \hspace{-0,42cm} \gcn{1}{1}{3}{2} \gvac{1} \gmp{-} \gnl
\gvac{1} \gibr \gcmu \gcl{1} \gnl
\gvac{1} \gcl{6} \gmu \gbr \gnl
\gvac{2} \gcn{1}{1}{2}{3} \gvac{1} \gcl{1} \gbr \gnl
\gvac{3} \gbr \gcl{1} \gcl{4} \gnl
\gvac{3} \gmu \gcn{1}{1}{1}{0} \gnl
\gvac{4} \hspace{-0,34cm} \glmpt \gnot{\hspace{-0,32cm} \crta \ev} \grmptb \gnl
\gvac{1} \gob{2}{B} \gvac{3} \gob{2}{H}
\gend
\end{equation}

The following result generalizes \cite[Proposition 4.6]{Rad2} to the braided case.

\begin{lma}
Assume that %$\Phi_{H,H}$ and
$\Phi_{H,H^*}$ is symmetric. Then the following are equivalent:

(i) $D(H)$ is commutative,

(ii) $H$ and $H^*$ are commutative;

(iii) $H$ and $H^*$ are cocommutative;

(iv) $D(H)$ is cocommutative.
\end{lma}

\begin{proof}
In view of \inref{alg-coalg} it suffices to prove the equivalence of (i) and (ii). We omit to type the
whole proof, we only give a sketch of it. First observe that we have identities: \vspace{-0,5cm}
%\begin{center}
%\begin{tabular}{p{4cm}p{1cm}p{9.6cm}}
%$$\hspace{-0,5cm}
%\scalebox{0.84}[0.84]{ \gbeg{4}{4} \got{1}{B} \got{1}{\hspace{0,16cm}B} \got{3}{H^{op}} \gnl
%\gcl{1} \gcl{1} \gcmu \gnl \gcl{1} \gbr \gcl{1} \gnl \gev \gev \gnl \gvac{2} \gob{2}{} \gend}
%\stackrel{(\ref{multH*})}{{=}}\hspace{-0,24cm} \scalebox{0.84}[0.84]{ \gbeg{4}{4}
%\got{1}{B} \got{1}{\hspace{0,16cm}B} \got{3}{H^{op}} \gnl \gcl{1} \gcl{1} \gcn{1}{2}{2}{2} \gnl
%\gmu \gnl \gvac{1} \hspace{-0,2cm} \gwev{3} \gnl \gvac{2} \gob{2}{}
%\gend}\quad\textnormal{implies:}\quad $$  & & \vspace{-0,4cm}
\begin{equation}\eqlabel{mult in B}
\scalebox{0.84}[0.84]{ \gbeg{4}{4}
\got{2}{H^{op}} \got{1}{B} \got{1}{B} \gnl
\gcmu \gcl{1} \gcl{2} \gnl \gcl{1} \gbr \gnl
\glmpt \gnot{\hspace{-0,32cm} \crta \ev} \grmpt \glmpt \gnot{\hspace{-0,32cm} \crta \ev} \grmpt \gnl
\gvac{2} \gob{1}{} \gend}=%\hspace{-0,24cm}=\hspace{-0,14cm}
\scalebox{0.84}[0.84]{
\gbeg{4}{5}
\got{2}{H^{op}} \got{1}{B} \got{1}{B} \gnl
\gcmu \gcl{1} \gcl{2} \gnl
\gcl{1} \gbr \gnl
\gbr \gbr \gnl
\gev \gev \gnl
\gvac{2} \gob{1}{} \gend} \hspace{-0,14cm}
\stackrel{\Phi_{H,H^*}}{{=}}\hspace{-0,14cm}
\scalebox{0.84}[0.84]{
\gbeg{4}{6}
\got{2}{H^{op}} \got{1}{B} \got{1}{B} \gnl
\gcmu \gcl{1} \gcl{2} \gnl
\gcl{1} \gbr \gnl
\gbr \gbr \gnl
\gcl{1} \gbr \gcl{1} \gnl
\gcl{1} \gbr \gcl{1} \gnl
\gev \gev \gnl
\gvac{2} \gob{1}{} \gend} \stackrel{(\ref{multH*})}{{=}} %\hspace{-0,14cm}\hspace{-0,14cm}
\scalebox{0.84}[0.84]{
\gbeg{4}{5}
\got{1}{H^{op}} \got{1}{B} \got{3}{B} \gnl
\gcl{1} \gwmu{3} \gnl
\gcn{2}{1}{1}{3} \gcl{1} \gnl
\gvac{1} \gbr \gnl
\gvac{1} \gev \gnl
\gvac{2} \gob{2}{} \gend}=
\scalebox{0.84}[0.84]{ \gbeg{4}{4}
\got{1}{H^{op}} \got{1}{B} \got{1}{B} \gnl
\gcn{1}{1}{0}{0} \gmu \gnl
\gcn{1}{1}{0}{2} \gcn{1}{1}{2}{0} \gnl
\gvac{1} \hspace{-0,34cm} \gmp{\crta\ev} \gnl
\gvac{2} \gob{2}{} \gend}
\end{equation}
%\end{tabular}%\vspace{-0,8cm}
%\end{center}
and
\begin{equation}\eqlabel{move cuerda}
\gbeg{3}{6}
\got{1}{H^{op}} \gnl
\gcl{1} \gdb \gnl
\gcl{1} \gibr \gnl
\gbr \gcl{2} \gnl
\gev \gnl
\gob{5}{H^{op}}
\gend\stackrel{nat.}{=}
\gbeg{3}{4}
\got{1}{} \got{3}{H^{op}} \gnl
\gdb \gcl{1} \gnl
\gcl{1} \gev \gnl
\gob{1}{H^{op}}
\gend=id_{H^{op}}
\end{equation}
Suppose that $D(H)$ is commutative. Using $\crta\ev=\ev\Phi$ and evaluating the product
in $D(H)$ at $H^{op}$, we obtain:
\begin{equation} \eqlabel{comm proof}
\gbeg{8}{11}
\got{1}{H^{op}} \got{1}{B} \gvac{2} \got{2}{H} \got{1}{B} \got{1}{H} \gnl
\gbr \gvac{2} \gcmu \gcl{4} \gcl{5} \gnl
\gcl{2} \gcn{1}{1}{1}{2} \gvac{1}  \gcn{1}{1}{3}{2} \gvac{1} \gmp{-} \gnl
\gvac{1} \gcmu \gcmu \gcl{1} \gnl
\gev \gmu \gbr \gnl
\gvac{2} \gcn{1}{1}{2}{3} \gvac{1} \gcl{1} \gbr \gnl
\gvac{3} \gbr \gcl{1} \gmu \gnl
\gvac{3} \gmu \gcn{1}{1}{1}{0} \gcn{1}{3}{2}{2} \gnl
\gvac{4} \hspace{-0,34cm} \gbr \gnl
\gvac{4} \gev \gnl
\gvac{7} \gob{1}{H}
\gend \stackrel{\nabla_{D(H)}}{\stackrel{\equref{move cuerda}}{=}}
\gbeg{4}{6}
\got{1}{\hspace{-0,2cm}H^{op}} \got{1}{\hspace{0,2cm}B\bowtie H} \got{3}{B\bowtie H} \gnl
\gcl{2} \gwmu{3} \gnl
\gvac{1} \glmpb \gnot{\hspace{-0,4cm} B\bowtie H} \grmptb \gnl
\gbr \gcl{2} \gnl
\gev \gnl
\gvac{2} \gob{1}{B}
\gend
=
\gbeg{4}{9}
\got{1}{\hspace{-0,2cm}H^{op}} \got{1}{\hspace{0,2cm}B\bowtie H} \got{3}{B\bowtie H} \gnl
\gcl{5} \gcn{1}{1}{1}{2} \gcn{1}{1}{3}{2} \gnl
\gvac{2} \hspace{-0,34cm} \gbr \gnl
\gvac{1} \gcn{1}{1}{3}{2} \gcn{1}{1}{3}{4} \gnl
\gvac{2} \hspace{-0,34cm} \gwmu{3} \gnl
\gvac{2} \glmpb \gnot{\hspace{-0,4cm} B\bowtie H} \grmptb \gnl
\gvac{1} \gbr \gcl{2} \gnl
\gvac{1} \gev \gnl
\gvac{3}
\gob{1}{B}
\gend \stackrel{\nabla_{D(H)}}{\stackrel{\equref{move cuerda}}{=}}
\gbeg{8}{14}
\gvac{1} \got{1}{H^{op}} \gvac{2} \got{1}{B} \got{1}{H} \got{1}{B} \got{1}{H} \gnl
\gvac{1} \gcl{1} \gvac{2} \gcl{1} \gbr \gcl{1} \gnl
\gvac{1} \gcl{1} \gvac{2} \gbr \gbr \gnl
\gvac{1} \gcl{1} \gcn{2}{1}{5}{1} \gvac{1} \gbr \gcl{1} \gnl
\gvac{1} \gbr \gvac{2} \hspace{-0,22cm} \gcmu \gcn{1}{1}{0}{1} \gcn{1}{1}{0}{1} \gnl
\gvac{1} \gcn{1}{1}{2}{1} \gcn{1}{1}{2}{2} \gvac{1}  \gcn{1}{1}{3}{2} \gvac{1} \gmp{-} \gcl{3} \gcl{4} \gnl
\gvac{1} \gcl{1} \gcmu \gcmu \gcl{1} \gnl
\gvac{1} \gev \gmu \gbr \gnl
\gvac{3} \gcn{1}{1}{2}{3} \gvac{1} \gcl{1} \gbr \gnl
\gvac{4} \gbr \gcl{1} \gmu \gnl
\gvac{4} \gmu \gcn{1}{1}{1}{0} \gcn{1}{3}{2}{2} \gnl
\gvac{5} \hspace{-0,34cm} \gbr \gnl
\gvac{5} \gev \gnl
\gvac{8} \gob{1}{H}
\gend
\end{equation}
Apply this to $H^{op}\ot B\ot\eta_H\ot B\ot\eta_H$ and compose the obtained identity with $\Epsilon_H$
to obtain:
$$
\gbeg{4}{6}
\got{1}{H^{op}} \got{1}{B} \got{3}{B} \gnl
\gbr \gvac{1} \gcl{3} \gnl
\gcl{2} \gcn{1}{1}{1}{2} \gnl
\gvac{1} \gcmu \gnl
\gev \gbr \gnl
\gvac{2} \gev
\gob{1}{}
\gend=
\gbeg{4}{7}
\got{1}{H^{op}} \got{1}{B} \got{1}{B} \gnl
\gcl{1} \gbr \gnl
\gbr \gcn{1}{1}{1}{3} \gnl
\gcl{2} \gcn{1}{1}{1}{2} \gvac{1} \gcl{2} \gnl
\gvac{1} \gcmu \gnl
\gev \gbr \gnl
\gvac{2} \gev
\gob{1}{}
\gend
$$
By \equref{mult in B} %and \leref{Zhang lemma}, %\rmref{Phi on H*-H* mult},
one gets that $B$, and hence $H^*$, is commutative. Applying \equref{comm proof} to
$\eta_H\ot\eta_B\ot H\ot\eta_B\ot H$, one obtains that $H$ is commutative.

Conversely, assuming (ii), using \equref{mult in B} and that $\Phi_{H,H^*}$ is symmetric, one may
prove that \equref{comm proof} - which expresses commutativity of $D(H)$ - holds true.
\qed\end{proof}

\section{Yetter-Drinfel'd modules as modules over the Drinfel'd double} \selabel{YD-DH}
\setcounter{equation}{0}

Since $D(H)$ is a bialgebra in $\C$, the category of its left (and right) modules
is monoidal. In this and the next section we study the isomorphism between these categories and the
appropriate categories of YD-modules. The functors we will consider
%will define functors between the categories of YD-modules
%and modules over the Drinfel'd double. They all
will act as identity functors on objects and morphisms,
we will only have to define the new (co)module structures. Let us regard the pair of functors
$$\hspace{-0,24cm}
\bfig
\putmorphism(200,30)(1,0)[\F: {}_{(H^{op})^*\bowtie H}\C` {}_H\YD(\C)^{H^{op}}: \G.`]{900}0a
\putmorphism(170,50)(1,0)[\phantom{{}_H\YD(\C)^{H^{op}}: \G}`\phantom{\F: {}_{(H^{op})^*\bowtie H}\C}` ]{900}1a
\putmorphism(160,10)(1,0)[\phantom{{}_H\YD(\C)^{H^{op}}: \G}`\phantom{\F: {}_{(H^{op})^*\bowtie H}\C}` ]{880}{-1}b
\efig
$$
For $M\in {}_{(H^{op})^*\bowtie H}\C$ and $K\in {}_H\YD(\C)^{H^{op}}$ we define:
$$%\begin{eqnarray}\label{Hcomod}
\scalebox{0.9}[0.9]{
\gbeg{3}{4}
\got{1}{\F(M)} \gnl
\grcm \gnl
\gcl{1} \gcn{1}{1}{1}{3} \gnl
\gob{1}{\F(M)} \gvac{1} \gob{1}{H}
\gend} = \scalebox{0.9}[0.9]{
\gbeg{7}{5}
\got{1}{} \got{1}{} \got{1}{M} \gnl
\gdb \gcl{1} \gnl
\gcn{1}{1}{1}{3} \glm \gnl
\gvac{1} \gbr \gnl
\gvac{1} \gob{1}{M} \gob{1}{H}
\gend}\textnormal{and}\qquad
\scalebox{0.9}[0.9]{
\gbeg{3}{4}
\got{1}{H^*} \gvac{1} \got{1}{\G(K)} \gnl
\gcn{1}{1}{1}{3} \gvac{1} \gcl{1} \gnl
\gvac{1} \glm \gnl
\gvac{2} \gob{1}{\G(K)}
\gend} = \scalebox{0.9}[0.9]{
\gbeg{2}{5}
\got{1}{H^*} \got{1}{K} \gnl
\gcl{1} \grcm \gnl
\gbr \gcl{1} \gnl
\gcl{1} \gev \gnl
\gob{1}{K.}
\gend}
$$%\end{eqnarray}
Regard $\F(M)$ as a left $H$-module by the action of $\eta_B\ot H$ on $M$, and
%(not to say by the ``restriction of scalars'').
consider $\G(K)=K$ as a left $H$-module. By \prref{Hmod-H*comod} we know that $\F(M)$ is a right
$H$-comodule and $\G(K)$ a left $H^*$-module. %The two functors act as identities on morphisms.

%Assume that $\Phi_{H, H}, \Phi_{H, H^*}$ and $\Phi_{H, M}$ are symmetric, where
%$M\in {}_{(H^{op})^*\bowtie H}\C$. In a long and tedious computation using identities
%\equref{B tie H -mod}, \equref{D(H)-action}, \equref{codiag B}, \equref{crta ev}, \equref{crta d}
%and (\ref{antipode2}) one proves that with $B=(H^{op})^*$ we have: \vspace{-0,12cm}
Assume that $\Phi_{H, H}$ %, \Phi_{H, H^*}$
and $\Phi_{H, M}$ are symmetric, where
$M\in {}_{(H^{op})^*\bowtie H}\C$. Let $B=(H^{op})^*$. We have: \vspace{-0,12cm}
$$\scalebox{0.84}[0.84]{
\gbeg{6}{8}
\got{1}{B} \got{2}{H} \got{1}{\F(M)} \gnl
\gcl{3} \gcn{1}{1}{2}{2} \gcn{1}{1}{3}{3} \gnl
\gvac{1} \gcmu \grcm \gnl
\gvac{1} \gcl{1} \gbr \gcl{1} \gnl
\gcn{1}{1}{1}{3} \glm \gmu \gnl
\gvac{1} \gbr \gcn{1}{1}{2}{1} \gnl
\gvac{1} \gcl{1} \gev \gnl
\gob{3}{\F(M)}
\gend}\stackrel{\F}{=}
\scalebox{0.84}[0.84]{
\gbeg{7}{9}
\got{1}{B} \got{2}{H} \gvac{2} \got{1}{M} \gnl
\gcl{4} \gcn{1}{1}{2}{2} \gvac{1} \gdb \gcl{1} \gnl
\gvac{1} \gcmu \gcn{1}{1}{1}{3} \glm \gnl
\gvac{1} \gcl{1} \gcn{1}{1}{1}{3} \gvac{1} \gbr \gnl
\gvac{1} \gcn{1}{1}{1}{3} \gvac{1} \gbr \gcl{1} \gnl
\gcn{1}{1}{1}{5} \gvac{1} \glm \gmu \gnl
\gvac{2} \gbr \gcn{1}{1}{2}{1} \gnl
\gvac{2} \gcl{1} \gev \gnl
\gob{5}{M}
\gend} \stackrel{\equref{B-codiag}}{=}
\scalebox{0.84}[0.84]{
\gbeg{8}{8}
\got{2}{B} \got{2}{H} \gvac{2} \got{1}{M} \gnl
\gcn{1}{2}{2}{2} \gvac{1} \gcmu \gdb \gcl{1} \gnl
\gvac{2} \gcl{1} \gibr \glm \gnl
\gcmu \gibr \gcl{1} \gvac{1} \gcl{3} \gnl
\gcl{1} \gbr \gibr \gnl
\gev \gev \gcn{1}{1}{1}{3} \gnl
\gvac{5} \glm \gnl
\gvac{4} \gob{5}{M}
\gend} \stackrel{\Phi_{H,H}}{\stackrel{nat.}{\stackrel{\equref{evdb=id}}{=}}}
\scalebox{0.84}[0.84]{
\gbeg{7}{9}
\got{2}{B} \got{2}{H} \got{1}{M} \gnl
\gcmu \gcmu \gcl{5} \gnl
\gcl{1} \gbr \gcl{1} \gnl
\gbr \gev \gnl
\gcl{1} \gcn{2}{1}{1}{5} \gnl
\gcn{3}{2}{1}{7} \glm \gnl
\gvac{4} \gcl{1} \gnl
\gvac{3} \glm \gnl
\gvac{4} \gob{1}{M}
\gend}
$$

$$\stackrel{\equref{D(H)-action}}{=}
\scalebox{0.84}[0.84]{
\gbeg{10}{17}
\gvac{3} \got{2}{B} \got{2}{H} \got{1}{M} \gnl
\gvac{3} \gcmu \gcmu \gcl{11} \gnl
\gvac{3} \gcl{1} \gbr \gcl{1} \gnl
\gvac{3} \gbr \gev \gnl
\gvac{3} \hspace{-0,22cm} \gcmu \gcn{1}{1}{0}{1} \gnl
\gdb \gvac{1} \hspace{-0,42cm} \gcn{1}{1}{3}{2} \gvac{1} \gmp{-} \gcl{3} \gnl
\gvac{1} \gibr \gcmu \gcl{1} \gnl
\gvac{1} \gcl{4} \gmu \gbr \gnl
\gvac{2} \gcn{1}{1}{2}{3} \gvac{1} \gcl{1} \gbr \gnl
\gvac{3} \gbr \gcl{1} \gcl{3} \gnl
\gvac{3} \gmu \gcn{1}{1}{1}{0} \gnl
\gvac{1} \gcn{2}{2}{1}{5} \gvac{1} \hspace{-0,34cm} \glmpt \gnot{\hspace{-0,32cm} \crta \ev} \grmptb \gnl
\gvac{5} \gcn{1}{1}{4}{1} \gvac{1} \gcn{1}{3}{5}{1} \gnl
\gvac{4} \hspace{-0,24cm} \glmp \gnot{\hspace{-0,4cm} B\bowtie H} \grmp \gnl
\gvac{5} \hspace{-0,12cm} \gcn{2}{1}{1}{4} \gnl
\gvac{7} \hspace{-0,34cm}\glm \gnl
\gvac{8} \gob{1}{M}
\gend} \stackrel{nat.}{=}
\scalebox{0.84}[0.84]{
\gbeg{9}{18}
\gvac{4} \got{1}{B} \got{1}{H} \got{3}{M} \gnl
\gvac{4} \gbr \gvac{1} \gcl{12} \gnl
\gvac{3} \gcn{1}{1}{3}{2} \gvac{1} \gcn{1}{1}{1}{2} \gnl
\gvac{3} \gcmu \gcmu \gnl
\gvac{3} \gcl{1} \gibr \gcl{1} \gnl
\gdb \gvac{1} \hspace{-0,2cm} \gcmu \gcn{1}{1}{0}{1} \hspace{-0,22cm} \gibr \gnl
\gvac{1} \gibr \gcmu \gcn{1}{1}{0}{1} \gcn{1}{1}{0}{0} \hspace{-0,4cm} \gev \gnl
\gvac{2} \gcl{5} \gmu \gcl{1} \gmp{-} \gcn{1}{1}{0}{1} \gnl
\gvac{3} \gcn{1}{1}{2}{2} \gvac{1} \gbr \gcl{1} \gnl
\gvac{3} \gcn{1}{1}{2}{3} \gvac{1} \gcl{1} \gbr \gnl
\gvac{4} \gbr \gcl{1} \gcl{3} \gnl
\gvac{4} \gmu \gcn{1}{1}{1}{0} \gnl
\gvac{2} \gcn{2}{2}{1}{5} \gvac{1} \hspace{-0,36cm} \glmpt \gnot{\hspace{-0,32cm} \crta \ev} \grmptb \gnl
\gvac{6} \gcn{1}{1}{4}{1} \gvac{1} \gcn{1}{3}{4}{1} \gnl
\gvac{5} \hspace{-0,24cm} \glmp \gnot{\hspace{-0,4cm} B\bowtie H} \grmp \gnl
\gvac{6} \hspace{-0,12cm} \gcn{2}{1}{1}{4} \gnl
\gvac{8} \hspace{-0,34cm} \glm \gnl
\gvac{9} \gob{1}{M}
\gend} \stackrel{\Phi_{H,B}}{\stackrel{\equref{crta ev}}{=}}
\scalebox{0.84}[0.84]{
\gbeg{9}{18}
\gvac{4} \got{1}{B} \got{1}{H} \got{3}{M} \gnl
\gvac{4} \gbr \gvac{1} \gcl{12} \gnl
\gvac{3} \gcn{1}{1}{3}{2} \gvac{1} \gcn{1}{1}{1}{2} \gnl
\gvac{3} \gcmu \gcmu \gnl
\gvac{3} \gcl{1} \gbr \gcl{1} \gnl
\gdb \gvac{1} \hspace{-0,22cm} \gcmu \gcn{1}{1}{0}{0} \hspace{-0,2cm}
   \glmpt \gnot{\hspace{-0,32cm} \crta \ev} \grmptb \gnl
\gvac{1} \gibr \gcmu \gcn{1}{1}{0}{1} \gcn{1}{1}{-1}{0} \gnl
\gvac{1} \gcl{5} \gmu \gcl{1} \gmp{-} \gcn{1}{1}{0}{1} \gnl
\gvac{2} \gcn{1}{1}{2}{2} \gvac{1} \gbr \gcl{1} \gnl
\gvac{2} \gcn{1}{1}{2}{3} \gvac{1} \gcl{1} \gbr \gnl
\gvac{3} \gbr \gcl{1} \gcl{3} \gnl
\gvac{3} \gmu \gcn{1}{1}{1}{0} \gnl
\gvac{1} \gcn{2}{2}{1}{5} \gvac{1} \hspace{-0,36cm} \glmpt \gnot{\hspace{-0,32cm} \crta \ev} \grmptb \gnl
\gvac{5} \gcn{1}{1}{4}{1} \gvac{1} \gcn{1}{3}{4}{1} \gnl
\gvac{4} \hspace{-0,24cm} \glmp \gnot{\hspace{-0,4cm} B\bowtie H} \grmp \gnl
\gvac{5} \hspace{-0,12cm} \gcn{2}{1}{1}{4} \gnl
\gvac{7} \hspace{-0,34cm}\glm \gnl
\gvac{8} \gob{1}{M}
\gend}
$$

$$\stackrel{nat.}{\stackrel{\equref{codiag B}}{=}}
\scalebox{0.84}[0.84]{
\gbeg{9}{18}
\gvac{4} \got{1}{B} \got{1}{H} \got{3}{M} \gnl
\gvac{4} \gbr \gvac{1} \gcl{14} \gnl
\gvac{3} \gcn{1}{1}{3}{2} \gvac{1} \gcn{1}{1}{1}{2} \gnl
\gvac{3} \gcmu \gcn{1}{1}{2}{2} \gnl
\gdb \gvac{1} \hspace{-0,2cm} \gcmu \gcn{1}{1}{0}{2} \gcn{1}{1}{1}{2} \gnl
\gvac{1} \hspace{-0,36cm} \gibr \gcmu \gcn{1}{1}{0}{1} \gcl{6} \gcl{8} \gnl
\gvac{1} \gcl{5} \gmu \gcl{1} \gmp{-} \gnl
\gvac{2} \gcn{1}{1}{2}{2} \gvac{1} \gbr \gcl{1} \gnl
\gvac{2} \gcn{1}{1}{2}{3} \gvac{1} \gcl{1} \gcl{2} \gnl
\gvac{3} \gbr \gnl
\gvac{3} \gmu \gcn{1}{1}{1}{0} \gnl
\gvac{1} \gcn{2}{1}{1}{4} \gvac{1} \hspace{-0,36cm} \gibr \gcn{1}{1}{2}{1} \gnl
\gvac{3} \glmp \gnot{\hspace{-0,4cm} B\bowtie H} \grmpb \gbr \gnl
\gvac{4} \gcl{2} \gmu \gcn{1}{1}{2}{0} \gnl
\gvac{6} \hspace{-0,34cm} \glmpt \gnot{\hspace{-0,32cm} \crta \ev} \grmptb \gnl
\gvac{5} \gcn{2}{1}{0}{4} \gcn{2}{1}{5}{2} \gnl
\gvac{7} \hspace{-0,34cm} \glm \gnl
\gvac{8} \gob{1}{M}
\gend} \stackrel{coass.}{\stackrel{ass.}{\stackrel{nat.}{=}}}
\scalebox{0.84}[0.84]{
\gbeg{9}{14}
\gvac{4} \got{1}{B} \got{1}{H} \got{3}{M} \gnl
\gvac{4} \gbr \gvac{1} \gcl{10} \gnl
\gvac{4} \gcl{1} \gcn{1}{1}{1}{3} \gnl
\gvac{1} \hspace{-0,34cm} \gdb \gvac{1} \hspace{-0,22cm} \gwcm{3} \gcn{1}{6}{1}{1} \gnl
\gvac{2} \hspace{-0,34cm} \gibr \gvac{1} \hspace{-0,42cm} \gcmu \gcmu \gnl
\gvac{3} \gcl{1} \gmu \gcn{1}{1}{1}{0} \gmp{-} \gcl{1} \gnl
\gvac{1} \gcn{1}{1}{5}{6} \gvac{3} \hspace{-0,36cm} \gibr \gvac{1} \hspace{-0,2cm} \gbr \gnl
\gvac{5} \hspace{-0,22cm} \glmpt \gnot{\hspace{-0,4cm} B\bowtie H} \grmptb \gcn{1}{1}{1}{3}
   \gvac{1} \hspace{-0,34cm} \gmu \gnl
\gvac{7} \gcn{1}{1}{0}{2} \gvac{1} \hspace{-0,34cm} \gbr \gnl
\gvac{8} \gcl{2} \gmu \gcn{1}{1}{1}{0} \gnl
\gvac{10} \hspace{-0,34cm} \glmpt \gnot{\hspace{-0,32cm} \crta \ev} \grmptb \gnl
\gvac{9} \gcn{2}{1}{0}{4} \gcn{2}{1}{4}{2} \gnl
\gvac{11} \hspace{-0,34cm} \glm \gnl
\gvac{12} \gob{1}{M}
\gend} \stackrel{\Phi_{H,H}}{\stackrel{(\ref{antipode2})}{\stackrel{\equref{BHmod-conv}}{=}}}
\scalebox{0.84}[0.84]{
\gbeg{6}{11}
\gvac{2} \got{1}{B} \got{1}{H} \got{1}{M} \gnl
\hspace{-0,4cm} \gdb \gvac{1} \hspace{-0,22cm} \gbr \gcn{1}{6}{1}{1} \gnl
\gvac{1} \hspace{-0,34cm} \gibr \gvac{1} \hspace{-0,42cm} \gcmu \gcn{1}{3}{0}{0}\gnl
\gvac{2} \gcl{4} \gmu \gcn{1}{1}{1}{0} \gnl
\gvac{4} \hspace{-0,36cm} \gibr \gnl
\gvac{4} \gcl{1} \glmpt \gnot{\hspace{-0,32cm} \crta \ev} \grmptb \gnl
\gvac{5} \gcn{2}{1}{-1}{3} \gnl
\gvac{2} \gcn{1}{2}{2}{9} \gvac{3} \glm \gnl
\gvac{7} \gcl{1} \gnl
\gvac{6} \glm \gnl
\gvac{7} \gob{1}{M}
\gend}
=:\Sigma.
$$
On the other hand, it is:
$$
\scalebox{0.84}[0.84]{
\gbeg{5}{10}
\got{1}{B} \got{3}{H} \got{1}{\F(M)} \gnl
\gcl{6} \gwcm{3} \gcl{1} \gnl
\gvac{1} \gcn{1}{2}{1}{3} \gvac{1} \glm \gnl
\gvac{2} \gcn{1}{1}{5}{3} \gnl
\gvac{2} \gbr \gnl
\gvac{1} \gcn{1}{1}{3}{1} \gcn{1}{1}{3}{5} \gnl
\gvac{1} \grcm \gcn{1}{1}{3}{3} \gnl
\gbr \gwmu{3} \gnl
\gcl{1} \gwev{3} \gnl
\gob{1}{\F(M)}
\gend}\stackrel{\F}{=}
\scalebox{0.84}[0.84]{
\gbeg{6}{10}
\got{1}{B} \gvac{1} \got{2}{H} \got{1}{M} \gnl
\gcl{4} \gvac{1} \gcmu \gcl{1} \gnl
\gvac{2} \gcn{1}{1}{1}{3} \glm \gnl
\gvac{1} \gdb \gbr \gnl
\gvac{1} \gcn{1}{1}{1}{3} \glm \gcl{2} \gnl
\gcn{1}{1}{1}{3} \gvac{1} \gbr \gnl
\gvac{1} \gbr \gmu \gnl
\gvac{1} \gcl{2} \gcl{1}\gcn{1}{1}{2}{1} \gnl
\gvac{2} \gev \gnl
\gvac{1} \gob{1}{M}
\gend} \stackrel{nat.}{\stackrel{\equref{crta d}}{\stackrel{\Phi_{H,B}}{=}}}\hspace{-0,26cm}
\scalebox{0.84}[0.84]{
\gbeg{6}{11}
\gvac{2} \got{1}{B} \got{1}{H} \got{1}{M} \gnl
\gvac{2} \gbr \gcl{1} \gnl
\gvac{2} \gcn{1}{1}{1}{0} \gbr \gnl
\gvac{1} \gcmu \gcl{1} \gcl{5} \gnl
\gvac{1} \gcn{1}{1}{1}{3} \glm \gnl
\gdb \gbr \gnl
\gcn{1}{1}{1}{3} \glm \gcl{2} \gnl
\gvac{1} \gbr \gnl
\gvac{1} \gcl{2} \gmu \gcn{1}{1}{1}{0} \gnl
\gvac{3} \hspace{-0,34cm} \glmpt \gnot{\hspace{-0,32cm} \crta \ev} \grmptb \gnl
\gvac{1} \gob{2}{M}
\gend} \stackrel{nat.}{=} \hspace{-0,26cm}
\scalebox{0.84}[0.84]{
\gbeg{7}{10}
\gvac{3} \got{1}{B} \got{1}{H} \got{1}{M} \gnl
\gvac{3} \gbr \gcl{1} \gnl
\gvac{3} \gcn{1}{1}{1}{0} \gbr \gnl
\gdb \gcmu \gcl{1} \gcl{4} \gnl
\gbr \gcl{1} \glm \gnl
\gcl{1} \gmu \gcn{1}{1}{3}{0} \gnl
\gcn{1}{1}{1}{2} \gvac{1} \hspace{-0,34cm} \gbr \gnl
\gvac{1} \glm \gcl{1} \gcn{1}{1}{4}{1} \gnl
\gvac{2} \gcl{1} \glmpt \gnot{\hspace{-0,32cm} \crta \ev} \grmptb \gnl
\gvac{2} \gob{1}{M}
\gend} \hspace{-0,2cm}\stackrel{nat.}{=}\hspace{-0,2cm}
\scalebox{0.84}[0.84]{
\gbeg{6}{10}
\gvac{3} \got{1}{B} \got{1}{H} \got{1}{M} \gnl
\gvac{3} \gbr \gcl{1} \gnl
\gvac{3} \gcn{1}{1}{1}{0} \gcl{1} \gcl{1} \gnl
\gdb \gcmu \gcl{1} \gcl{4} \gnl
\gbr \gcl{1} \gibr \gnl
\gcl{2} \gmu \gcn{1}{1}{1}{0} \glm \gnl
\gvac{2} \hspace{-0,2cm} \glmpt \gnot{\hspace{-0,32cm} \crta \ev} \grmptb \gnl
\gcn{1}{1}{2}{6} \gcn{1}{1}{10}{6} \gnl
\gvac{3} \hspace{-0,36cm} \glm \gnl
\gvac{4} \gob{1}{M}
\gend} \hspace{0,2cm}\stackrel{\Phi_{H,B}}{\stackrel{\Phi_{H,H}}{\stackrel{nat.}{=}}}\Sigma
$$
From the universal property of $H^*=[H, I]$ the obtained identity implies that $\F(M)$ obeys
(\ref{YD-mix*}), thus $\F$ is well defined.
%Using \equref{crta ev}, \equref{crta d} and (\ref{antipode2}) one
%proves that $\G(K)$ satisfies the identity \equref{B tie H -mod} (apply to it \equref{D(H)-action}).
For the converse assume that moreover $\Phi_{H, K}$ is symmetric for $K\in {}_H\YD(\C)^{H^{op}}$.
We will need:
\begin{equation} \eqlabel{petljaH}
\scalebox{0.9}[0.9]{
\gbeg{3}{6}
\got{1}{} \got{3}{H} \gnl
\gdb \gcl{2} \gnl
\gibr \gnl
\gcl{1} \gibr \gnl
\gev \gcl{1} \gnl
\gob{5}{H}
\gend}=\scalebox{0.9}[0.9]{
\gbeg{3}{4}
\got{1}{} \got{3}{H} \gnl
\gdb \gcl{1} \gnl
\gcl{1} \gev \gnl
\gob{1}{H}
\gend}\stackrel{\equref{dbev=id}}{=} id_H
\end{equation}
Now we compute:
$$\scalebox{0.9}[0.9]{
\gbeg{9}{14}
\gvac{4} \got{2}{H} \got{1}{B} \got{1}{\hspace{0,24cm}\G(K)} \gnl
\gvac{4} \gcmu \gcl{4} \gcl{11} \gnl
\gvac{1} \gdb \gvac{1} \hspace{-0,42cm} \gcn{1}{1}{3}{2} \gvac{1} \gmp{-} \gnl %\gvac{1} \hspace{-0,48cm}
\gvac{2} \gibr \gcmu \gcl{1} \gnl
\gvac{2} \gcl{4} \gmu \gbr \gnl
\gvac{3} \gcn{1}{1}{2}{3} \gvac{1} \gcl{1} \gbr \gnl
\gvac{4} \gbr \gcl{1} \gcl{3} \gnl
\gvac{4} \gmu \gcn{1}{1}{1}{0} \gnl
\gvac{2} \gcn{2}{2}{1}{5} \gvac{1} \hspace{-0,34cm} \glmpt \gnot{\hspace{-0,32cm} \crta \ev} \grmptb \gnl
\gvac{6} \gcn{1}{1}{4}{1} \gnl
\gvac{5} \hspace{-0,26cm} \glmp \gnot{\hspace{-0,4cm} B\bowtie H} \grmp \gnl
\gvac{6} \hspace{-0,12cm} \gcn{2}{1}{2}{5} \gnl
\gvac{8} \hspace{-0,12cm} \glm \gnl
\gvac{9} \gob{1}{\G(K)}
\gend} \stackrel{\equref{BHmod-conv}}{\stackrel{\G}{=}}
\scalebox{0.9}[0.9]{
\gbeg{8}{14}
\gvac{3} \got{2}{H} \got{1}{B} \got{1}{K} \gnl
\gvac{3} \gcmu \gcl{4} \gcl{5} \gnl
\gdb \gvac{1} \hspace{-0,42cm} \gcn{1}{1}{3}{2} \gvac{1} \gmp{-} \gnl
\gvac{1} \gibr \gcmu \gcl{1} \gnl
\gvac{1} \gcl{4} \gmu \gbr \gnl
\gvac{2} \gcn{1}{1}{2}{3} \gvac{1} \gcl{1} \gbr \gnl
\gvac{3} \gbr \gcl{1} \glm \gnl
\gvac{3} \gmu \gcn{1}{1}{1}{0} \gvac{1} \gcl{1} \gnl
\gvac{1} \gcn{2}{2}{1}{6} \gvac{1} \hspace{-0,34cm} \glmpt \gnot{\hspace{-0,32cm} \crta \ev} \grmptb
   \gcn{2}{2}{4}{-1} \gnl \gnl
\gvac{4} \gcl{1} \grcm \gnl
\gvac{4} \gbr \gcl{1} \gnl
\gvac{4} \gcl{1} \gev \gnl
\gvac{4} \gob{1}{K}
\gend}\stackrel{\equref{petljaH}}{\stackrel{nat.}{=}}
\scalebox{0.9}[0.9]{
\gbeg{7}{17}
\gvac{1} \got{2}{H} \got{1}{B} \got{1}{K} \gnl
\gvac{1} \gcmu \gcl{4} \gcl{5} \gnl
\gcn{1}{1}{3}{2} \gvac{1} \gmp{-} \gnl
\gcmu \gcl{1} \gnl
\gcl{6} \gbr \gnl
\gvac{1} \gcl{4} \gbr \gnl
\gvac{2} \gcl{2} \glm \gnl
\gvac{4} \grcm \gnl
\gvac{2} \gcn{1}{1}{1}{3} \gvac{1} \gibr \gnl
\gvac{1} \gcn{1}{1}{1}{3} \gvac{1} \gibr \gcl{7} \gnl
\gcn{1}{1}{1}{3} \gvac{1} \gibr \gcl{4} \gnl
\gvac{1} \gibr \gcl{1} \gnl
\gvac{1} \gmu \gcn{1}{1}{1}{0} \gnl
\gvac{2} \hspace{-0,34cm} \gbr \gnl
\gvac{2} \gmu \gcn{1}{1}{2}{0} \gnl
\gvac{3} \hspace{-0,34cm} \glmpt \gnot{\hspace{-0,32cm} \crta \ev} \grmptb \gnl
\gvac{6} \gob{1}{K}
\gend} \stackrel{\Phi_{H,H}}{\stackrel{\Phi_{H,B}}{=}}%\hspace{-0,26cm}
\scalebox{0.9}[0.9]{
\gbeg{7}{17}
\gvac{1} \got{2}{H} \got{1}{B} \got{1}{K} \gnl
\gvac{1} \gcmu \gcl{4} \gcl{5} \gnl
\gcn{1}{1}{3}{2} \gvac{1} \gmp{-} \gnl
\gcmu \gcl{1} \gnl
\gcl{6} \gibr \gnl
\gvac{1} \gcl{4} \gibr \gnl
\gvac{2} \gcl{2} \glm \gnl
\gvac{4} \grcm \gnl
\gvac{2} \gcn{1}{1}{1}{3} \gvac{1} \gibr \gnl
\gvac{1} \gcn{1}{1}{1}{3} \gvac{1} \gbr \gcl{7} \gnl
\gcn{1}{1}{1}{3} \gvac{1} \gbr \gcl{4} \gnl
\gvac{1} \gibr \gcl{1} \gnl
\gvac{1} \gmu \gcn{1}{1}{1}{0} \gnl
\gvac{2} \hspace{-0,34cm} \gbr \gnl
\gvac{2} \gmu \gcn{1}{1}{2}{0} \gnl
\gvac{3} \hspace{-0,34cm} \glmpt \gnot{\hspace{-0,32cm} \crta \ev} \grmptb \gnl
\gvac{6} \gob{1}{K}
\gend}
$$

$$\stackrel{nat.}{=}%\hspace{-0,26cm}
\scalebox{0.9}[0.9]{
\gbeg{6}{14}
\gvac{1} \got{2}{H} \got{1}{B} \got{1}{K} \gnl
\gvac{1} \gcmu \gbr \gnl
\gcn{1}{1}{3}{2} \gvac{1} \gbr \gcn{1}{1}{1}{2} \gnl
\gcmu \gcl{1} \gmp{-} \gcn{1}{8}{2}{2} \gnl
\gcl{2} \glm \gcn{1}{1}{1}{2} \gnl
\gvac{2} \grcm \gcn{1}{5}{0}{0} \gnl
\gcn{1}{1}{1}{3} \gvac{1} \gibr \gnl
\gvac{1} \gibr \gcl{1} \gnl
\gvac{1} \gmu \gcn{1}{1}{1}{0} \gnl
\gvac{2} \hspace{-0,34cm} \gbr \gnl
\gvac{2} \gcl{3} \gbr \gnl
\gvac{3} \gmu \gcn{1}{1}{1}{0} \gnl
\gvac{4} \hspace{-0,34cm} \glmpt \gnot{\hspace{-0,32cm} \crta \ev} \grmptb \gnl
\gvac{2} \gob{2}{K}
\gend}\stackrel{nat.}{\stackrel{\Phi_{H,H}}{=}}
\scalebox{0.9}[0.9]{
\gbeg{7}{16}
\gvac{2} \got{1}{H} \got{1}{B} \got{1}{K} \gnl
\gvac{2} \gcl{1} \gbr \gnl
\gvac{2} \gbr \gcn{1}{1}{1}{2} \gnl
\gvac{2} \gcn{1}{1}{1}{0} \hspace{-0,2cm} \gcmu \gcn{1}{10}{1}{1} \gnl
\gvac{2} \gibr \gmp{-} \gnl
\gvac{2} \gcn{1}{1}{1}{0} \gcl{2} \gcn{1}{6}{1}{1} \gnl
\gvac{1} \gcmu \gnl
\gvac{1} \gcn{1}{1}{1}{3} \glm \gnl
\gvac{2} \gbr \gnl
\gvac{1} \gcn{1}{1}{3}{1} \gvac{1} \gcl{2} \gnl
\gvac{1} \grcm \gnl
\gvac{1} \gcl{4} \gmu \gcn{1}{1}{1}{0} \gnl
\gvac{3} \hspace{-0,34cm} \gbr \gnl
\gvac{3} \gmu \gcn{1}{1}{2}{0} \gnl
\gvac{4} \hspace{-0,34cm} \glmpt \gnot{\hspace{-0,32cm} \crta \ev} \grmptb \gnl
\gvac{2} \gob{1}{K}
\gend} \stackrel{\Phi_{H,K}}{\stackrel{YD}{\stackrel{nat.}{=}}}
\scalebox{0.9}[0.9]{
\gbeg{7}{11}
\gvac{2} \got{1}{H} \gvac{1} \got{1}{B} \got{1}{K} \gnl
\gvac{1} \gwcm{3} \gbr \gnl
\gvac{1} \gcl{1} \gvac{1} \gbr \gcn{1}{6}{1}{1} \gnl
\gvac{1} \gcn{1}{1}{1}{0} \gcn{1}{1}{3}{1} \gvac{1} \gmp{-} \gnl
\gcmu \grcm \gcn{1}{1}{1}{0} \gnl
\gcl{1} \gbr \gcl{1} \gcn{1}{2}{0}{0}\gnl
\glm \gmu \gnl
\gvac{1} \gcl{3} \gvac{1} \hspace{-0,34cm} \gbr \gnl
\gvac{3} \gmu \gcn{1}{1}{2}{0} \gnl
\gvac{4} \hspace{-0,34cm} \glmpt \gnot{\hspace{-0,32cm} \crta \ev} \grmptb \gnl
\gvac{2} \gob{1}{K}
\gend} \stackrel{coass.}{\stackrel{ass.}{=}}
\scalebox{0.9}[0.9]{
\gbeg{7}{13}
\gvac{2} \got{1}{H} \gvac{2} \got{1}{B} \got{1}{K} \gnl
\gvac{1} \gwcm{3} \gvac{1} \gbr \gnl
\gvac{1} \gcn{1}{1}{1}{0} \gvac{1} \hspace{-0,34cm} \gcmu \gcn{1}{1}{2}{1} \gcn{1}{8}{2}{2} \gnl
\gvac{1} \gcl{5} \gcn{1}{1}{3}{1} \gvac{1} \gmp{-} \gcl{1} \gnl
\gvac{2} \gcl{3} \gvac{1} \gbr \gnl
\gvac{3} \gcn{1}{1}{3}{1} \gvac{1} \gcl{2} \gnl
\gvac{3} \grcm \gnl
\gvac{2} \gbr \gbr \gnl
\gvac{1} \glm \gbr \gcl{1} \gnl
\gvac{2} \gcl{3} \gmu \gcn{1}{1}{1}{0} \gnl
\gvac{4} \hspace{-0,34cm} \gmu \gcn{1}{1}{3}{0} \gnl
\gvac{5} \hspace{-0,34cm} \glmpt \gnot{\hspace{-0,32cm} \crta \ev} \grmptb \gnl
\gvac{3} \gob{1}{K}
\gend}
$$

$$\stackrel{\Phi_{H,H}}{\stackrel{(\ref{antipode2})}{=}}
\scalebox{0.9}[0.9]{
\gbeg{4}{6}
\got{1}{H} \gvac{1} \got{1}{B} \got{1}{K} \gnl
\gcl{3} \gvac{1} \gbr \gnl
\gvac{1} \gcn{1}{1}{3}{1} \gvac{1} \gcl{1} \gnl
\gvac{1} \grcm \gcl{1} \gnl
\glm \glmpt \gnot{\hspace{-0,32cm} \crta \ev} \grmptb \gnl
\gvac{1} \gob{1}{K}
\gend} \stackrel{\equref{crta ev}}{\stackrel{\Phi_{H,B}}{\stackrel{nat.}{=}}}
\scalebox{0.9}[0.9]{
\gbeg{4}{5}
\got{1}{H} \got{1}{B} \got{1}{K} \gnl
\gcl{2} \gcl{1} \grcm \gnl
\gvac{1} \gbr \gcl{1} \gnl
\glm \gev \gnl
\gvac{1} \gob{1}{K}
\gend}\stackrel{\equref{BHmod-conv}}{\stackrel{\G}{=}}
\scalebox{0.9}[0.9]{
\gbeg{4}{5}
\got{1}{H} \got{1}{B} \got{1}{\hspace{0,24cm}\G(K)} \gnl
\gcl{1} \glm \gnl
\gcn{1}{1}{1}{3} \gvac{1} \gcl{1} \gnl
\gvac{1} \glm \gnl
\gvac{2} \gob{1}{\G(K)}
\gend}
$$
By \equref{B tie H -mod} and \equref{D(H)-action} this proves that $\G(K)$ is a module over $B\bowtie H$.
From \prref{Hmod-H*comod} we then know that $\F$ and $\G$ make an isomorphism of categories.
Let us show that $\F$ is a monoidal functor. Take $M, N\in {}_{(H^{op})^*\bowtie H}\C$, then:
$$\hspace{0,24cm}\scalebox{0.9}[0.9]{
\gbeg{3}{6}
\got{1}{\F(M\ot N)} \gnl
\gcl{1} \gnl
\grcm \gnl
\gcl{2} \gcn{1}{1}{1}{3} \gnl
\gvac{2} \gcl{1} \gnl
\gob{1}{\F(M\ot N)} \gob{3}{H}
\gend}\hspace{-0,14cm}=\hspace{-0,34cm}
\scalebox{0.9}[0.9]{
\gbeg{5}{9}
\gvac{3} \got{2}{M} \got{1}{\hspace{-0,14cm}N} \gnl
\gwdb{3} \gcn{1}{3}{2}{2} \gcn{1}{4}{2}{2} \gnl
\gcl{2} \gvac{1} \hspace{-0,34cm} \gcmu \gnl
\gvac{2} \gibr \gnl
\gcn{1}{1}{2}{3} \gvac{1} \gcn{1}{1}{1}{1} \gbr \gnl
\gvac{1} \gcn{1}{1}{1}{3} \glm \glm \gnl
\gvac{2} \gbr \gcn{1}{1}{3}{1} \gnl
\gvac{2} \gcl{1} \gbr \gnl
\gvac{2} \gob{1}{M} \gob{1}{N} \gob{1}{H} \gend}
=%\stackrel{\equref{YDDH-monoidal}}{=} %\hspace{-0,24cm}
\scalebox{0.84}[0.84]{
\gbeg{6}{8}
\gvac{4} \got{1}{M} \got{1}{N} \gnl
\gdb \gdb \gcl{1} \gcl{2} \gnl
\gcl{1} \gibr \gbr \gnl
\gmu \glm \glm \gnl
\gcn{1}{1}{2}{3} \gvac{1} \gcn{1}{1}{3}{1} \gvac{1} \gcn{1}{1}{3}{1} \gnl
\gvac{1} \gbr \gcn{1}{1}{3}{1} \gnl
\gvac{1} \gcl{1} \gbr \gnl
\gvac{1} \gob{1}{M} \gob{1}{N} \gob{1}{H} \gend}
\stackrel{nat.}{=} \hspace{-0,4cm}
\scalebox{0.9}[0.9]{
\gbeg{6}{10}
\gvac{2} \got{1}{M} \gvac{2} \got{1}{N} \gnl
\gdb \gcl{1} \gdb \gcl{1} \gnl
\gcl{1} \glm \gcl{1} \glm \gnl
\gcn{1}{1}{1}{3} \gvac{1} \gibr \gvac{1} \gcl{3} \gnl
\gvac{1} \gbr \gcl{3} \gnl
\gvac{1} \gmu \gnl
\gvac{1} \gcn{1}{1}{2}{3} \gcn{1}{2}{7}{5} \gnl
\gvac{2} \gbr \gnl
\gvac{2} \gcl{1} \gbr \gnl
\gvac{2} \gob{1}{M} \gob{1}{N} \gob{1}{H}
\gend}\stackrel{nat.}{=} \hspace{-0,54cm}
\scalebox{0.9}[0.9]{
\gbeg{6}{9}
\gvac{2} \got{1}{M} \gvac{2} \got{1}{N} \gnl
\gdb \gcl{1} \gdb \gcl{1} \gnl
\gcn{1}{1}{1}{3} \glm \gcn{1}{1}{1}{3} \glm \gnl
\gvac{1} \gbr \gvac{1} \gbr \gnl
\gvac{1} \gcl{4} \gcn{1}{1}{1}{3} \gvac{1} \gcl{1} \gcl{2} \gnl
\gvac{3} \gbr \gnl
\gvac{3} \gcl{2} \gbr \gnl
\gvac{4} \gmu \gnl
\gvac{1} \gob{1}{M} \gvac{1} \gob{1}{N} \gob{2}{H}
\gend}= %\hspace{-0,24cm}
\scalebox{0.9}[0.9]{
\gbeg{5}{6}
\got{1}{\F(M)} \gvac{1} \got{1}{\F(N)} \gnl
\gcl{1} \gvac{1} \gcl{1} \gnl
\grcm \grcm \gnl
\gcl{2} \gbr \gcl{1} \gnl
\gvac{1} \gcl{1} \gmu \gnl
\gob{1}{\hspace{-0,34cm}\F(M)} \gob{2}{\F(N)} \gob{1}{\hspace{0,1cm}H^{op}}
\gend}
$$
Finally, for $M, N\in {}_{(H^{op})^*\bowtie H}\C$ consider:
$$\Psi_{M, N}:=
\scalebox{0.9}[0.9]{
\gbeg{5}{9}
\gvac{4} \got{1}{M} \got{1}{N} \gnl
\gu{1} \gvac{2} \gu{1} \gcl{3} \gcl{4} \gnl
\gcl{1} \gdb \gcl{1} \gnl
\glmpt \gnot{\hspace{-0,3cm}D(H)} \grmptb
 \glmpt \gnot{\hspace{-0,3cm}D(H)} \grmptb \gnl
\gvac{1} \gcn{1}{1}{1}{3} \gvac{1} \gbr \gnl
\gvac{2} \glm \glm \gnl
\gvac{2} \gcn{1}{1}{3}{4} \gcn{1}{1}{5}{4} \gnl
\gvac{4} \hspace{-0,34cm} \gbr \gnl
\gvac{4} \gob{1}{N} \gob{2}{\hspace{-0,2cm}M}
\gend} \hspace{0,2cm}= \hspace{-0,2cm}
\scalebox{0.9}[0.9]{
\gbeg{4}{7}
\got{1}{} \got{1}{} \got{1}{M} \got{1}{N} \gnl
\gdb \gcl{1} \gcl{2} \gnl
\gcl{1} \gbr \gnl
\glm \glm \gnl
\gcn{1}{1}{3}{4} \gvac{1} \gcn{1}{1}{3}{2} \gnl
\gvac{2} \hspace{-0,34cm} \gbr \gnl
\gvac{2} \gob{1}{N} \gob{1}{M}
\gend}\stackrel{nat.}{=}
\scalebox{0.9}[0.9]{
\gbeg{4}{9}
\got{1}{} \got{1}{} \got{1}{M} \got{1}{N} \gnl
\gdb \gcl{1} \gcl{2} \gnl
\gcl{4} \gbr \gnl
\gvac{1} \gcl{1} \glm \gnl
\gvac{1} \gcl{1} \gcn{1}{1}{3}{1} \gnl
\gvac{1} \gbr \gnl
\gbr \gcl{1} \gnl
\gcl{1} \glm \gnl
\gob{1}{N} \gvac{1} \gob{1}{M}
\gend}\stackrel{nat.}{=}
\scalebox{0.9}[0.9]{
\gbeg{4}{10}
\got{1}{} \got{1}{} \got{1}{M} \got{1}{N} \gnl
\gdb \gcl{1} \gcl{2} \gnl
\gcl{1} \gbr \gnl
\gbr \glm \gnl
\gcl{2} \gcl{1} \gcn{1}{1}{3}{1} \gnl
\gvac{1} \gbr \gnl
\gbr \gcl{1} \gnl
\gcl{1} \gibr \gnl
\gcl{1} \glm \gnl
\gob{1}{N} \gvac{1} \gob{1}{M}
\gend}\stackrel{\F}{=}
\scalebox{0.9}[0.9]{
\gbeg{3}{6}
\got{1}{M} \got{1}{N} \gnl
\gcl{1} \grcm \gnl
\gbr \gcl{1} \gnl
\gcl{1} \gibr \gnl
\gcl{1} \glm \gnl
\gob{1}{N} \gvac{1} \gob{1}{M}
\gend}\stackrel{\Phi_{H,M}}{\stackrel{nat.}{=}}
\scalebox{0.9}[0.9]{
\gbeg{4}{6}
\got{1}{M} \got{1}{N} \gnl
\gbr \gnl
\gcn{1}{1}{1}{0} \gcn{1}{1}{1}{2} \gnl
\hspace{-0,34cm} \grcm \gcn{1}{1}{1}{1} \gnl
\gcn{1}{1}{1}{1} \glm \gnl
\gob{1}{N} \gvac{1} \gob{1}{M.}
\gend}
$$
Note that the right hand-side is $\Phi^{1+}_{M, N}$. Then we have that $\Psi$ becomes the braiding in
${}_{(H^{op})^*\bowtie H}\C$. Its inverse is given by: \vspace{-0,24cm}
$$\Psi_{M, N}^{-1}
\scalebox{0.9}[0.9]{
\gbeg{4}{6}
\got{1}{} \got{1}{} \got{1}{N} \got{1}{M} \gnl
\gdb \gcl{1} \gcl{1} \gnl
\gcl{1} \gmp{-} \gbr \gnl
\gcl{1} \gbr \gcl{1} \gnl
\glm \glm \gnl
\gvac{1} \gob{1}{M} \gvac{1} \gob{1}{N.}
\gend}
$$

\begin{prop} \prlabel{YD-DH}
Assume $H\in\C$ is a finite Hopf algebra with a bijective antipode. Suppose that $\Phi_{H, M}$ is symmetric
for all $M\in {}_H\YD(\C)^{H^{op}}$. The categories ${}_H\YD(\C)^{H^{op}}$ and ${}_{D(H)}\C$ are isomorphic
as braided monoidal categories.
\end{prop}

\vspace{0,2cm}

In \cite[Definition 1.2]{Maj2} Majid defined an ``opposite comultiplication'' $\Delta^{op}$
for a bialgebra $H$. Let $\Oo(H, \Delta^{op})$ denote the subcategory of those
$H$-modules with respect to which $\Delta^{op}$ is an opposite comultiplication. If
$\R:I\to H\ot H$ is a quasitriangular structure for $H$, \cite[Definition 1.3]{Maj2},
then by \cite[Proposition 3.2]{Maj2} the subcategory $\Oo(H, \Delta^{op})$ is braided by
%$$M\ot N
$$ \scalebox{0.88}[0.88]{
\bfig
%\putmorphism(-2300,0)(1,0)[M\ot N`M\ot N\ot H\ot H`M\ot N\ot\R]{1100}1a
%\putmorphism(-1200,0)(1,0)[\phantom{M\ot N\ot H\ot H}`M\ot H\ot N\ot H`
% M\ot\Phi_{H, N}\ot H]{1400}1a
%\putmorphism(100,0)(1,0)[\phantom{\C(A\ot M\ot X\ot B, N)}`M\ot N`
% \nu_M\ot\nu_N]{1000}1a
%\putmorphism(1100,0)(1,0)[\phantom{M\ot N)}`N\ot M.`\Phi_{M,N}]{600}1a
\putmorphism(-2300,0)(1,0)[M\ot N`H\ot H\ot M\ot N`\R\ot M\ot N]{1100}1a
\putmorphism(-1200,0)(1,0)[\phantom{M\ot N\ot H\ot H}`H\ot M\ot H\ot N`
 H\ot\Phi_{H, M}\ot N]{1400}1a
\putmorphism(100,0)(1,0)[\phantom{\C(A\ot M\ot X\ot B, N)}`M\ot N`
 \mu_M\ot\mu_N]{1000}1a
\putmorphism(1100,0)(1,0)[\phantom{M\ot N)}`N\ot M.`\Phi_{M,N}]{600}1a
\efig}
$$
We denote this composition by $\Phi(\R)$. It is straightforward to check that if $\Phi_{H, M}$ is
symmetric for all $H$-modules $M$ in $\C$, then $\Delta^{op}:=\Phi_{H, H}\Delta_H$ is an opposite
comultiplication for $H$ with respect to the whole category ${}_H\C$, i.e. $\Oo(H, \Delta^{op})
={}_H\C$. The same is true for $\C_H$. In particular, the above holds for $D(H)$. The morphism:
$$\R:=
\scalebox{0.9}[0.9]{
\gbeg{3}{5}
\got{1}{} \gnl
\gu{1} \gvac{2} \gu{1} \gnl
\gcl{1} \gdb \gcl{2} \gnl
\gcl{1} \gcl{1} \gcl{1} \gnl
\gob{1}{\hspace{-0,1cm}H^*} \gob{1}{H} \gob{1}{H^*} \gob{1}{\hspace{-0,12cm}H}
\gend}
$$
is a quasitriangular structure for $D(H)$ and it induces a braiding $\Phi(\R)$ on ${}_{D(H)}\C$. Note
that it equals to our $\Psi$ from above. As it is the case in the category of modules over a commutative
ring and a usual quasitriangular Hopf algebra, the axioms of quasitriangularity of $D(H)$ are
equivalent to the two braiding axioms for $\Psi$, its left $D(H)$-linearity and invertibility,
{\em if %$\Phi_{H, H^*}$ and
$\Phi_{H, M}$ is symmetric for all $M\in {}_{D(H)}\C$}.
\par\medskip

\subsection{Bosonization and an isomorphism of categories}

Bespalov proved in \cite[Lemma 5.3.1 and Section 5.4]{Besp} that
a left (right) module over a quasitriangular bialgebra $(H, \R)$ can be equipped with a left
(right) comodule structure over $H$ so that the subcategory $\Oo(H, \Delta^{op})$ becomes a full
braided subcategory of ${}_H^H\YD(\C)$ ($\YD(\C)_H^H$). This is a braided version of the classical
result from \cite{Maj}.

Assume that $H$ is a quasitriangular Hopf algebra with respect to the whole category $\C_H$ (e.g.
if $\Phi_{H, M}$ is symmetric for all $M\in\C_H$). Then $\C_H$ is braided. Let
$B$ be a Hopf algebra in $\C_H$. Equipped with a right $H$-comodule structure:
$$\rho_B=
\scalebox{0.9}[0.9]{
\gbeg{3}{5}
\got{1}{B} \gnl
\gcl{1} \gvac{1}  \glmpb \gnot{\hspace{-0,4cm}\R} \grmpb \gnl
\gcl{1} \gcn{1}{1}{3}{1} \gvac{1} \gcn{1}{1}{1}{3} \gnl
\grm \gvac{2} \gcl{1} \gnl
\gob{1}{B} \gvac{3} \gob{1}{H}
\gend}
$$
$B$ becomes a right-right YD-module. The structure morphisms of $B$ are right $H$-linear. Since
$H$ is quasitriangular, they turn out to be also right $H$-colinear. We show this for the multiplication:
$$\hspace{-0,2cm} \scalebox{0.9}[0.9]{
\gbeg{2}{4}
\got{1}{B} \got{1}{B} \gnl
\gmu \gnl
\gvac{1}  \hspace{-0,2cm} \grcm \gnl
\gvac{1}  \gob{1}{B} \gob{1}{H}
\gend}= \hspace{-0,2cm}
\scalebox{0.9}[0.9]{
\gbeg{5}{6}
\got{1}{B} \gvac{1} \got{1}{B} \gnl
\gwmu{3} \gnl
\gvac{1} \gcl{1} \gvac{1} \glmpb \gnot{\hspace{-0,4cm}\R} \grmpb \gnl
\gvac{1} \gcl{1} \gcn{1}{1}{3}{1} \gvac{1} \gcn{1}{1}{1}{3} \gnl
\gvac{1} \grm \gvac{2} \gcl{1} \gnl
\gvac{1} \gob{1}{B} \gvac{2} \gob{3}{H}
\gend}\hspace{0,2cm}= \hspace{-0,2cm}
\scalebox{0.9}[0.9]{
\gbeg{5}{8}
\got{1}{B} \got{1}{B} \gnl
\gcl{4} \gcl{3} \gvac{1} \glmpb \gnot{\hspace{-0,4cm}\R} \grmpb \gnl
\gvac{2} \gcn{1}{1}{3}{2} \gvac{1} \gcn{1}{1}{1}{3} \gnl
\gvac{2} \gcmu \gvac{1} \gcl{4} \gnl
\gvac{1} \gbr \gcl{1} \gnl
\grm \grm \gnl
\gwmu{3} \gnl
\gvac{1} \gob{1}{B} \gvac{2} \gob{3}{H}
\gend} \hspace{0,2cm} \stackrel{\star}{=} \hspace{-0,2cm}
\scalebox{0.9}[0.9]{
\gbeg{5}{9}
\got{1}{B} \got{1}{B} \gnl
\gcl{5} \gcl{4} \gvac{1} \glmpb \gnot{\hspace{-0,4cm}\R} \grmpb \gnl
\gvac{2} \gcn{1}{1}{3}{1} \gvac{1} \gcn{1}{1}{1}{3} \gnl
%\gvac{2} \gcn{1}{1}{3}{1} \gbmp{\hspace{-0,4cm}\R} \gcn{1}{1}{1}{3} \gnl
\gvac{2} \gcl{1} \glmpb \gnot{\hspace{-0,4cm}\R} \grmpb \gcl{1} \gnl
\gvac{2} \gibr \gmu \gnl
\gvac{1} \gbr \gcl{1} \gcn{1}{3}{2}{2} \gnl
\grm \grm \gnl
\gwmu{3} \gnl
\gvac{1} \gob{1}{B} \gvac{2} \gob{2}{H}
\gend} \hspace{0,2cm} \stackrel{\Phi_{H,B}}{\stackrel{nat.}{=}} \hspace{-0,2cm}
\scalebox{0.9}[0.9]{
\gbeg{8}{7}
\got{1}{B} \gvac{3} \got{1}{B} \gnl
\gcl{2} \gvac{1}  \glmpb \gnot{\hspace{-0,4cm}\R} \grmpb
 \gcl{2} \gvac{1}  \glmpb \gnot{\hspace{-0,4cm}\R} \grmpb \gnl
\gcl{1} \gcn{1}{1}{3}{1} \gvac{1} \gcl{2} \gvac{1} \gcn{1}{1}{3}{1} \gvac{1} \gcl{3} \gnl
\grm \gvac{2} \grm \gnl
\gcl{1} \gvac{2} \gbr \gnl
\gwmu{4} \gwmu{4} \gnl
\gvac{1} \gob{2}{B} \gvac{2} \gob{2}{H}
\gend} \hspace{0,2cm} = \hspace{-0,2cm}
\scalebox{0.9}[0.9]{
\gbeg{4}{5}
\got{1}{B} \gvac{1} \got{1}{B} \gnl
\grcm  \grcm  \gnl
\gcl{1} \gbr \gcl{1} \gnl
\gmu \gmu \gnl
\gob{2}{B} \gob{2}{H}
\gend}
$$
where at the place $\star$ we applied the quasitriangular axiom $(\Delta^{op}\ot H)\R=\R_{23}\R_{13}$.
Since $\C_H$ is a braided subcategory of $\YD(\C)^H_H$, we have that the braiding in $\C_H$ induced by $\R$
(the right hand-side version of $\Phi(\R)$) equals $\Phi^R$ from \equref{braid-LR} and $B$ is indeed a Hopf
algebra in $\YD(\C)^H_H$. By \cite[Theorem 4.1.2]{Besp} the cross product algebra $H\ltimes B$ is then a Hopf
algebra in
$\C$, the bosonization of the braided Hopf algebra $B$. Its multiplication and comultiplication are given by:
$$
\scalebox{0.86}{
\gbeg{4}{6}
\got{1}{H} \got{1}{B} \got{2}{H} \got{1}{B} \gnl
\gcl{1} \gcl{1} \gcmu \gcl{1} \gnl
\gcl{1} \gbr \gcl{1} \gcl{1} \gnl
\gcl{1} \gcl{1} \grm \gcl{1} \gnl
\gmu \gwmu{3} \gnl
\gob{2}{H} \gob{3}{B}
\gend}\qquad\textnormal{and}\quad
\scalebox{0.86}{
\gbeg{4}{6}
\got{2}{H} \got{3}{B} \gnl
\gcmu \gwcm{3} \gnl
\gcl{3} \gcl{1} \grcm \gcl{3} \gnl
\gvac{1} \gbr \gcl{1} \gnl
\gvac{1} \gcl{1} \gmu \gnl
\gob{1}{H} \gob{1}{B} \gob{2}{H} \gob{1}{B}
\gend}
$$
which are the tensor product algebra and coalgebra respectively in the category $\YD(\C)^H_H$.
The antipode of $H\ltimes B$ is given by $S_{H\ltimes B}:=
\Phi^R_{B,H}(S_B\ot S_H)\Phi^R_{H,B}$. %, where is the braiding of $\YD(\C)^H_H$.
Similarly as in \leref{matched pair mod} one proves that the categories $\C_{H\ltimes B}$
and $(\C_H)_B$ are isomorphic. An object of the latter category is a right
$H$- and a right $B$-module $M$ satisfying the compatibility condition:
$$
\scalebox{0.86}{
\gbeg{4}{6}
\got{1}{M} \gvac{1} \got{1}{B} \got{1}{H} \gnl
\gcl{1} \gcn{1}{1}{3}{1} \gvac{1} \gcl{1} \gnl
\grm \gcn{1}{2}{3}{-1} \gnl
\gcl{1} \gnl
\grm \gnl
\gob{1}{M}
\gend}=
\scalebox{0.86}{
\gbeg{4}{7}
\got{1}{M} \got{1}{B} \got{2}{H} \gnl
\gcl{1} \gcl{1} \gcmu \gnl
\gcl{1} \gbr \gcl{1} \gnl
\grm \grm \gnl
\gcl{1} \gcn{1}{1}{3}{1} \gnl
\grm \gnl
\gob{1}{M}
\gend}
$$
Moreover, the isomorphism $\F: (\C_H)_B \to \C_{H\ltimes B}$ is monoidal, since for
$M, N\in (\C_H)_B$ it is:
$$\hspace{0,24cm}\scalebox{0.9}[0.9]{
\gbeg{3}{6}
\got{1}{\hspace{-0,34cm}\F(M\ot N)} \gvac{1} \got{1}{H\ltimes B} \gnl
\gcl{1} \gvac{1} \gcl{1} \gnl
\gcl{1} \gcn{1}{1}{3}{1} \gnl
\grm \gnl
\gcl{1} \gnl
\gob{1}{\F(M\ot N)}
\gend}\hspace{-0,14cm}=\hspace{-0,34cm}
\scalebox{0.9}[0.9]{
\gbeg{5}{8}
\got{1}{M} \got{1}{N} \got{2}{H} \got{1}{B} \gnl
\gcl{3} \gcl{1} \gcmu \gcl{2} \gnl
\gcl{2} \gbr \gcl{1} \gnl
\grm \grm \gcn{1}{1}{1}{0}\gnl
\gcn{1}{2}{1}{3} \gvac{1} \gcn{1}{1}{1}{1} \gcmu \gnl
\gvac{2} \glmptb \gnot{\hspace{-0,4cm}\Phi(\R)} \grmptb \gcl{1} \gnl
\gvac{1} \grm \grm \gnl
\gvac{1} \gob{1}{M} \gob{3}{N} \gend} %\stackrel{\equref{other4categories}}{=}
=\hspace{-0,34cm}
\scalebox{0.9}[0.9]{
\gbeg{5}{8}
\got{1}{M} \got{1}{N} \got{2}{H} \got{1}{B} \gnl
\gcl{3} \gcl{1} \gcmu \gcl{2} \gnl
\gcl{2} \gbr \gcl{1} \gnl
\grm \grm \gcn{1}{1}{1}{0}\gnl
\gcn{1}{2}{1}{3} \gvac{1} \gcn{1}{1}{1}{1} \gcmu \gnl
\gvac{2} \glmptb \gnot{\hspace{-0,3cm}\Phi^R} \grmptb \gcl{1} \gnl
\gvac{1} \grm \grm \gnl
\gvac{1} \gob{1}{M} \gob{3}{N} \gend} = \hspace{-0,34cm}
\scalebox{0.9}[0.9]{
\gbeg{7}{8}
\got{1}{M} \got{1}{N} \got{2}{H} \got{3}{B} \gnl
\gcl{3} \gcl{1} \gcmu \gwcm{3} \gnl
\gcl{2} \gbr \gcl{1} \grcm \gcl{3} \gnl
\grm \grm \gcn{1}{1}{1}{-1} \gcl{1} \gnl
\gcn{1}{1}{1}{3} \gvac{1} \gbr \gcn{1}{1}{3}{1} \gnl
\gvac{1} \grm \grm \gcn{2}{1}{3}{-1} \gnl
\gvac{1} \gcl{1} \gvac{1} \grm \gnl
\gvac{1} \gob{1}{M} \gob{3}{N} \gend}
=  \hspace{-0,34cm} \scalebox{0.9}[0.9]{
\gbeg{7}{9}
\got{1}{M} \got{1}{N} \got{2}{H} \got{3}{B} \gnl
\gcl{3} \gcl{1} \gcmu \gwcm{3} \gnl
\gcl{2} \gcl{1} \gcl{1} \gcl{1} \grcm \gcl{3} \gnl
\gvac{1} \gbr \gbr \gcl{1} \gnl
\grm \gbr \gmu \gnl
\gcl{1} \gcn{1}{1}{3}{1} \gvac{1} \gcl{1} \gcn{1}{1}{2}{1} \gcn{2}{2}{3}{-1} \gnl
\grm \gvac{1} \grm \gnl
\gcl{1} \gvac{2} \grm \gnl
\gob{1}{M} \gvac{1} \gob{3}{N} \gend}
= %\hspace{-0,34cm}
\scalebox{0.9}[0.9]{
\gbeg{5}{6}
\got{1}{\F(M)\ot\F(N)} \gvac{2} \got{1}{H\ltimes B} \gnl
\gcl{1} \gvac{1} \gcn{1}{1}{3}{1} \gnl
\gcl{1} \gcn{1}{1}{3}{1} \gnl
\grm \gnl
\gcl{1} \gnl
\gob{1}{\F(M)\ot\F(N).}
\gend}
$$
Thus we have proved:

\begin{prop} \prlabel{cross prod isom}
Let $H$ be a quasitriangular Hopf algebra such that $\Phi_{H, M}$ is symmetric
for all $M\in\C_H$. Let $B$ be a Hopf algebra in $\C_H$. Then $H\ltimes B$ is a Hopf algebra in $\C$
and there is a monoidal isomorphism of categories $\C_{H\ltimes B}\iso(\C_H)_B$.
\end{prop}

\section{Other versions of Yetter-Drinfel'd categories} \selabel{other-cases}
\setcounter{equation}{0}

We start this section by giving equivalent conditions for the left-left and the right-right
YD-compatibility conditions and relating the corresponding categories with that of modules over
the Drinfel'd double. Subsequently, we will study two versions of left-right, as well as two
versions of right-left YD-categories. At the end we will relate all the categories we have studied.

\subsection{Left-left and right-right YD-modules as modules over the Drinfel'd double} \sslabel{left-left}

At the beginning of \seref{YD-mods} we noted that the categories ${}_H ^H\YD(\C)$ of left-left YD-modules
and $\YD(\C)_H^H$, of right-right YD-modules, are braided monoidal categories without any further conditions.
However, in order to prove that these categories are isomorphic to that of left (respectively right) $D(H)$-modules
in $\C$ for a finite Hopf algebra $H$ with a bijective antipode, one has to require the
same symmetricity conditions on the braiding as in \prref{YD-DH}. Before supporting this claim, we
note that the expressions \equref{left YD} and \equref{right YD} are equivalent to: \vspace{-0,9cm}
\begin{center}
\begin{tabular}{p{6.4cm}p{0cm}p{8.4cm}}
\begin{equation} \eqlabel{left YD-oth} \hspace{-0,2cm}
%\textcolor{rojo}{
%\gbeg{4}{8}
%\got{2}{H} \got{1}{N} \gnl
%\gcmu \gcl{2} \gnl
%\gbr \gnl
%\gcl{1} \glm \gnl
%\gcl{1} \glcm \gnl
%\gibr \gcl{1} \gnl
%\gmu \gcl{1} \gnl
%\gob{2}{H} \gob{1}{N}
%\gend}=
\gbeg{4}{8}
\got{1}{H} \got{5}{N} \gnl
\gcn{2}{2}{1}{5} \gvac{1} \gcl{2} \gnl \gnl
\gvac{2} \glm \gnl
\gvac{2} \glcm \gnl
\gcn{2}{2}{5}{1} \gvac{1} \gcl{2} \gnl \gnl
\gob{1}{H} \gob{5}{N}
\gend=
\gbeg{6}{9}
\gvac{1} \got{2}{H} \got{4}{N} \gnl
\gvac{1} \gcmu \gcn{1}{1}{4}{4} \gnl
\gvac{1} \gcn{1}{1}{1}{0} \hspace{-0,21cm} \gcmu \glcm \gnl
\gvac{1} \gcl{3} \gbr \gcl{1} \gcl{3} \gnl
\gvac{2}  \gmp{+} \gbr \gnl
\gvac{1} \gvac{1} \gbr \gcl{1} \gnl
\gvac{1} \gcn{1}{1}{1}{2} \gmu \glm \gnl
\gvac{2} \hspace{-0,21cm} \gmu \gcn{1}{1}{4}{4} \gnl
\gvac{2} \gob{2}{H} \gob{4}{N}
\gend
\quad
\end{equation}
& &
\begin{equation} \eqlabel{right YD-oth} \textnormal{and}\qquad\quad
\gbeg{4}{8}
\got{1}{L} \got{5}{H} \gnl
\gcl{2} \gcn{2}{2}{5}{1} \gnl \gnl
\grm \gnl
\grcm \gnl
\gcl{2} \gcn{2}{2}{1}{5} \gnl \gnl
\gob{1}{L} \gob{5}{H}
\gend=
\gbeg{6}{9}
\got{1}{L} \got{5}{H} \gnl
\gcl{1} \gvac{2} \hspace{-0,34cm} \gcmu \gnl
\gvac{1} \hspace{-0,22cm} \grcm \gcmu \gcn{1}{1}{0}{1} \gnl
\gvac{1} \gcl{3} \gcl{1} \gbr \gcl{3} \gnl
\gvac{2} \gbr \gmp{+} \gnl
\gvac{2} \gcl{1} \gbr \gnl
\gvac{1} \grm \gmu \gcn{1}{1}{1}{0}\gnl
\gvac{1} \gcl{1} \gvac{2} \hspace{-0,34cm} \gmu \gnl
\gvac{1} \gob{2}{L} \gob{4}{H}
\gend
\end{equation}
\end{tabular}
\end{center}
respectively, if $\Phi_{H, N}$ ($\Phi_{H, L}$) is symmetric for
$N\in {}_H ^H\YD(\C)$ and $L\in \YD(\C)_H^H$. The same symmetricity conditions are necessary
to prove that ${}_H ^H\YD(\C)$ and $\YD(\C)_H^H$, characterized by \equref{left YD-oth} and
\equref{right YD-oth} respectively, are monoidal categories.

Consider the functors 	$
\bfig
\putmorphism(0,30)(1,0)[\F_l: {}_{D(H)}\C` {}_H ^H\YD(\C): \G_l`]{700}1a
\putmorphism(-40,-10)(1,0)[\phantom{{}_H ^H\YD(\C): \G_l}`\phantom{\F_l: {}_{D(H)}\C}` ]{660}{-1}b
\efig
$
defined by
$$
\scalebox{0.9}[0.9]{
\gbeg{3}{4}
\gvac{1} \got{3}{\F_l(M)} \gnl
\gvac{1} \glcm \gnl
\gcn{1}{1}{3}{1} \gvac{1} \gcl{1} \gnl
\gob{1}{H} \gob{3}{\F_l(M)}
\gend} = \scalebox{0.9}[0.9]{
\gbeg{7}{5}
\got{1}{} \got{1}{} \got{1}{M} \gnl
\gdb \gcl{2} \gnl
\gcl{1} \gmp{+} \gnl
\gcl{1} \glm \gnl
\gob{1}{H} \gob{3}{M}
\gend}\textnormal{and}\qquad
\scalebox{0.9}[0.9]{
\gbeg{3}{4}
\got{1}{H^*} \gvac{1} \got{1}{\G_l(N)} \gnl
\gcn{1}{1}{1}{3} \gvac{1} \gcl{1} \gnl
\gvac{1} \glm \gnl
\gvac{2} \gob{1}{\G_l(N)}
\gend} = \scalebox{0.9}[0.9]{
\gbeg{2}{5}
\got{1}{H^*} \got{3}{N} \gnl
\gcl{1} \glcm \gnl
\gcl{1} \gmp{-} \gcl{2} \gnl
\gev \gnl
\gob{5}{N}
\gend}
$$
for $M\in {}_{(H^{op})^*\bowtie H}\C$ and $N\in {}_H ^H\YD(\C)$, where $\F_l(M)$ is a left
$H$-module by the action of $\eta_B\ot H$ on $M$, and $\G_l(N)=N$ as a left $H$-module.
%The two functors act as identities on morphisms.
Even though one uses \equref{left YD} as the
defining relation for the category ${}_H ^H\YD(\C)$, one has that $\F_l$ and $\G_l$ define
an isomorphism of categories {\em if $\Phi_{H, N}$ is symmetric}.
We show here only that this is a monoidal isomorphism. Observe first: \vspace{-0,34cm}
\begin{equation}\eqlabel{YDDH-monoidal}
\scalebox{0.84}[0.84]{
\gbeg{4}{5}
\got{1}{} \gnl
\gwdb{3} \gnl
\gcl{2} \gvac{1} \hspace{-0,34cm} \gcmu \gnl
\gvac{2} \gcn{1}{1}{1}{1} \gcn{1}{1}{1}{1} \gnl
\gob{2}{H} \gob{1}{B} \gob{1}{B} \gend} =
\scalebox{0.84}[0.84]{
\gbeg{8}{6}
\got{2}{} \gnl
\gwdb{3} \gnl
\gcl{3} \gvac{1} \hspace{-0,34cm} \gcmu \gdb \gdb \gnl
\gvac{2} \gcl{1} \gbr \gibr \gcl{2} \gnl
\gvac{2} \gev \gev \gcl{1} \gnl
\gob{2}{H} \gvac{4} \gob{2}{\hspace{-0,14cm}H^*} \gob{1}{\hspace{-0,4cm}H^*} \gend}
\stackrel{\equref{B-codiag}}{=}
\scalebox{0.84}[0.84]{
\gbeg{7}{7}
\got{1}{} \gnl
\gvac{2} \gdb \gdb \gnl
\gdb \gcl{1} \gibr \gcl{1} \gnl
\gcn{1}{3}{1}{1} \gcn{1}{1}{1}{1} \gbr \gcl{3} \gcl{3} \gnl
\gvac{1} \gcn{1}{1}{1}{2} \gmu \gnl
\gvac{2} \hspace{-0,22cm} \gev \gnl
\gob{2}{H} \gvac{3} \gob{1}{\hspace{-0,2cm}H^*} \gob{1}{H^*} \gend} =
\scalebox{0.84}[0.84]{
\gbeg{6}{6}
\got{1}{} \gnl
\gdb \gdb \gnl
\gcl{1} \gibr \gcl{1} \gnl
\gbr \gcl{2} \gcl{2} \gnl
\gmu \gnl
\gob{2}{H} \gob{1}{H^*} \gob{1}{H^*} \gend}
\end{equation}
Now for $M, N\in {}_{D(H)}\C$ we have:
$$\scalebox{0.9}[0.9]{
\gbeg{4}{6}
\got{3}{\F_l(M\ot N)} \gnl
\gvac{2} \gcl{1} \gnl
\gvac{1} \glcm \gnl
\gcn{1}{1}{3}{1} \gvac{1} \gcl{2} \gnl
\gcl{1} \gnl
\gob{1}{H} \gob{3}{\F_l(M\ot N)}
\gend}\hspace{-0,14cm}=\hspace{-0,24cm}
\scalebox{0.9}[0.9]{
\gbeg{6}{8}
\gvac{3} \got{2}{M} \got{1}{\hspace{-0,14cm}N} \gnl
\gwdb{3} \gcn{1}{3}{2}{2} \gcn{1}{4}{2}{2} \gnl
\gcl{5} \gvac{1} \gmp{+} \gnl
\gvac{2} \hspace{-0,34cm} \gcmu \gnl
\gvac{2} \gcn{1}{1}{1}{1} \gbr \gnl
\gvac{2} \glm \glm \gnl
\gvac{3} \gcl{1} \gvac{1} \gcl{1} \gnl
\gob{2}{H} \gob{3}{M} \gob{1}{N} \gend} =
\scalebox{0.9}[0.9]{
\gbeg{6}{8}
\gvac{3} \got{2}{M} \got{1}{\hspace{-0,14cm}N} \gnl
\gwdb{3} \gcn{1}{4}{2}{2} \gcn{1}{5}{2}{2} \gnl
\gcl{5} \gvac{1} \hspace{-0,34cm} \gcmu \gnl
\gvac{2} \gbr \gnl
\gvac{2} \gmp{+} \gmp{+} \gnl
\gvac{2} \gcn{1}{1}{1}{1} \gbr \gnl
\gvac{2} \glm \glm \gnl
\gob{2}{H} \gob{3}{M} \gob{1}{N} \gend}
\stackrel{\equref{YDDH-monoidal}}{=} \hspace{-0,24cm}
\scalebox{0.84}[0.84]{
\gbeg{7}{8}
\gvac{4} \got{1}{M} \got{1}{N} \gnl
\gdb \gdb \gcl{4} \gcl{5} \gnl
\gcl{1} \gibr \gcl{1} \gnl
\gbr \gbr \gnl
\gmu \gmp{+} \gmp{+} \gnl
\gcn{1}{2}{2}{2} \gvac{1} \gcl{1} \gbr \gnl
\gvac{2} \glm \glm \gnl
\gob{2}{H} \gob{3}{M} \gob{1}{N} \gend}
\stackrel{\Phi_{H, H^*}}{=} \hspace{-0,24cm}
\scalebox{0.84}[0.84]{
\gbeg{6}{8}
\gvac{4} \got{1}{M} \got{1}{N} \gnl
\gwdb{4} \gcl{3} \gcl{4} \gnl
\gcl{1} \gdb \gcl{1} \gnl
\gbr \gmp{+} \gmp{+}\gnl
\gmu \gcl{1} \gbr \gnl
\gcn{1}{2}{2}{2} \gvac{1} \glm \glm \gnl
\gvac{3} \gcl{1} \gvac{1} \gcl{1} \gnl
\gob{2}{H} \gob{3}{M} \gob{1}{N} \gend}
$$

$$
\stackrel{nat.}{=} \hspace{-0,4cm}
\scalebox{0.9}[0.9]{
\gbeg{6}{7}
\gvac{2} \got{1}{M} \gvac{2} \got{1}{N} \gnl
\gdb \gcl{1} \gdb \gcl{1} \gnl
\gcl{3} \gmp{+} \gcl{1} \gcl{1} \gmp{+} \gcl{1} \gnl
\gvac{1} \glm \gcl{1} \glm \gnl
\gvac{2} \gibr \gvac{1} \gcl{1} \gnl
\gwmu{3} \gcl{1} \gvac{1} \gcl{1} \gnl
\gvac{1} \gob{1}{H} \gvac{1} \gob{1}{M} \gvac{1} \gob{1}{N}
\gend}\stackrel{\Phi_{H, M}}{=} \hspace{-0,4cm}
\scalebox{0.9}[0.9]{
\gbeg{6}{7}
\gvac{2} \got{1}{M} \gvac{2} \got{1}{N} \gnl
\gdb \gcl{1} \gdb \gcl{1} \gnl
\gcl{3} \gmp{+} \gcl{1} \gcl{1} \gmp{+} \gcl{1} \gnl
\gvac{1} \glm \gcl{1} \glm \gnl
\gvac{2} \gbr \gvac{1} \gcl{1} \gnl
\gwmu{3} \gcl{1} \gvac{1} \gcl{1} \gnl
\gvac{1} \gob{1}{H} \gvac{1} \gob{1}{M} \gvac{1} \gob{1}{N}
\gend}= %\hspace{-0,24cm}
\scalebox{0.9}[0.9]{
\gbeg{5}{7}
\got{2}{\F_l(M)} \gvac{1} \got{1}{\F_l(N)} \gnl
\gvac{1} \gcl{1} \gvac{1} \gcl{1} \gnl
\glcm \glcm \gnl
\gcl{1} \gbr \gcl{3} \gnl
\gmu \gcl{2} \gnl
\gcn{1}{1}{2}{2} \gnl
\gob{2}{H} \gob{1}{\F_l(M)} \gob{2}{\hspace{0,34cm}\F_l(N).}
\gend}
$$

\vspace{0,2cm}

It is easily shown that the functor $\Ll: {}_H ^H\YD(\C)\to \YD(\C)^H_H$
\label{funktor L} given by
$$
\scalebox{0.9}[0.9]{
\gbeg{3}{5}
\got{1}{\hspace{0,1cm}\Ll(M)} \got{3}{H} \gnl  %\hspace{-0,24cm}
\gcl{1} \gvac{1} \gcl{1} \gnl
\gcl{1} \gcn{1}{1}{3}{1} \gnl
\grm \gnl
%\gcl{1} \gnl
\gob{1}{\Ll(M)}
\gend} = \scalebox{0.9}[0.9]{
\gbeg{2}{6}
\got{1}{M} \got{1}{H} \gnl
\gibr \gnl
\gmp{+} \gcl{2} \gnl
\gmp{+} \gnl
\glm \gnl
\gob{3}{M}
\gend}\qquad\textnormal{and}\quad
\scalebox{0.9}[0.9]{
\gbeg{2}{5}
\got{1}{\Ll(M)} \gnl
\grcm \gnl
\gcl{1} \gcn{1}{1}{1}{3} \gnl
\gcl{1} \gvac{1} \gcl{1} \gnl
\gob{1}{\Ll(M)} \gob{3}{H}
\gend} = \scalebox{0.9}[0.9]{
\gbeg{3}{6}
\got{1}{} \got{1}{M} \gnl
\glcm \gnl
\gmp{-} \gcl{2} \gnl
\gmp{-} \gnl
\gbr \gnl
\gob{1}{M} \gob{1}{H}
\gend}
$$
for $M,N\in {}_H ^H\YD(\C)$ is an isomorphism of categories. It is even monoidal if $\Phi_{H,M}$ is
symmetric for every $M\in {}_H ^H\YD(\C)$. We show only the compatibility of the $H$-module structures
on the tensor products:
\vspace{-0,2cm}
$$\hspace{0,2cm}
\scalebox{0.9}[0.9]{
\gbeg{3}{6}
\got{1}{\Ll(M\ot N)} \got{3}{H} \gnl
\gcl{2} \gvac{1} \gcl{1} \gnl
\gvac{1} \gcn{1}{1}{3}{1} \gnl
\grm \gnl
\gcl{1} \gnl
\gob{1}{\Ll(M\ot N)}
\gend}= \hspace{-0,4cm}
\scalebox{0.9}[0.9]{
\gbeg{5}{10}
\gvac{1} \got{1}{M} \got{1}{N} \got{1}{H} \gnl
\gvac{1} \gcl{1} \gibr \gnl
\gvac{1} \gibr \gcl{7} \gnl
\gvac{1} \gmp{+} \gcl{4} \gnl
\gvac{1} \gmp{+} \gnl
\gcn{1}{1}{3}{2} \gnl
\gcmu \gnl
\gcl{1} \gbr \gnl
\glm \glm \gnl
\gvac{1} \gob{1}{M} \gob{3}{N}
\gend} \stackrel{\Phi_{H,H}}{=}
\scalebox{0.9}[0.9]{
\gbeg{5}{10}
\gvac{1} \got{1}{M} \got{1}{N} \got{1}{H} \gnl
\gvac{1} \gcl{1} \gibr \gnl
\gvac{1} \gibr \gcl{7} \gnl
\gcn{1}{1}{3}{2} \gvac{1} \gcl{4} \gnl
\gcmu \gnl
\gmp{+} \gmp{+} \gnl
\gmp{+} \gmp{+} \gnl
\gcl{1} \gbr \gnl
\glm \glm \gnl
\gvac{1} \gob{1}{M} \gob{3}{N}
\gend} =
\scalebox{0.9}[0.9]{
\gbeg{5}{9}
\got{1}{M} \got{1}{N} \got{2}{H} \gnl
\gcl{2} \gcl{1} \gcmu \gnl
\gvac{1} \gibr \gcl{1} \gnl
\gibr \gibr \gnl
\gmp{+} \gcl{1} \gmp{+} \gcl{1} \gnl
\gmp{+} \gcl{1} \gmp{+} \gcl{1} \gnl
\glm \glm \gnl
\gvac{1} \gcl{1} \gvac{1} \gcl{1} \gnl
\gvac{1} \gob{1}{M} \gob{3}{N}
\gend} =
\scalebox{0.9}[0.9]{
\gbeg{5}{7}
\got{1}{\Ll(M)} \gvac{1} \got{1}{\Ll(N)} \got{2}{H} \gnl
\gcn{1}{1}{1}{3} \gvac{1} \gcl{1} \gcmu \gnl
\gvac{1} \gcl{1} \gibr \gcl{1} \gnl
\gvac{1} \grm \grm \gnl
\gvac{1} \gcl{2} \gvac{1} \gcl{2} \gnl
\gvac{1} \gob{1}{\Ll(M)} \gob{3}{\Ll(N)}
\gend} \stackrel{\Phi_{H,N}}{=}
\scalebox{0.9}[0.9]{
\gbeg{5}{7}
\got{1}{\Ll(M)\ot \Ll(N)} \gvac{2} \got{1}{H} \gnl
\gcl{1} \gvac{1} \gcn{1}{1}{3}{1} \gnl
\gcl{1} \gcn{1}{1}{3}{1} \gnl
\grm \gnl
\gcl{2} \gnl
\gob{1}{\Ll(M)\ot \Ll(N).}
\gend}
$$
However, $\Ll$ does not respect the braidings.

\subsection{Left-right and right-left YD-modules} \sslabel{mixing}

In \seref{YD-mods} we studied the categories of left-right YD-modules ${}_H\YD(\C)^{H^{op}}$
and ${}_{H^{cop}}\YD(\C)^H$, \rmref{lr-Hcop-cat}. Symmetrically, we may consider the category
${}^H\YD(\C)_H$ of right-left YD-modules. These are right $H$-modules and left $H$-comodules
which satisfy the compatibility condition \equref{YD-other-mix}. If $\Phi_{H,H}$ is symmetric,
this condition is equivalent to \equref{YD-S-other}. \vspace{-0,9cm}
\begin{center}
\begin{tabular}{p{6.4cm}p{0cm}p{8.4cm}}
\begin{eqnarray} \eqlabel{YD-other-mix}
\gbeg{3}{9}
\got{1}{M} \got{2}{H} \gnl
\gcl{1} \gcmu \gnl
\grm \gcn{1}{2}{1}{-1} \gnl
\gcl{1} \gnl
\gbr \gnl
\gcl{1} \gcn{1}{1}{1}{3} \gnl
\gcn{1}{1}{1}{1} \glcm \gnl
\gmu \gcl{1} \gnl
\gob{2}{H} \gob{1}{M}
\gend=
\gbeg{3}{7}
\gvac{1} \got{1}{M} \got{2}{H} \gnl
\gvac{1} \gcl{1} \gcn{1}{1}{2}{2} \gnl
\glcm \gcmu \gnl
\gcl{1} \gbr \gcl{1} \gnl
\gmu \grm \gnl
\gcn{1}{1}{2}{2} \gvac{1} \gcl{1} \gnl
\gob{2}{H} \gob{1}{M}
\gend
\end{eqnarray}
& &
\begin{eqnarray} \eqlabel{YD-S-other}
\gbeg{7}{9}
\gvac{3} \got{1}{M} \got{5}{H} \gnl
\gvac{3} \gcl{2} \gcn{2}{2}{5}{1}\gnl \gnl
\gvac{3} \grm \gnl
\gvac{3} \gcl{1} \gnl
\gvac{2} \glcm \gnl
\gcn{2}{2}{5}{1} \gvac{1} \gcl{2} \gnl
\gob{1}{H} \gob{5}{M}
\gend=
\gbeg{6}{10}
\gvac{1} \got{1}{M} \got{4}{H} \gnl
\gvac{1} \gcl{2} \gvac{1} \gcmu \gnl
\gvac{1} \gcl{1} \gcn{1}{1}{3}{2} \gvac{1} \gcl{1} \gnl
\glcm \gcmu \gmp{-} \gnl
\gcl{1} \gbr \gcl{1} \gcl{1} \gnl
\gmu \grm \gcn{1}{1}{1}{-1} \gnl
\gcn{1}{1}{2}{3} \gvac{1} \gbr \gnl
\gvac{1} \gbr \gcl{2} \gnl
\gvac{1} \gmu \gnl
\gvac{1} \gob{2}{H} \gob{1}{M}
\gend
\end{eqnarray}
\end{tabular}
\end{center}
The category ${}^H\YD(\C)_{H^{cop}}$ is monoidal if $\Phi_{H,N}$ is symmetric
for all $N\in {}^H\YD(\C)_{H^{cop}}$. This is a braided monoidal category with braiding:
$$
\Phi^{3+}_{M, N}=
\gbeg{4}{6}
\got{1}{M} \got{3}{N} \gnl
\gcn{1}{1}{1}{1} \glcm \gnl
\grm \gcn{1}{1}{1}{1} \gnl
\gcn{1}{1}{1}{2} \gcn{1}{1}{3}{2} \gnl
\gvac{1} \hspace{-0,34cm} \gbr \gnl
\gvac{1} \gob{1}{N} \gob{1}{M}
\gend
$$
for $M,N\in {}^H\YD(\C)_{H^{cop}}$. Another possibility for the braiding is $\Phi^{3-}$
%taken with $\Phi^{-1}_{N, M}$ instead of $\Phi_{M, N}$ as above.
(similarly as in \rmref{Phi-mix-changed}). Using the fact that $\Phi_{H,H}$ is
symmetric, one may show that the functor $\A:{}^H\YD(\C)_{H^{cop}}\to {}_H\YD(\C)^{H^{op}}$
\label{funktor A} given by:
$$
\scalebox{0.9}[0.9]{
\gbeg{2}{4}
\got{1}{\hspace{-0,24cm}H} \got{1}{\hspace{0,1cm}\A(M)} \gnl
\glm \gnl
\gvac{1} \gcl{1} \gnl
\gvac{1} \gob{1}{\A(M)}
\gend} = \scalebox{0.9}[0.9]{
\gbeg{2}{5}
\got{1}{H} \got{1}{M} \gnl
\gibr \gnl
\gcl{1} \gmp{+} \gnl
\grm \gnl
\gob{1}{M}
\gend}\qquad\textnormal{and}\quad
\scalebox{0.9}[0.9]{
\gbeg{2}{4}
\got{1}{\hspace{0,3cm}\A(M)} \gnl
\grcm \gnl
\gcl{1} \gcl{1} \gnl
\gob{1}{\A(M)} \gob{1}{\hspace{0,14cm}H}
\gend} = \scalebox{0.9}[0.9]{
\gbeg{3}{5}
\got{1}{} \got{1}{M} \gnl
\glcm \gnl\gmp{-} \gcl{1} \gnl
\gbr \gnl
\gob{1}{M} \gob{1}{H}
\gend}
$$
for $M,N\in {}^H\YD(\C)_{H^{cop}}$ is an isomorphism of %braided monoidal
categories. We show that it is monoidal. For the right $H$-comodule structures we have:
$$\scalebox{0.9}[0.9]{
\gbeg{3}{6}
\got{1}{\A(M\ot N)} \gnl
\gcl{1} \gnl
\grcm \gnl
\gcl{2} \gcn{1}{1}{1}{3} \gnl
\gvac{2} \gcl{1} \gnl
\gob{1}{\A(M\ot N)} \gob{3}{H}
\gend}=\scalebox{0.9}[0.9]{
\gbeg{2}{5}
\got{1}{} \got{1}{M\ot N} \gnl
\glcm \gnl\gmp{-} \gcl{1} \gnl
\gbr \gnl
\gob{1}{M\ot N} \gob{1}{\hspace{0,24cm}H}
\gend}=\scalebox{0.9}[0.9]{
\gbeg{4}{9}
\gvac{1} \got{1}{M} \gvac{1} \got{1}{N} \gnl
\gvac{1} \gcl{1} \gvac{1} \gcl{1} \gnl
\glcm \glcm \gnl
\gcl{1} \gbr \gcl{2} \gnl
\gmu \gcl{1} \gnl
\gvac{1} \hspace{-0,34cm} \gmp{-} \gcn{1}{1}{2}{1} \gcn{1}{2}{2}{1} \gnl
\gvac{1} \gbr \gnl
\gvac{1} \gcl{1} \gbr \gnl
\gvac{1} \gob{1}{M} \gob{1}{N} \gob{1}{H}
\gend}=\scalebox{0.9}[0.9]{
\gbeg{4}{10}
\gvac{1} \got{1}{M} \gvac{1} \got{1}{N} \gnl
\gvac{1} \gcl{1} \gvac{1} \gcl{1} \gnl
\glcm \glcm \gnl
\gcl{1} \gbr \gcl{3} \gnl
\gmp{-} \gmp{-} \gcl{2} \gnl
\gibr \gnl
\gmu \gcn{1}{1}{1}{0} \gcn{1}{2}{1}{0} \gnl
\gvac{1} \hspace{-0,34cm} \gbr \gnl
\gvac{1} \gcl{1} \gbr \gnl
\gvac{1} \gob{1}{M} \gob{1}{N} \gob{1}{H}
\gend}=\scalebox{0.9}[0.9]{
\gbeg{4}{11}
\gvac{1} \got{1}{M} \gvac{1} \got{1}{N} \gnl
\gvac{1} \gcl{1} \gvac{1} \gcl{1} \gnl
\glcm \glcm \gnl
\gmp{-} \gcl{1} \gmp{-} \gcl{3} \gnl
\gcl{1} \gbr \gnl
\gcl{1} \gbr \gnl
\gbr \gbr \gnl
\gcl{1} \gbr \gcl{1} \gnl
\gcl{2} \gcl{2} \gibr \gnl
\gvac{2} \gmu \gnl
\gob{1}{M} \gob{1}{N} \gob{2}{H}
\gend}\stackrel{\Phi_{H,M}}{\stackrel{\Phi_{H,H}}{=}}
\scalebox{0.9}[0.9]{
\gbeg{4}{9}
\gvac{1} \got{1}{M} \gvac{1} \got{1}{N} \gnl
\gvac{1} \gcl{1} \gvac{1} \gcl{1} \gnl
\glcm \glcm \gnl
\gmp{-} \gcl{1} \gmp{-} \gcl{1} \gnl
\gbr \gbr \gnl
\gcl{1} \gbr \gcl{1} \gnl
\gcl{2} \gcl{2} \gbr \gnl
\gvac{2} \gmu \gnl
\gob{1}{M} \gob{1}{N} \gob{2}{H}
\gend}=\scalebox{0.9}[0.9]{
\gbeg{5}{7}
\got{1}{\A(M)} \gvac{1} \got{1}{\A(N)} \gnl
\gcl{1} \gvac{1} \gcl{1} \gnl
\grcm \grcm \gnl
\gcl{3} \gbr \gcl{1} \gnl
\gvac{1} \gcl{2} \gbr \gnl
\gvac{2} \gmu \gnl
\gob{1}{\hspace{-0,34cm}\A(M)} \gob{2}{\A(N)} \gob{1}{H.}
\gend}
$$
For the left $H$-module structures we find:
$$
\gbeg{5}{6}
\got{1}{H} \got{5}{\A(M\ot N)} \gnl
\gcn{2}{2}{1}{5} \gvac{1} \gcl{2} \gnl \gnl
\gvac{2} \glm \gnl
\gvac{3} \gcl{1} \gnl
\gvac{2} \gob{3}{\A(M\ot N)}
\gend=\scalebox{0.9}[0.9]{
\gbeg{2}{5}
\got{1}{H} \got{2}{M\ot N} \gnl
\gibr \gnl
\gcl{1} \gmp{+} \gnl
\grm \gnl
\gob{1}{M\ot N}
\gend}=\scalebox{0.9}[0.9]{
\gbeg{5}{10}
\got{1}{H} \got{1}{M} \got{1}{N} \gnl
\gibr \gcl{1} \gnl
\gcl{3} \gibr \gnl
\gvac{1} \gcl{2} \gmp{+} \gnl
\gvac{2} \gcn{1}{1}{1}{2} \gnl
\gcl{4} \gcl{2} \gcmu \gnl
\gvac{2} \gibr \gnl
\gvac{1} \gbr \gcl{1} \gnl
\grm \grm \gnl
\gob{1}{M} \gob{3}{N}
\gend} =
\scalebox{0.9}[0.9]{
\gbeg{5}{9}
\got{1}{H} \got{1}{M} \got{1}{N} \gnl
\gibr \gcl{1} \gnl
\gcl{5} \gibr \gnl
\gvac{1} \gcl{3} \gcn{1}{1}{1}{2} \gnl
\gvac{2} \gcmu \gnl
\gvac{2} \gmp{+} \gmp{+} \gnl
\gvac{1} \gbr \gcl{1} \gnl
\grm \grm \gnl
\gob{1}{M} \gob{3}{N}
\gend}\stackrel{nat.}{=}
\scalebox{0.9}[0.9]{
\gbeg{5}{8}
\got{2}{H} \got{1}{M} \got{1}{N} \gnl
\gcmu \gcl{1} \gcl{2} \gnl
\gcl{1} \gibr \gnl
\gibr \gibr \gnl
\gcl{1} \gmp{+} \gcl{1} \gmp{+} \gnl
\grm \grm \gnl
\gcl{1} \gvac{1} \gcl{1} \gnl
\gob{1}{M} \gob{3}{N}
\gend} \stackrel{\Phi_{H,M}}{=}
\scalebox{0.9}[0.9]{
\gbeg{5}{8}
\got{2}{\hspace{-0,12cm}H} \got{1}{\hspace{-0,3cm}\A(M)} \got{2}{\A(N)} \gnl
\gcmu \gcl{1} \gcl{2} \gnl
\gcl{1} \gbr \gnl
\glm \glm \gnl
\gvac{1} \gcl{3} \gvac{1} \gcl{3} \gnl
\gvac{1} \gob{1}{\A(M)} \gvac{1} \gob{1}{\A(N).}
\gend}
$$

\par\medskip
%\rmref{auto-isom-mix}

Analogously to the two versions of left-right YD-categories, we have two versions of righ-left
YD-categories, where the second one is: ${}^{H^{op}}\YD(\C)_H$. It is monoidal if $\Phi_{H,M}$ is
symmetric for all $M\in {}^{H^{op}}\YD(\C)_H$. This is a braided monoidal category with braiding:
$$
\Phi^{4+}_{M, N}=
\gbeg{4}{6}
\gvac{1} \got{1}{M} \got{1}{N} \gnl
\gvac{1} \gbr \gnl
\gcn{1}{1}{3}{2} \gcn{1}{1}{3}{4} \gnl
\gcn{1}{1}{2}{2} \gvac{1} \hspace{-0,34cm} \glcm \gnl
\gvac{1} \grm \gcl{1} \gnl
\gob{3}{N} \gob{1}{M}
\gend
$$
for $M,N\in {}^{H^{op}}\YD(\C)_H$. As in \rmref{lr-Hcop-cat} we have that $\Phi^{3\pm}$ is not
a braiding for ${}^{H^{op}}\YD(\C)_H$.

\vspace{0,2cm}

%\begin{rem}\rmlabel{auto-isom-mix}

Let us next examine the relation between the categories ${}_H\YD(\C)^{H^{op}}$ and
${}_{H^{cop}}\YD(\C)^H$, on the one hand, and ${}^{H^{op}}\YD(\C)_H$ and ${}^H\YD(\C)_{H^{cop}}$,
on the other hand. First of all recall that the corresponding identity functors are not isomorphisms
of braided monoidal categories (\rmref{lr-Hcop-cat}). Take $M\in {}_H\YD(\C)^H$. The object
$\ch(M)=M$ with structures:
$$
\scalebox{0.9}[0.9]{
\gbeg{2}{4}
\got{1}{\ch(M)} \gnl
\grcm \gnl
\gcl{1} \gcl{1} \gnl
\gob{1}{\ch(M)} \gob{1}{\hspace{0,24cm}H}
\gend} = \scalebox{0.9}[0.9]{
\gbeg{2}{4}
\got{1}{M} \gnl
\grcm \gnl
\gcl{1} \gmp{-} \gnl
\gob{1}{M} \gob{1}{H}
\gend}\quad\textnormal{and}\quad
\scalebox{0.9}[0.9]{
\gbeg{3}{4}
\got{1}{\hspace{-0,24cm}H} \got{1}{\hspace{0,24cm}\ch(M)} \gnl
\glm \gnl
\gvac{1} \gcl{1} \gnl
\gob{3}{\ch(M)}
\gend} = \scalebox{0.9}[0.9]{
\gbeg{3}{4}
\got{1}{H} \got{1}{M} \gnl
\gmp{+} \gcl{1} \gnl
\glm \gnl
\gvac{1} \gob{1}{M}
\gend}
$$
is a right $H^{cop}$-comodule and a left $H^{op}$-module. This defines a (bijective) functor
$\ch:{}_H\YD(\C)^H \to {}_{H^{op, cop}}\YD(\C)^{H^{op, cop}}$ %which acts as identity on morphisms.
(the objects of ${}_{H^{op, cop}}\YD(\C)^{H^{op, cop}}$ are left-right YD-modules over the Hopf
algebra $H^{op, cop}$). Indeed,
$$\gbeg{5}{8}
\got{2}{H} \got{1}{M} \gnl
\gcmu \gcl{1} \gnl
\gibr \grcm \gnl
\gmp{+} \gbr \gmp{-} \gnl
\glm \gbr \gnl
\gvac{1} \gcl{2} \gmu \gnl
\gvac{3} \hspace{-0,2cm} \gmp{+} \gnl
\gvac{1} \gob{2}{M} \gob{1}{H}
\gend=
\gbeg{5}{8}
\got{2}{H} \got{1}{M} \gnl
\gcmu \gcl{1} \gnl
\gibr \grcm \gnl
\gmp{+} \gmp{+} \gcl{1} \gmp{-} \gnl
\gcl{1} \gbr \gmp{+} \gnl
\glm \gmu \gnl
\gvac{1} \gcl{1} \gcn{1}{1}{2}{2} \gnl
\gvac{1} \gob{1}{M} \gob{2}{H}
\gend=
\gbeg{4}{7}
\got{1}{H} \got{2}{M} \gnl
\gmp{+} \gcn{1}{1}{2}{2} \gnl
\hspace{-0,34cm} \gcmu \grcm \gnl
\gcl{1} \gbr \gcl{1} \gnl
\glm \gmu \gnl
\gcn{1}{1}{3}{3} \gcn{1}{1}{4}{4} \gnl
\gob{3}{M} \gob{1}{\hspace{-0,26cm}H}
\gend\stackrel{M\in {}_H\YD(\C)^H}{=}
\gbeg{3}{9}
\got{1}{H} \got{2}{M} \gnl
\gmp{+} \gvac{1} \hspace{-0,34cm} \gcl{2} \gnl
\gcmu \gnl
\gcn{1}{1}{1}{3} \glm \gnl
\gvac{1} \gibr \gnl
\gcn{1}{1}{3}{1} \gvac{1} \gcl{2} \gnl
\grcm \gnl
\gcl{1} \gmu \gnl
\gob{1}{M} \gob{2}{H}
\gend=
\gbeg{3}{10}
\got{2}{H} \got{1}{M} \gnl
\gcmu \gcl{3} \gnl
\gibr \gnl
\gmp{+} \gmp{+} \gnl
\gcn{1}{1}{1}{3} \glm \gnl
\gvac{1} \gibr \gnl
\gcn{1}{1}{3}{1} \gvac{1} \gcl{2} \gnl
\grcm \gnl
\gcl{1} \gmu \gnl
\gob{1}{M} \gob{2}{H}
\gend=
\gbeg{4}{12}
\got{2}{H} \got{1}{M} \gnl
\gcmu \gcl{3} \gnl
\gibr \gnl
\gcl{1} \gmp{+} \gnl
\gcn{1}{1}{1}{3} \glm \gnl
\gvac{1} \gibr \gnl
\gcn{1}{1}{3}{1} \gvac{1} \gcl{3} \gnl
\grcm \gnl
\gcl{1} \gmp{-} \gnl
\gcl{1} \gmp{+} \gmp{+} \gnl
\gcl{1} \gmu \gnl
\gob{1}{M} \gob{2}{H}
\gend=
\gbeg{4}{13}
\got{2}{H} \got{1}{M} \gnl
\gcmu \gcl{3} \gnl
\gibr \gnl
\gcl{1} \gmp{+} \gnl
\gcn{1}{1}{1}{3} \glm \gnl
\gvac{1} \gibr \gnl
\gcn{1}{1}{3}{1} \gvac{1} \gcl{3} \gnl
\grcm \gnl
\gcl{1} \gmp{-} \gnl
\gcl{1} \gbr \gnl
\gcl{1} \gmu \gnl
\gcl{1} \gvac{1} \hspace{-0,34cm} \gmp{+} \gnl
\gob{2}{M} \gob{1}{H}
\gend
$$
is equivalent to
$$
\gbeg{4}{7}
\got{2}{H} \got{1}{\ch(M)} \gnl
\gcmu \gcl{1} \gnl
\gibr \grcm \gnl
\gcl{1} \gbr \gcl{1} \gnl
\glm \gbr \gnl
\gvac{1} \gcl{1} \gmu \gnl
\gvac{1} \gob{1}{\ch(M)} \gob{2}{H}
\gend=
\gbeg{5}{9}
\got{2}{H} \got{1}{M} \gnl
\gcn{1}{1}{2}{2} \gvac{1} \gcl{2} \gnl
\gcmu \gnl
\gibr \grcm \gnl
\gmp{+} \gbr \gmp{-} \gnl
\glm \gbr \gnl
\gvac{1} \gcl{2} \gmu \gnl
\gvac{2} \gcn{1}{1}{2}{2} \gnl
\gvac{1} \gob{1}{M} \gob{2}{H}
\gend=
\gbeg{4}{12}
\got{2}{H} \got{1}{M} \gnl
\gcmu \gcl{3} \gnl
\gibr \gnl
\gcl{1} \gmp{+} \gnl
\gcn{1}{1}{1}{3} \glm \gnl
\gvac{1} \gibr \gnl
\gcn{1}{1}{3}{1} \gvac{1} \gcl{3} \gnl
\grcm \gnl
\gcl{1} \gmp{-} \gnl
\gcl{1} \gbr \gnl
\gcl{1} \gmu \gnl
\gob{1}{M} \gob{2}{H}
\gend=
\gbeg{4}{10}
\got{2}{H} \got{1}{\ch(M)} \gnl
\gcmu \gcl{3} \gnl
\gibr \gnl
\gcn{1}{1}{1}{3} \glm \gnl
\gvac{1} \gibr \gnl
\gcn{1}{1}{3}{1} \gvac{1} \gcl{2} \gnl
\grcm \gnl
\gcl{1} \gbr \gnl
\gcl{1} \gmu \gnl
\gob{1}{\ch(M)} \gob{2}{H.}
\gend
$$
%Let us identify these categories by the Hopf algebra isomorphism $H\iso H^{op, cop}$.
The functor $\ch$ restricts to monoidal functors $\ch_1: {}_H\C\to {}_{H^{cop}}\C$ and
$\ch_2: \C^H\to \C^{H^{op}}$. (For $M,M'\in {}_H\YD(\C)^H$, the module structures of $\ch_i(M\ot M')$
and $\ch_i(M)\ot\ch_i(M')$, for $i=1$, are compatible since the antipode is a coalgebra anti-morphism
and since $\Phi_{H,H}$ is symmetric, while the corresponding comodule structures for $i=2$ are
compatible since the antipode of $H$ is an algebra anti-morphism.)
%\textcolor{rojo}{However, writing that $\ch$ is a (monoidal) functor ${}_H\YD(\C)^{H^{op}}
%\to {}_{H^{cop}}\YD(\C)^H$ would be misleading, because $\ch$ sends an $H$-module (resp. $H$-comodule)
%to an $H^{op}$-module (resp. $H^{cop}$-comodule), which is not reflected in this notation.
%We should write $\ch: {}_H\YD(\C)^{H^{op}} \to {}_{H^{op,cop}}\YD(\C)^{H^{cop}}$, but here we are
%not interested in this kind of categories. Thus we will not regard $\ch$ as a monoidal functor
%${}_H\YD(\C)^{H^{op}} \to {}_{H^{cop}}\YD(\C)^H$. }
Hence $\ch$ induces a monoidal functor %$\ch': {}_H\YD(\C)^{H^{op}} \to {}_{H^{op,cop}}\YD(\C)^{H^{cop}}$.
$\ch'$ from ${}_H\YD(\C)^{H^{op}}$ to the
%The latter is understood as the
category $\D$ with objects in ${}_{H^{op, cop}}\YD(\C)^{H^{op, cop}}$, whose
monoidal structure and the braiding are like the ones in ${}_{H^{cop}}\YD(\C)^H$. The category $\D$ is
shown to be indeed a braided monoidal category, %A part from the fact that its objects are $H^{cop}$-comodules
%and $H^{op}$-modules, while the objects of ${}_{H^{cop}}\YD(\C)^H$ are $H$-comodules and $H^{cop}$-modules,
%%we will not identify it with the category ${}_{H^{cop}}\YD(\C)^H$, whose objects are
%%$H$-comodules and $H^{cop}$-modules. Moreover,
however the functor $\ch'$ does not respect the braidings. It is easily seen that $\ch(\Phi^{1+})=\Phi^{1+}
\not=\Phi^{2+}$. Thus we will not consider that $\ch'$ induces a braided monoidal functor
${}_H\YD(\C)^{H^{op}} \to {}_{H^{cop}}\YD(\C)^H$.

%In \cite[Lemma 3.5.1]{Besp} a functor ${}_H\YD(\C)^H \to {}_{H^{op}}\YD(\C)^{H^{op}}$ is proposed
%which gives a new module structure like $\ch$, but leaves the coaction unchanged. %Under the
%%conventions of our work this is a well-defined functor if $\Phi_{H,M}$ and $\Phi_{H,H}$ are
%%symmetric, for all YD-modules $M$. By the second part of \cite[Lemma 3.5.1]{Besp}, this
%It is claimed to extend to a braided monoidal functor ${}_H\YD(\C)^{H^{op}} \to
%{}_{H^{op, cop}}\YD(\crta\C)^{H^{op}}$. On the other hand,
In \cite[Lemma 3.5.4]{Besp} it is proved that ${}_H\YD(\C)^{H^{op}} \stackrel{(\Id, \Omega)}{\to}({}_{H^{cop}}\YD(\C)^H)^{cop}$ is an isomorphism of braided monoidal categories, where
$(\Id, \Omega)$ is the extension of the braided monoidal isomorphism functor $\C\to\C^{cop}$.
As announced in the introduction
of our paper, we do not make this kind of identifications, we stick to the original category $\C$.

\vspace{0,2cm}

Similarly, there is a functor $\B:{}^H\YD(\C)_H \to {}^{H^{op,cop}}\YD(\C)_{H^{op,cop}}$ defined via:
$$
\scalebox{0.9}[0.9]{
\gbeg{2}{4}
\got{3}{\B(M)} \gnl
\glcm \gnl
\gcl{1} \gcl{1} \gnl
\gob{1}{H} \gob{1}{\hspace{0,24cm}\B(M)}
\gend} = \scalebox{0.9}[0.9]{
\gbeg{2}{4}
\got{3}{M} \gnl
\glcm \gnl
\gmp{-} \gcl{1} \gnl
\gob{1}{H} \gob{1}{M}
\gend}\quad\textnormal{and}\quad
\scalebox{0.9}[0.9]{
\gbeg{3}{4}
\got{1}{\B(M)} \got{1}{\hspace{0,24cm}H} \gnl
\grm \gnl
\gcl{1} \gnl
\gob{1}{\B(M)}
\gend} = \scalebox{0.9}[0.9]{
\gbeg{3}{4}
\got{1}{M} \got{1}{H} \gnl
\gcl{1} \gmp{+} \gnl
\grm \gnl
\gob{1}{M}
\gend}
$$
for $M\in {}^H\YD(\C)_{H^{cop}}$. It induces monoidal functors $\B_1:{}^H\C \to {}^{H^{op}}\C$
and $\B_2:\C_H \to \C_{H^{cop}}$, %and hence a monoidal functor $\B':{}^H\YD(\C)_{H^{cop}} \to
%{}^{H^{op,cop}}\YD(\C)_{H^{op}}$. Thus $\B$ does
but not a braided monoidal functor ${}^H\YD(\C)_{H^{cop}} \to {}^{H^{op}}\YD(\C)_H$.
%but we will not consider $\B:{}^H\YD(\C)_{H^{cop}} \to {}^{H^{op}}\YD(\C)_H$ as a monoidal functor.

%\vspace{2cm}

\subsection{Comparing all the categories} \selabel{compare all}

To sum up the results of this section consider the following diagram:
\begin{equation}\eqlabel{first4categories}
\scalebox{0.84}{
\bfig
\putmorphism(-300,200)(1,0)[{}_H\YD(\C)^{H^{op}}`{}^H\YD(\C)_{H^{cop}}`\A^{-1}]{1140}1b
\putmorphism(880,885)(-2,-1)[``]{1310}1l
\putmorphism(880,910)(-2,-1)[``\F_1]{1310}0l
\putmorphism(820,880)(0,-1)[``\F_2]{680}1r
\putmorphism(-430,880)(1,0)[{}_{D(H)}\C`{}_H ^H\YD(\C)`\F_l]{1220}1a
\putmorphism(-380,880)(0,-1)[``\F]{660}1l
\put(-210,630){\fbox{1}}
\put(510,430){\fbox{2}}
\efig}
\end{equation}
We define the functors $\F_1$ and $\F_2$ so that the triangles $\langle 1\rangle$ and
$\langle 2\rangle$ commute. We write out the functor $\F_1$ explicitly:
$$
\scalebox{0.9}[0.9]{
\gbeg{2}{4}
\got{1}{\hspace{0,3cm}\F_1(M)} \gnl
\grcm \gnl
\gcl{1} \gcl{1} \gnl
\gob{1}{\hspace{-0,16cm}\F_1(M)} \gob{1}{\hspace{0,14cm}H}
\gend} = \scalebox{0.9}[0.9]{
\gbeg{3}{5}
\got{1}{} \got{1}{M} \gnl
\glcm \gnl\gmp{-} \gcl{1} \gnl
\gbr \gnl
\gob{1}{M} \gob{1}{H}
\gend}\quad\textnormal{with inverse}\quad
\scalebox{0.9}[0.9]{
\gbeg{2}{4}
\got{1}{\hspace{0,3cm}\F_1^{-1}(N)} \gnl
\glcm \gnl
\gcl{1} \gcl{1} \gnl
\gob{1}{\hspace{-0,24cm}\F_1^{-1}(N)} \gob{1}{\hspace{0,24cm}H}
\gend} = \scalebox{0.9}[0.9]{
\gbeg{3}{5}
\got{1}{N} \gnl
\grcm \gnl
\gcl{1} \gmp{+} \gnl
\gibr \gnl
\gob{1}{H} \gob{1}{N.}
\gend}
$$
We saw that the functors $\F_l, \F$ and $\A$ are monoidal isomorphisms, so we have four mutually
isomorphic monoidal categories.
%Then the triangles $\langle 1\rangle$ and $\langle 2\rangle$ in \equref{7 categories}
%commute as arrows of monoidal isomorphism functors. Recall that $\F_r$ is not a monoidal functor.
We now compare their braidings. We have: % of the respective .
$$\Phi_{M,N}^L=
\scalebox{0.9}[0.9]{\gbeg{3}{5}
\got{1}{} \got{1}{M} \got{1}{N} \gnl
\glcm \gcl{1} \gnl
\gcl{1} \gbr \gnl
\glm \gcl{1} \gnl
\gob{1}{} \gob{1}{N} \gob{1}{M}
\gend}\stackrel{\F_1^{-1}}{=}
\scalebox{0.9}[0.9]{\gbeg{3}{7}
\got{1}{M} \got{1}{} \got{1}{N} \gnl
\grcm \gcl{3} \gnl
\gcl{1} \gmp{+} \gnl
\gibr \gnl
\gcl{1} \gbr \gnl
\glm \gcl{1} \gnl
\gob{1}{} \gob{1}{N} \gob{1}{M}
\gend}\stackrel{\Phi_{H,M}}{\stackrel{nat.}{=}}
\scalebox{0.9}[0.9]{\gbeg{4}{7}
\got{2}{M} \got{2}{N} \gnl
\gvac{1} \hspace{-0,34cm} \grcm \gcn{1}{1}{1}{1} \gnl
\gvac{1} \gcl{1} \gmp{+} \gcl{1} \gnl
\gcn{1}{1}{3}{3} \gvac{1} \glm \gnl
\gcn{1}{1}{3}{4} \gcn{1}{1}{5}{4} \gnl
\gvac{2} \hspace{-0,34cm} \gbr \gnl
\gvac{2} \gob{1}{N} \gob{1}{M}
\gend}=(\Phi^{1-}_{N, M})^{-1}
$$
and
$$
\Phi^{3+}_{M, N}=
\gbeg{4}{6}
\got{1}{M} \got{3}{N} \gnl
\gcn{1}{1}{1}{1} \glcm \gnl
\grm \gcn{1}{1}{1}{1} \gnl
\gcn{1}{1}{1}{2} \gcn{1}{1}{3}{2} \gnl
\gvac{1} \hspace{-0,34cm} \gbr \gnl
\gob{3}{N} \gob{1}{\hspace{-0,6cm}M}
\gend\stackrel{\A^{-1}}{=}
\gbeg{4}{9}
\got{1}{M} \got{1}{N} \gnl
\gcl{4} \grcm \gnl
\gvac{1} \gcl{1} \gmp{+} \gnl
\gvac{1} \gibr \gnl
\gvac{1} \gmp{-} \gcl{3} \gnl
\gbr \gnl
\glm \gnl
\gvac{1} \gbr \gnl
\gvac{1} \gob{1}{N} \gob{1}{M}
\gend=
\gbeg{4}{7}
\got{1}{M} \got{1}{N} \gnl
\gcl{2} \grcm \gnl
\gvac{1} \gibr \gnl
\gbr \gcl{2} \gnl
\glm \gnl
\gvac{1} \gbr \gnl
\gvac{1} \gob{1}{N} \gob{1}{M}
\gend\stackrel{nat.}{=}
\gbeg{4}{6}
\gvac{1} \got{1}{M} \got{1}{N} \gnl
\gvac{1} \gbr \gnl
\gcn{1}{1}{3}{2} \gcn{1}{1}{3}{4} \gnl
\gvac{1} \hspace{-0,34cm} \grcm \gcn{1}{1}{1}{1} \gnl
\gcn{1}{1}{3}{3} \gvac{1} \glm \gnl
\gob{3}{N} \gob{1}{M}
\gend=\Phi^{1+}_{M, N}.
$$
This proves that the functors $\F_1:{}_H ^H\YD(\C)\to {}_H\YD(\C)^{H^{op}}$ and
$\A:{}^H\YD(\C)_{H^{cop}}\to {}_H\YD(\C)^{H^{op}}$ are isomorphisms of braided monoidal
categories. By \prref{YD-DH}, $\F: {}_{D(H)}\C \to {}_H\YD(\C)^{H^{op}}$ is also such a
functor. Then by commutativity of $\langle 1\rangle$ and $\langle 2\rangle$
in \equref{first4categories} we have four mutually isomorphic braided monoidal categories.
\par\medskip

Symmetrically as in \equref{first4categories}, we may consider:
\begin{equation}\eqlabel{other4categories}
\scalebox{0.84}{
\bfig
\putmorphism(-300,200)(1,0)[{}_{H^{cop}}\YD(\C)^H`{}^{H^{op}}\YD(\C)_H`\E]{1140}1b
\putmorphism(-460,885)(2,-1)[``]{1310}1l
\putmorphism(-480,910)(2,-1)[``\F_3]{1310}0r
\putmorphism(820,880)(0,-1)[``\F_4]{680}1r
\putmorphism(-430,880)(1,0)[\C_{D(H)}`\YD(\C)_H ^H`\S]{1220}1a
\putmorphism(-380,880)(0,-1)[``\T]{660}1l
\put(-210,430){\fbox{3}}
\put(510,630){\fbox{4}}
\efig}
\end{equation}
The functors $\S: \C_{D(H)}\to \YD(\C)_H ^H, \quad\T:\C_{D(H)}\to {}_{H^{cop}}\YD(\C)^H$ and
$\E: {}_{H^{cop}}\YD(\C)^H \to {}^{H^{op}}\YD(\C)_H$ are given by:
$$
\scalebox{0.9}[0.9]{
\gbeg{3}{4}
\got{1}{\S(M)} \gnl
\grcm \gnl
\gcl{1} \gcn{1}{1}{1}{3} \gnl
\gob{1}{\S(M)} \gob{3}{H}
\gend} = \scalebox{0.9}[0.9]{
\gbeg{7}{5}
\got{1}{M} \gnl
\gcl{2} \gdb \gnl
\grm \gmp{+} \gnl
\gcl{1} \gvac{1} \gcl{1} \gnl
\gob{1}{M} \gob{3}{H}
\gend}\hspace{-0,4cm}\textnormal{with}\qquad\quad
\scalebox{0.9}[0.9]{
\gbeg{3}{4}
\got{1}{\S^{-1}(N)} \gvac{1} \got{1}{H^*} \gnl
\gcl{1} \gcn{1}{1}{3}{1} \gnl
\grm \gnl
\gob{1}{\S^{-1}(N)}
\gend} = \scalebox{0.9}[0.9]{
\gbeg{2}{5}
\got{1}{N} \got{3}{H^*} \gnl
\grcm \gcl{2} \gnl
\gcl{2} \gmp{-} \gnl
\gvac{1} \gev \gnl
\gob{1}{N}
\gend}
$$
$$
\scalebox{0.9}[0.9]{
\gbeg{3}{4}
\got{1}{\T(M)} \gnl
\grcm \gnl
\gcl{1} \gcn{1}{1}{1}{3} \gnl
\gob{1}{\T(M)} \gob{3}{H}
\gend} = \scalebox{0.9}[0.9]{
\gbeg{7}{5}
\got{1}{M} \gnl
\gcl{2} \gdb \gnl
\grm \gmp{+} \gnl
\gcl{1} \gvac{1} \gcl{1} \gnl
\gob{1}{M} \gob{3}{H}
\gend}\hspace{-1cm}; \qquad
\scalebox{0.9}[0.9]{
\gbeg{3}{4}
\got{1}{H} \gvac{1} \got{1}{\T(M)} \gnl
\gcn{1}{1}{1}{3} \gvac{1} \gcl{1} \gnl
\gvac{1} \glm \gnl
\gvac{2} \gob{1}{\T(M)}
\gend} = \scalebox{0.9}[0.9]{
\gbeg{2}{5}
\got{1}{H} \got{1}{M} \gnl
\gbr \gnl
\gcl{1} \gmp{-} \gnl
\grm \gnl
\gob{1}{M}
\gend}
$$
and
$$\scalebox{0.9}[0.9]{
\gbeg{2}{4}\got{1}{} \got{1}{\hspace{0,3cm}\E(K)} \gnl\glcm \gnl\gcl{1} \gcl{1} \gnl
\gob{1}{\hspace{-0,14cm}H} \gob{1}{\E(K)}  \gend} = \scalebox{0.9}[0.9]{\gbeg{3}{5}
\got{1}{K} \gnl\grcm \gnl\gibr \gnl\gmp{+} \gcl{1} \gnl\gob{1}{H} \gob{1}{K}
\gend};\quad
\scalebox{0.9}[0.9]{\gbeg{3}{4}\got{1}{\E(K)} \got{3}{H} \gnl
\gcl{1} \gcn{1}{1}{3}{1} \gnl\grm \gnl\gob{1}{\E(K)} \gend} = \scalebox{0.9}[0.9]{\gbeg{4}{5}
\got{1}{K} \got{1}{H} \gnl\gcl{1} \gmp{-} \gnl\gbr \gnl\glm \gnl\gvac{1} \gob{1}{K} \gend}
\quad\textnormal{with}\quad
\scalebox{0.9}[0.9]{
\gbeg{2}{4}\got{1}{\hspace{0,3cm}\E^{-1}(L)} \gnl\grcm \gnl\gcl{1} \gcl{1} \gnl
\gob{1}{\hspace{-0,24cm}H} \gob{1}{\hspace{0,24cm}\E^{-1}(L)}  \gend} = \scalebox{0.9}[0.9]{\gbeg{3}{5}
\got{1}{} \got{1}{L} \gnl\glcm \gnl\gbr \gnl\gcl{1} \gmp{-} \gnl\gob{1}{L} \gob{1}{H}
\gend};\quad
\scalebox{0.9}[0.9]{\gbeg{3}{4}\got{1}{H} \got{3}{\E^{-1}(L)} \gnl
\gcn{1}{1}{1}{3} \gvac{1} \gcl{1} \gnl\gvac{1} \glm \gnl\gvac{1} \gob{2}{\E^{-1}(L)} \gend} =
\scalebox{0.9}[0.9]{\gbeg{4}{5}
\got{1}{H} \got{1}{L} \gnl\gmp{+} \gcl{1} \gnl\gibr \gnl\grm \gnl\gob{1}{L} \gend}
$$
(in the definitions of $\S$ and $\T$ the symbols \hspace{-0,7cm}
$\gbeg{1}{1}\got{1}{}\gdb\gnl\gob{1}{} \gob{1}{}\gend$ \hspace{0,7cm}
and \hspace{-0,7cm} $\gbeg{1}{1}\got{1}{}\gev\gnl\gob{1}{} \gob{1}{}\gend$ \hspace{0,7cm}
stand for the morphisms $d':I\to H^*\ot H$ and $e':H\ot H^*\to I$, recall \inref{r.adj. dual}).
The functors
$\F_3$ and $\F_4$ are defined so that the triangles $\langle 3\rangle$ and $\langle 4\rangle$ in
\equref{other4categories} commute. The proofs that $\S, \F_3$ and $\E$ are monoidal functors
are analogous to the corresponding proofs for the functors $\F_l, \F$ and $\A$, respectively.
%By commutativity of $\langle 3\rangle$ and $\langle 4\rangle$
Then clearly also $\T$ and $\F_4$ are monoidal. The braiding in $\C_{D(H)}$ is given by:
$$\Psi^R_{M, N}:=
\scalebox{0.9}[0.9]{
\gbeg{7}{9}
\got{1}{M} \got{1}{N} \gnl
\gcl{4} \gcl{3} \gu{1} \gvac{2} \gu{1} \gnl
\gvac{2} \gcl{1} \gdb \gcl{1} \gnl
\gvac{2} \glmptb \gnot{\hspace{-0,3cm}D(H)} \grmpt
 \glmptb \gnot{\hspace{-0,3cm}D(H)} \grmpt \gnl
\gvac{1} \gbr \gcn{1}{1}{3}{1} \gnl
\grm \grm \gnl
\gcn{1}{1}{1}{2} \gcn{1}{1}{3}{2} \gnl
\gvac{1} \hspace{-0,34cm} \gibr \gnl
\gvac{1} \gob{1}{N} \gob{2}{\hspace{-0,2cm}M}
\gend}=
\scalebox{0.9}[0.9]{
\gbeg{4}{7}
\got{1}{M} \got{1}{N} \gnl
\gcl{2} \gcl{1} \gdb \gnl
\gvac{1} \gbr \gcl{1} \gnl
\grm \grm \gnl
\gcn{1}{1}{1}{2} \gvac{1} \gcn{1}{1}{1}{0} \gnl
\gvac{1} \hspace{-0,34cm} \gibr \gnl
\gvac{1} \gob{1}{N} \gob{1}{M}
\gend}
$$
and we have:
$$\Psi^R_{M,N}=
\scalebox{0.9}[0.9]{
\gbeg{4}{7}
\got{1}{M} \got{1}{N} \gnl
\gcl{2} \gcl{1} \gdb \gnl
\gvac{1} \gbr \gcl{1} \gnl
\grm \grm \gnl
\gcn{1}{1}{1}{2} \gvac{1} \gcn{1}{1}{1}{0} \gnl
\gvac{1} \hspace{-0,34cm} \gibr \gnl
\gvac{1} \gob{1}{N} \gob{1}{M}
\gend}\stackrel{\S^{-1}}{=}
\scalebox{0.9}[0.9]{
\gbeg{6}{7}
\got{1}{M} \gvac{1} \got{1}{N} \gnl
\grcm \gcl{1} \gdb \gnl
\gcl{2} \gmp{-} \gbr \gcl{1} \gnl
\gvac{1} \gev \grm \gnl
\gcn{1}{1}{1}{3} \gvac{1} \gcn{1}{1}{3}{1} \gnl
\gvac{1} \gibr \gnl
\gvac{1} \gob{1}{N} \gob{1}{M}
\gend}=
\scalebox{0.9}[0.9]{
\gbeg{4}{7}
\got{1}{M} \gvac{1} \got{1}{N} \gnl
\grcm \gcl{1} \gnl
\gcl{3} \gibr \gnl
\gvac{1} \gcl{1} \gmp{-} \gnl
\gvac{1} \grm \gnl
\gibr \gnl
\gob{1}{N} \gob{1}{M}
\gend}\stackrel{nat.}{=}
\scalebox{0.9}[0.9]{
\gbeg{4}{7}
\got{1}{M} \got{1}{N} \gnl
\gbr \gnl
\gcl{3} \grcm \gnl
\gvac{1} \gcl{1} \gmp{-} \gnl
\gvac{1} \gibr \gnl
\grm \gcl{1} \gnl
\gob{1}{N} \gob{3}{M}
\gend}=(\Phi^R_{N,M})^{-1}
$$

$$\Psi^R_{M,N}=
\scalebox{0.9}[0.9]{
\gbeg{4}{7}
\got{1}{M} \got{1}{N} \gnl
\gcl{2} \gcl{1} \gdb \gnl
\gvac{1} \gbr \gcl{1} \gnl
\grm \grm \gnl
\gcn{1}{1}{1}{2} \gvac{1} \gcn{1}{1}{1}{0} \gnl
\gvac{1} \hspace{-0,34cm} \gibr \gnl
\gvac{1} \gob{1}{N} \gob{1}{M}
\gend}\stackrel{\T^{-1}}{=}
\scalebox{0.9}[0.9]{
\gbeg{6}{9}
\got{1}{M} \gvac{1} \got{1}{N} \gnl
\grcm \gcl{1} \gdb \gnl
\gcl{3} \gmp{-} \gbr \gcl{1} \gnl
\gvac{1} \gev \gibr \gnl
\gvac{3} \gmp{+} \gcl{1} \gnl
\gcn{2}{2}{1}{7} \gvac{1} \glm \gnl
\gvac{4} \gcl{1} \gnl
\gvac{3} \gibr \gnl
\gvac{3} \gob{1}{N} \gob{1}{M}
\gend}\stackrel{\Phi_{H, N}}{=}
\gbeg{4}{6}
\got{1}{M} \got{3}{N} \gnl
\grcm \gcn{1}{1}{1}{1} \gnl
\gcn{1}{1}{1}{1} \glm \gnl
\gcn{1}{1}{1}{2} \gcn{1}{1}{3}{2} \gnl
\gvac{1} \hspace{-0,34cm} \gibr \gnl
\gob{3}{N} \gob{1}{\hspace{-0,6cm}M}
\gend=\Phi^{2-}_{M, N}
$$
and
$$
\Phi^{2+}_{M, N}=
\gbeg{4}{6}
\got{1}{M} \got{3}{N} \gnl
\grcm \gcn{1}{1}{1}{1} \gnl
\gcn{1}{1}{1}{1} \glm \gnl
\gcn{1}{1}{1}{2} \gcn{1}{1}{3}{2} \gnl
\gvac{1} \hspace{-0,34cm} \gbr \gnl
\gob{3}{N} \gob{1}{\hspace{-0,6cm}M}
\gend\stackrel{\E^{-1}}{=}
\scalebox{0.9}[0.9]{
\gbeg{4}{9}
\gvac{1} \got{1}{M} \got{1}{N} \gnl
\glcm \gcl{4} \gnl
\gbr \gnl
\gcl{4} \gmp{-} \gnl
\gvac{1} \gmp{+} \gnl
\gvac{1} \gibr \gnl
\gvac{1} \grm \gnl
\gbr \gnl
\gob{1}{N} \gob{1}{M}
\gend}\stackrel{\Phi_{H,M}}{\stackrel{\Phi_{H,N}}{=}}
\scalebox{0.9}[0.9]{
\gbeg{4}{7}
\gvac{1} \got{1}{M} \got{1}{N} \gnl
\glcm \gcl{2} \gnl
\gibr \gnl
\gcl{2} \gbr \gnl
\gvac{1} \grm \gnl
\gbr \gnl
\gob{1}{N} \gob{1}{M}
\gend}\stackrel{nat.}{=}
\scalebox{0.9}[0.9]{
\gbeg{4}{9}
\got{1}{M} \got{1}{N} \gnl
\gbr \gnl
\gcl{3} \gcn{1}{1}{1}{3} \gnl
\gvac{1} \glcm \gnl
\gvac{1} \gibr \gnl
\gibr \gcl{1} \gnl
\gcl{1} \grm \gnl
\gbr \gnl
\gob{1}{N} \gob{1}{M}
\gend}=
\gbeg{4}{6}
\gvac{1} \got{1}{M} \got{1}{N} \gnl
\gvac{1} \gbr \gnl
\gcn{1}{1}{3}{2} \gcn{1}{1}{3}{4} \gnl
\gcn{1}{1}{2}{2} \gvac{1} \hspace{-0,34cm} \glcm \gnl
\gvac{1} \grm \gcl{1} \gnl
\gob{3}{N} \gob{1}{M}
\gend=\Phi^{4+}_{M, N}.
$$
(The braiding $\Phi^{2+}_{M, N}$ is the one from \rmref{lr-Hcop-cat}.) This proves that the functors
$\S, %: \C_{D(H)}\to \YD(\C)_H ^H,
\T$ %:\C_{D(H)}\to {}_{H^{cop}}\YD(\C)^H$
and $\E$ %: {}_{H^{cop}}\YD(\C)^H \to {}^{H^{op}}\YD(\C)_H$
respect the braidings.

Note that our result that $\E:{}_{H^{cop}}\YD(\C)^H\to {}^{H^{op}}\YD(\C)_H$ is an
isomorphisms of braided monoidal categories generalizes \cite[Lemma 3.5.2]{Besp}, where the braided
monoidal isomorphism functor ${}^{H^{op}}\YD(\C)_H\to ({}_{H^{cop}}\YD(\C)^H)^{op,cop}$ is given if
$\C$ has right duals. It sends an object from the source category to its dual object.
\par\medskip

Finally, let us %discuss the relation between
record that we do not find any braided monoidal functor which would connect the two groups of categories
from \equref{first4categories} and \equref{other4categories}.
At the end of \ssref{mixing} we showed that a natural candidate $\ch'$
for a monoidal functor from ${}_H\YD(\C)^{H^{op}}$ to ${}_{H^{cop}}\YD(\C)^H$ is not a braided functor.
Likewise, at the end of \ssref{left-left} we showed that $\Ll: {}_H ^H\YD(\C)\to \YD(\C)^H_H$ is a
monoidal but not a braided functor. In the relation (3.5.1) after \cite[Corollary 3.5.5]{Besp} two
(mutually isomorphic) isomorphism functors $\G_1, \G_2: \YD(\C)_H ^H \to {}_H ^H\YD(\C)$ are given.
For $M\in \YD(\C)_H ^H$ with right module and comodule structure morphisms $\mu$ and $\rho$ respectively,
the functors $\G_1$ and $\G_2$ are defined by
$\G_1(M, \nu, \rho)=(M, \mu_1=\nu\Phi^{-1}(S^{-1}\ot M), \lambda_1=(S\ot M)\Phi\rho)$ and
$\G_2(M, \nu, \rho)=(M, \mu_2=\nu\Phi(S\ot M), \lambda_2=(S^{-1}\ot M)\Phi^{-1}\rho)$, respectively.
Here $\mu_i$ and $\lambda_i$ denote the left module and comodule structure morphisms of $\G_i(M)=M$,
respectively, for $i=1,2$.
That these functors are well-defined one can check directly applying \equref{right YD-oth}. However,
that they are not monoidal we can see even when $\C=Vec$, the category of vector spaces. Let us see this for $\G_1$:
$h \tr (m\ot n)=(m\ot n)\tl S^{-1}(h) =m\tl S^{-1}(h_{(2)}) \ot n\tl S^{-1}(h_{(1)}) = h_{(2)}\tr m \ot h_{(1)}\tr n$,
which shows that $\G_1$ restricts to a monoidal functor $\M_H \to {}_{H^{cop}}\M$.
Moreover, a direct check shows that if $\Phi_{H,M}$ is symmetric for any $M\in \YD(\C)_H ^H$, the
functor $\G_1$ is a braided monoidal isomorphism $\YD(\C)_H ^H \to {}_{H^{cop}} ^{H^{op}}\YD(\C)$,
where ${}_{H^{cop}} ^{H^{op}}\YD(\C)$ is a braided monoidal category with braiding:
$$
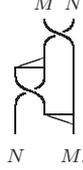
\scalebox{0.9}[0.9]{
\gbeg{4}{6}
\gvac{1} \got{1}{M} \got{1}{N} \gnl
\gvac{1} \gbr \gnl
\glcm \gcl{2} \gnl
\gbr \gnl
\gcl{1} \glm \gnl
\gob{1}{N} \gvac{1} \gob{1}{M.}
\gend}%=\Phi^*_{M, N}.
$$
Thus, we can complete \equref{other4categories}, and symmetrically \equref{first4categories},
into commutative diagrams of isomorphic braided monoidal categories:
$$
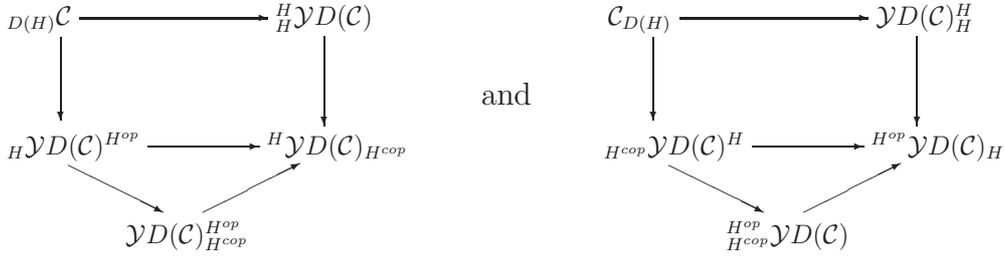
\scalebox{0.84}{
\bfig
\putmorphism(-300,-160)(1,0)[{}_H\YD(\C)^{H^{op}}`{}^H\YD(\C)_{H^{cop}}`]{1000}1b
\putmorphism(650,330)(0,-1)[``]{480}1r
\putmorphism(-430,330)(1,0)[{}_{D(H)}\C`{}_H ^H\YD(\C)`]{1080}1a
\putmorphism(-350,330)(0,-1)[``]{460}1l
\putmorphism(-460,-160)(2,-1)[` `]{640}1b
\putmorphism(-570,-150)(2,-1)[`\YD(\C)^{H^{op}}_{H^{cop}}`]{700}0b
\putmorphism(330,-340)(2,1)[``]{90}1b
\efig}
\qquad\textnormal{ and}\qquad
\scalebox{0.84}{
\bfig
\putmorphism(-300,-160)(1,0)[{}_{H^{cop}}\YD(\C)^H`{}^{H^{op}}\YD(\C)_H`]{1000}1b
\putmorphism(620,330)(0,-1)[``]{480}1r
\putmorphism(-430,330)(1,0)[\C_{D(H)}`\YD(\C)_H ^H`]{1080}1a
\putmorphism(-380,330)(0,-1)[``]{460}1l
\putmorphism(-460,-160)(2,-1)[` `]{640}1b
\putmorphism(-570,-150)(2,-1)[` {}_{H^{cop}}^{H^{op}}\YD(\C)`]{700}0b
\putmorphism(330,-340)(2,1)[``]{90}1b
\efig}
$$
\par\bigskip
\par\bigskip

There are further monoidal isomorphisms for YD-categories. In \cite[Lemma 3.5.6]{Besp}
there is given a monoidal isomorphism ${}_H\YD(\C)^{H^{op}} \to {}_H\YD(\C)_A$, where $A$ is a
further bialgebra with a bialgebra pairing $\ro: H\ot A\to I$. Here the latter is a
monoidal category without any symmetricity conditions, but the former requires some. On the other hand,
we checked that there is a monoidal isomorphism ${}_H^H\YD(\C) \to {}_{H, A^{cop}}\YD(\C)$, where the
latter does require some symmetricity conditions whereas the former does not. The objects $M$ of
${}_{H, A^{cop}}\YD(\C)$ satisfy the condition:
$$
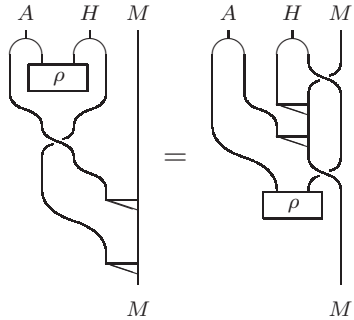
\gbeg{5}{10}
\got{2}{A} \got{2}{H} \got{1}{M} \gnl
\gcmu \gcmu \gcl{7} \gnl
\gcl{1} \glmpt \gnot{\hspace{-0,4cm}\ro} \grmpt  \gcl{1} \gnl
\gcn{1}{1}{1}{3} \gcn{1}{1}{5}{3}  \gnl
\gvac{1} \gbr \gnl
\gvac{1} \gcl{1} \gcn{1}{1}{1}{3} \gnl
\gvac{1} \gcn{1}{2}{1}{5} \gvac{1} \glm \gnl
\gvac{4} \gcl{1} \gnl
\gvac{3} \glm \gnl
\gvac{3} \gob{3}{M}
\gend=
\gbeg{5}{10}
\got{2}{A} \got{2}{H} \got{1}{M} \gnl
\gcmu \gcmu \gcl{1} \gnl
\gcl{1} \gcl{1} \gcl{1} \gbr \gnl
\gcl{1} \gcn{1}{1}{1}{3} \glm \gcl{2} \gnl
\gcn{1}{2}{1}{5} \gvac{1} \glm \gnl
\gvac{3} \gbr \gnl
\gvac{2} \glmpt \gnot{\hspace{-0,4cm}\ro} \grmpt \gcl{3} \gnl
\gvac{3} \gob{3}{M}
\gend
$$
where $\ro:A\ot H\to I$ is a bialgebra pairing. This is another example of the apeearance that
a (braided) monoidal isomorphism functor from a YD-category in $\C$ necessarily requires that the braiding
in $\C$ be symmetric between $H$ and any object of the category.

\section{Center construction} \selabel{center}
\setcounter{equation}{0}

The center construction for monoidal categories has been introduced independently by Drinfel'd %\cite{Dp}
\footnote{Private communication to Majid in response to the preprint of \cite{Maj1}, February 1990.}
and Joyal and Street \cite{JS}. It consists of assigning a braided monoidal category
called {\em center of $\C$} to a monoidal category $\C$. We will differ the left $\Z_l(\C)$ and
the right $\Z_r(\C)$ center of $\C$. We recall here the definition of the (right) center from
\cite[Definition XIII.4.1]{K}.

\begin{propdefn}
For a monoidal category $\C$ the objects of $\Z_r(\C)$ are pairs $(V, c_{-, V})$ with
$V\in\C$, where $c_{-, V}$ is a family of natural isomorphisms
$c_{X, V}: X\ot V\to V\ot X$ for $X\in\C$ such that for all
$Y\in\C$ it is
\begin{eqnarray}\eqlabel{braid-rel}
c_{X\ot Y, V}=(c_{X, V}\ot Y)(X\ot c_{Y, V}).
\end{eqnarray}
A morphism between $(V, c_{-, V})$ and $(W, c_{-, W})$ is a morphism
$f:V\to W$ in $\C$ such that for all $X\in\C$ it is
\begin{eqnarray}\eqlabel{c-morf}
(f\ot X)c_{X, V}=c_{X, W}(X\ot f).
\end{eqnarray}
The identity morphism in $\C$ is the identity morphism in $\Z_r(\C)$ and the composition
of two morphisms in $\C$ is a morphisms in $\Z_r(\C)$. Thus $\Z_r(\C)$ is a category,
called {\em the right center of $\C$}.
\end{propdefn}

From the definition it is clear that $c_{-, -}$ is a transformation natural in both arguments.
In \cite[Theorem XIII.4.2]{K} it is proved that $\Z_r(\C)$ is a braided monoidal category. The unit
object is $(I, \Id)$, the tensor product of $(V, c_{-, V})$ and $(W, c_{-, W})$ is
$(V\ot W, c_{-, V\ot W})$, where $c_{X, V\ot W}: X\ot V\ot W\to V\ot W\ot X$ is a morphism
in $\C$ defined for all $X\in\C$ by
\begin{eqnarray}\eqlabel{c-monoid}
c_{X, V\ot W}=(V\ot c_{X, W})(c_{X, V}\ot W).
\end{eqnarray}
The braiding in $\Z_r(\C)$ is given by:
$$c_{V, W}: (V, c_{-, V}) \ot (W, c_{-, W}) \to (W, c_{-, W})\ot (V, c_{-, V}).$$

The left center $\Z_l(\C)$ of $\C$ is defined analogously -- an object in $\Z_l(\C)$
has the form $(V, c_{V, -})$ with $V\in\C$.
\par\medskip

%The center has the following universal property, see \cite[Proposition XIII.4.3]{K}. Let
%$F:\D\to\C$ be a strict monoidal functor from a strict braided monoidal category $\D$ to
%a strict monoidal category $\C$. Suppose that $F$ is bijective on objects and surjective
%on morphisms. Then there exists a unique strict braided functor $Z(F): \D\to\Z(\C)$ such
%that $F=\U\comp Z(F)$, where $\U: \Z(\C)\to\C$ is the strict monoidal functor given by
%$\U((V, c_{-, V}))=V$. Taking $F$ to be identity and $\D=\C$ one has:

%\begin{prop}\cite[Corollary XIII.4.4]{K}%\prlabel{univ-prop-center}
%For any strict braided monoidal category $\C$ there exists a unique strict braided functor
%$Z: \C\to\Z(\C)$ such that $Id_{\C}=\U\comp Z$.
%\end{prop}
%\par\bigskip

%From the beginnings of the center construction, \cite{Dp}, it was understood that the center of
%the monoidal category of modules over a Hopf algebra $H$ is isomorphic to the category of
%YD-modules over $H$.
%In \cite[Example 1.3]{Maj4} it can be found an explicit proof for the observation of Drinfel'd
%\cite{Dp} that
For a Hopf algebra $H$ over a field the left center of the category of left modules over $H$ is
isomorphic to ${}_H ^H\YD$ %\cite{Dp},
\cite[Example 1.3]{Maj4}, %whereas in \cite[Theorem XIII.5.1]{K} it is proved that
and the right center of the category of left modules over $H$ is isomorphic to ${}_H \YD^H$
\cite[Theorem XIII.5.1]{K}. %Because of the isomorphism of the latter category with that of
%modules over the Drinfel'd double of $H$ the center of the category of $H$-modules is also
%called {\em inner double}.
Generalizing these results to a braided monoidal category $\C$, Bespalov indicated in
\cite[Proposition 3.6.1]{Besp} that %the YD-categories over a Hopf algebra $H\in\C$ are isomorphic
${}_H ^H\YD(\C)$ is isomorphic as a braided monoidal category to a subcategory $\Z_l^{\C}({}_H\C)$
of the (left) center of ${}_H\C$. The condition that the objects $(V, c_{V, -})$ of $\Z_l^{\C}({}_H\C)$
fulfill is that for every $X\in\C$ with trivial $H$-action (via the counit)
the morphism $c_{V, X}$ coincides with the braiding $\Phi_{V, X}$ in $\C$.
%He points out that analogously the category $\YD(\C)_H^H$ is isomorphic to $\Z^{\C}(\C^H)$.
In other words, with the forgetful functor $\U: {}_H\C \to \C$ one has that $c_{V, \U(X)}=\Phi_{V, \U(X)}$
for every $X\in {}_H\C$. For completeness we present below the proof for an analogous statement.

\begin{prop} \prlabel{Besp helped}
The categories $\Z^{\C}_r({}_H\C)$ and ${}_H\YD(\C)^{H^{op}}$ are isomorphic as braided
monoidal categories.
\end{prop}

\begin{proof}
First of all, note that for $(V, c_{-, V})\in\Z^{\C}_r({}_H\C)$ we have:
\begin{equation} \eqlabel{moj uslov}
c_{H,V}\stackrel{nat.}{=}
\gbeg{3}{5}
\gvac{1} \got{1}{H} \got{1}{V} \gnl
\gu{1} \gcl{1} \gcl{1} \gnl
\glmptb \gcmptb \gnot{\hspace{-0,76cm} c_{H\ot H, V}} \grmptb \gnl
\gcl{1} \gmu \gnl
\gob{1}{V} \gob{2}{H}
\gend\stackrel{\equref{braid-rel}}{=}
\gbeg{3}{5}
\gvac{1} \got{1}{H} \got{1}{V} \gnl
\gu{1} \glmptb \gnot{\hspace{-0,36cm} c_{H, V}} \grmptb \gnl
\glmptb \gnot{\hspace{-0,36cm} c_{H, V}} \grmptb \gcl{1} \gnl
\gcl{1} \gmu \gnl
\gob{1}{V} \gob{2}{H}
\gend=
\gbeg{3}{5}
\gvac{1} \got{1}{H} \got{1}{V} \gnl
\gu{1} \gbr \gnl
\glmptb \gnot{\hspace{-0,36cm} c_{H, V}} \grmptb \gcl{1} \gnl
\gcl{1} \gmu \gnl
\gob{1}{V} \gob{2}{H.}
\gend
\end{equation}
The morphism $\rho:=c_{H, V}(\eta_H\ot V):V\to V\ot H$
defines a right $H$-comodule structure on $V$:
$$
\gbeg{4}{5}
\got{1}{V} \gnl
\grcm \gnl
\gcl{1} \gcn{2}{2}{1}{5} \gnl
\grcm \gnl
\gob{1}{V} \gob{1}{H} \gob{3}{H}
\gend=
\gbeg{4}{6}
\got{1}{} \got{1}{} \got{1}{V} \gnl
\gvac{1} \gu{1} \gcl{1} \gnl
\gvac{1} \glmptb \gnot{\hspace{-0,32cm} c_{H, V}} \grmptb \gnl
\gu{1} \gcl{1} \gcl{2} \gnl
\glmptb \gnot{\hspace{-0,32cm} c_{H, V}} \grmptb \gnl
\gob{1}{V} \gob{1}{H} \gob{1}{H}
\gend\stackrel{\equref{braid-rel}}{=}
\gbeg{3}{6}
\got{1}{} \got{1}{} \got{1}{V} \gnl
\gu{1} \gu{1} \gcl{1} \gnl
\glmptb \gcmptb \gnot{\hspace{-0,76cm} c_{H\ot H, V}} \grmptb \gnl
\gcl{2} \gcl{2} \gcl{2} \gnl
\gob{1}{V} \gob{1}{H} \gob{1}{H}
\gend=
\gbeg{4}{6}
\got{2}{} \got{2}{V} \gnl
\gvac{1} \gu{1} \gcn{1}{2}{2}{2} \gnl
\gvac{1} \hspace{-0,36cm} \gcmu \gnl
\gvac{1} \glmptb \gcmptb \gnot{\hspace{-0,76cm} c_{H\ot H, V}} \grmptb \gnl
\gvac{1} \gcl{1} \gcl{1} \gcl{1} \gnl
\gvac{1} \gob{1}{V} \gob{1}{H} \gob{1}{H}
\gend\stackrel{nat.}{=}
\gbeg{3}{6}
\got{1}{} \got{1}{V} \gnl
\gu{1} \gcl{1} \gnl
\glmptb \gnot{\hspace{-0,36cm} c_{H, V}} \grmptb \gnl
\gcl{2} \gcn{1}{1}{1}{2} \gnl
\gvac{1} \gcmu \gnl
\gob{1}{V} \gob{1}{H} \gob{1}{H}
\gend=
\gbeg{4}{5}
\got{1}{V} \gnl
\grcm \gnl
\gcl{2} \gcn{1}{1}{1}{2} \gnl
\gvac{1} \gcmu \gnl
\gob{1}{V} \gob{1}{H} \gob{1}{H.}
\gend
$$
The counit property follows from $c_{I,V}=id_V$ (see \equref{braid-rel}).
With this $H$-comodule and the existing $H$-module structure $V$ is a left-right YD-module:
$$
\gbeg{4}{9}
\got{3}{H} \got{1}{V} \gnl
\gwcm{3} \gcl{1} \gnl
\gcn{1}{2}{1}{3} \gvac{1} \glm \gnl
\gvac{1} \gcn{1}{1}{5}{3} \gnl
\gvac{1} \gbr \gnl
\gcn{1}{1}{3}{1} \gcn{1}{1}{3}{5} \gnl
\grcm \gcn{1}{1}{3}{3} \gnl
\gcl{1} \gwmu{3} \gnl
\gob{1}{V} \gob{3}{H}
\gend=
\gbeg{4}{9}
\got{3}{H} \got{1}{V} \gnl
\gwcm{3} \gcl{1} \gnl
\gcn{1}{2}{1}{3} \gvac{1} \glm \gnl
\gvac{1} \gcn{1}{1}{5}{3} \gnl
\gvac{1} \gbr \gnl
\gu{1} \gcl{1} \gcn{1}{1}{1}{3} \gnl
\glmptb \gnot{\hspace{-0,36cm} c_{H, V}} \grmptb \gcn{1}{1}{3}{3} \gnl
\gcl{1} \gwmu{3} \gnl
\gob{1}{V} \gob{3}{H}
\gend\stackrel{\equref{moj uslov}}{=}
\gbeg{4}{7}
\got{2}{H} \got{1}{V} \gnl
\gcmu \gcl{1} \gnl
\gcl{1} \glm \gnl
\gcn{1}{1}{1}{3} \gvac{1} \gcl{1} \gnl
\gvac{1} \glmptb \gnot{\hspace{-0,34cm} c_{H, V}} \grmptb \gnl
\gvac{1} \gcl{1} \gcl{1} \gnl
\gvac{1} \gob{1}{V} \gob{1}{H}
\gend=
\gbeg{4}{7}
\got{2}{H} \got{3}{V} \gnl
\gcmu \gu{1} \gcl{1} \gnl
\gcl{1} \gbr \gcl{1} \gnl
\gmu \glm \gnl
\gcn{1}{1}{2}{3} \gcn{1}{1}{5}{3} \gnl
\gvac{1} \glmptb \gnot{\hspace{-0,34cm} c_{H, V}} \grmptb \gnl
\gvac{1} \gob{1}{V} \gob{1}{H}
\gend\stackrel{c_{H, V}\in {}_H\C}{=}
\gbeg{4}{7}
\got{2}{H} \got{3}{V} \gnl
\gcn{1}{2}{2}{2} \gvac{1} \gu{1} \gcl{1} \gnl
\gvac{2} \glmptb \gnot{\hspace{-0,34cm} c_{H, V}} \grmptb \gnl
\gcmu \gcl{1} \gcl{1} \gnl
\gcl{1} \gbr \gcl{1} \gnl
\glm \gmu \gnl
\gob{3}{V} \gob{1}{\hspace{-0,32cm}H}
\gend=
\gbeg{3}{7}
\got{2}{H} \got{1}{V} \gnl
\gcn{1}{1}{2}{2} \gcn{1}{1}{3}{3} \gnl
\gcmu \grcm \gnl
\gcl{1} \gbr \gcl{1} \gnl
\glm \gmu \gnl
\gcn{1}{1}{3}{3} \gcn{1}{1}{4}{4} \gnl
\gob{3}{V} \gob{1}{\hspace{-0,26cm}H.}
\gend
$$
A morphism $f:V\to W$ in $\Z^{\C}_r({}_H\C)$ becomes a morphism of left-right YD-modules
-- it is right $H$-colinear because of \equref{c-morf}. This defines a functor $\K$ from
$\Z^{\C}_r({}_H\C)$ to the category of left-right YD-modules. We now prove that
$\K:\Z^{\C}_r({}_H\C)\to {}_H\YD(\C)^{H^{op}}$ is monoidal. Let $(V, c_{-, V})$
and $(W, c_{-, W})$ be in $\Z^{\C}_r({}_H\C)$. Then we have:
$$
\gbeg{4}{5}
\got{1}{\K(V\ot W)} \gnl
\grcm \gnl
\gcl{2} \gcn{2}{2}{1}{5} \gnl \gnl
\gob{1}{\K(V\ot W)} \gob{5}{H}
\gend=
\gbeg{4}{5}
\got{1}{} \got{3}{V\ot W} \gnl
\gu{1} \gvac{1} \gcl{1} \gnl
\glmptb \gcmp \gnot{\hspace{-0,76cm} c_{H, V\ot W}} \grmptb \gnl
\gcl{1} \gvac{1} \gcl{1} \gnl
\gob{1}{V\ot W} \gob{3}{H}
\gend\stackrel{\equref{c-monoid}}{=}
\gbeg{4}{6}
\got{1}{} \got{1}{V} \got{1}{W} \gnl
\gu{1} \gcl{1} \gcl{2} \gnl
\glmptb \gnot{\hspace{-0,32cm} c_{H, V}} \grmptb \gnl
\gcl{1} \glmptb \gnot{\hspace{-0,32cm} c_{H, W}} \grmptb \gnl
\gcl{1} \gcl{1} \gcl{1} \gnl
\gob{1}{V} \gob{1}{W} \gob{1}{H}
\gend\stackrel{\equref{moj uslov}}{=}
\gbeg{4}{8}
\got{1}{} \got{1}{V} \got{1}{W} \gnl
\gu{1} \gcl{1} \gcl{2} \gnl
\glmptb \gnot{\hspace{-0,32cm} c_{H, V}} \grmptb \gnl
\gcl{1} \gcn{1}{1}{1}{3} \gcn{1}{1}{1}{3} \gnl
\gcl{1} \gu{1} \gbr \gnl
\gcl{1} \glmptb \gnot{\hspace{-0,36cm} c_{H, W}} \grmptb \gcl{1} \gnl
\gcl{1} \gcl{1} \gmu \gnl
\gob{1}{V} \gob{1}{W} \gob{2}{H}
\gend\stackrel{nat.}{=}
\gbeg{4}{7}
\got{1}{} \got{1}{V} \got{3}{W} \gnl
\gu{1} \gcl{1} \gu{1} \gcl{1} \gnl
\glmptb \gnot{\hspace{-0,32cm} c_{H, V}} \grmptb \glmptb \gnot{\hspace{-0,32cm} c_{H, W}} \grmptb \gnl
\gcl{1} \gbr \gcl{1} \gnl
\gcl{2} \gcl{2} \gbr \gnl
\gvac{2} \gmu \gnl
\gob{1}{V} \gob{1}{W} \gob{2}{H}
\gend=
\gbeg{5}{7}
\gvac{1} \got{1}{\K(V)} \got{1}{} \got{1}{\K(W)} \gnl
\gvac{1} \gcl{1} \gvac{1} \gcl{1} \gnl
\gvac{1} \grcm \grcm \gnl
\gvac{1} \gcl{1} \gbr \gcl{1} \gnl
\gvac{1} \gcl{1} \gcl{2} \gbr \gnl
\gcn{1}{1}{3}{1} \gvac{2} \gmu \gnl
\gob{1}{\K(V)} \gob{2}{\hspace{0,32cm}\K(W)} \gob{2}{H}
\gend
$$
If $(V, c_{-, V})\in\Z^{\C}_r({}_H\C)$, then $\Phi^{1+}_{-, V}=c_{-, V}$ because of
\equref{moj uslov}. On the other hand, for $M\in {}_H\YD(\C)^{H^{op}}$its
comodule structure morphism is obviously equal to $\Phi^{1+}_{H, M}(\eta_H\ot M)$. Hence
the inverse functor of $\K$ is given by sending a YD-module $M$ into the
pair $(M, \Phi^{1+}_{-, M})$. %, where $\Phi^*$ is the braiding in ${}_H\YD(\C)^{H^{op}}$.
Consequently, $\K$ respects the braiding and this finishes the proof.
\qed\end{proof}

%Applying our approach in the study of the categories in \equref{first4categories} and \equref{other4categories}
Similarly, one may prove %we have
that the following categories are braided monoidally isomorphic:
$$\Z^{\C}_l({}_H\C)\iso {}_H ^H\YD(\C)\iso \Z^{\C}_r({}^H\C),
\qquad \qquad \Z^{\C}_r(\C_H)\iso \YD(\C)_H ^H\iso \Z^{\C}_l(\C^H)$$
\vspace{0,1cm}
$${}^H\YD(\C)_{H^{cop}}\iso \Z^{\C}_l({}^H\C),
\qquad {}_{H^{cop}}\YD(\C)^H\iso \Z^{\C}_r(\C^H), \qquad
{}^{H^{op}}\YD(\C)_H\iso \Z^{\C}_l(\C_H).$$
The above center subcategories are defined analogously to $\Z^{\C}_r({}_H\C)$. Adding to this list the
categories ${}_{(H^{op})^*\bowtie H}\C$ and $\C_{(H^{op})^*\bowtie H}$, we may identify
\begin{eqnarray}\eqlabel{center-double}
{}_{(H^{op})^*\bowtie H}\C \iso\Z^{\C}_l({}_H\C)\quad\textnormal{and}\quad\C_{(H^{op})^*\bowtie H}\iso \Z^{\C}_r(\C_H)
\end{eqnarray}
having in mind that the corresponding $H$-module structures remain unchanged by the isomorphism functors.
Then due to \equref{first4categories} and \equref{other4categories} we obtain the following
diagrams of isomorphic braided monoidal categories:
$$%\begin{eqnarray}\eqlabel{center diagrams}
\scalebox{0.84}{
\bfig
\putmorphism(-400,-130)(1,0)[\Z^{\C}_r({}_H\C)`\Z^{\C}_l({}^H\C)`]{1000}1b
\putmorphism(630,330)(0,-1)[``]{430}1r
\putmorphism(-430,330)(1,0)[\Z^{\C}_l({}_H\C)`\Z^{\C}_r({}^H\C)`]{1030}1a
\putmorphism(-380,320)(0,-1)[``]{420}1l
\efig}
\qquad\quad\textnormal{ and}\qquad
\scalebox{0.84}{
\bfig
\putmorphism(-400,-140)(1,0)[\Z^{\C}_r(\C^H)`\Z^{\C}_l(\C_H).`]{1020}1b
\putmorphism(660,330)(0,-1)[``]{450}1r
\putmorphism(-380,330)(1,0)[\Z^{\C}_r(\C_H)`\Z^{\C}_l(\C^H)`]{1000}1a
\putmorphism(-350,330)(0,-1)[``]{440}1l
\efig}
$$%\end{eqnarray}
\par\bigskip

\subsection{Transparency and M\"uger's centers $\Z_1$ and $\Z_2$}

Throughout the paper we have used the condition that $\Phi_{H, M}$ is symmetric for every $M\in\C$.
This means that {\em $H$ is transparent in $\C$} in terms of \cite{Al1}, or that $H$ belongs to M\"uger's
center $\Z_2(\C)=\{X\in\C \vert \Phi_{Y,X}\Phi_{X,Y}=id_{X\ot Y} \hspace{0,2cm}\textnormal{for all}
\hspace{0,2cm} Y\in\C\}$, \cite[Definition 2.9]{M}.
Note that due to \leref{H - H* transp.}, $H$ is transparent if and only if so is $H^*$.
The center of a monoidal category $\D$ that we studied above is denoted by $\Z_1(\D)$ in \cite{M}
(neglecting the difference between the left and the right center). Then we may state:

\begin{prop}
Let $H$ be a finite Hopf algebra with a bijective antipode in a braided monoidal category $\C$.
If $H\in\Z_2(\C)$, then there are embeddings of braided monoidal categories:
$${}_H ^H\YD(\C)\iso {}_{D(H)}\C \hookrightarrow \Z_{1, l}({}_H\C) \qquad
\textnormal{and}\qquad \YD(\C)_H ^H\iso \C_{D(H)} \hookrightarrow \Z_{1, r}(\C_H).$$
\end{prop}

\par\bigskip

\subsection{The whole center category and the coend}

The center category of a monoidal category $\C$ is a particular case of the Pontryagin dual monoidal
category introduced by Majid in \cite[Section 3]{Maj1}. For $\C$ braided, rigid and
cocomplete from \cite[Theorem 3.2]{Maj1a} one deduces that there is an isomorphism of monoidal categories:
\begin{eqnarray}\eqlabel{Majid}
\Z_l(\C)\iso\C_{\Aut(\C)}
\end{eqnarray}
where
$$%\begin{eqnarray}\eqlabel{coend}
\Aut(\C)\iso \int^X X^*\ot X
$$%\end{eqnarray}
is the coend in $\C$. It has a structure of a bialgebra in $\C$ and if $\C$ is rigid, it is a
Hopf algebra. %Moreover, it is ``braided cocommutative'' in the sense that $\Delta^{op}=\Delta$
%and quasitriangular with $\R=\eta\ot\eta$.
%When $\C$ is the category of modules over a finite-dimensional quasitriangular Hopf algebra $H$,
%this yields $\Z_l(\M_H)\iso\M_{D'(H)}$.
As we observed in \seref{YD-DH}, if $H$ is a quasitriangular Hopf algebra such that $\Phi_{H, M}$ is
symmetric for all $M\in\C$, i.e. $H\in\Z_2(\C)$, then the whole category
$\C_H$ is braided. Thus for $\C$ rigid $\Aut(\C_H)$ becomes a Hopf algebra in $\C_H$
and according to \prref{cross prod isom} the categories $(\C_H)_{\Aut(\C_H)}$ and
$\C_{H\ltimes\Aut(\C_H)}$ are monoidally isomorphic.
By the identity \equref{Majid} we then have:

\begin{prop}
Let $\C$ be a rigid braided monoidal category and $H\in\C$ a quasitriangular Hopf algebra such that
$H\in\Z_2(\C)$. There is a monoidal isomorphism of categories:
$$\Z_l(\C_H)\iso\C_{H\ltimes\Aut(\C_H)}.$$
\end{prop}

\par\medskip

When $\C=Vec$ and $H$ is a finite-dimensional quasitriangular Hopf algebra, $\Aut(\M_H)=H^*$ as a
vector space with a modified multiplication, \cite{Maj1a}, and the above yields $\Z_l(\M_H)\iso
\M_{D'(H)}$, where $D'(H)=H\bowtie H^{* op}$ is a version of the Drinfel'd double.
%According to \cite[Theorem 6.1.1]{Besp} the categories $(\C_H)_{\Aut(\C_H)}$ and
%$\C_{H\ltimes\Aut(\C_H)}$ are isomorphic as braided monoidal categories. The identity
%\equref{Majid} then yields that there is a monoidal isomorphism of categories:
%$$\Z_l(\C_H)\iso\C_{H\ltimes\Aut(\C_H)}.$$
Symmetrically to \equref{Majid} one has $\Z_r(\C)\iso {}_{\Aut(\C)}\C$. For $H\in\Z_2(\C)$
this yields the monoidal isomorphism $\Z_r({}_H\C)\iso {}_{\Aut({}_H\C)\rtimes H}\C$. Here
$\Aut({}_H\C)\rtimes H$ is the bosonization of the braided Hopf algebra $\Aut({}_H\C)$ in ${}_H\C$.
\par\bigskip

Another approach to the center construction of monoidal categories and the Drinfel'd double uses
monads \cite{Al}. Assume  $T$ is a Hopf monad in a rigid monoidal category $\C$ for which the
coend $C_T(X)=\int^{Y\in\C} T(Y)^*\ot X\ot Y$ exists for every $X\in\C$. The authors construct a
quasitriangular Hopf monad $\Dd_T$, called {\em the double of $T$}, and prove the braided
monoidal isomorphism ${}_{\Dd_T}\C\iso\Z({}_T\C)$, \cite[Theorem 6.5]{Al}. Relying on monads,
this construction generalizes the Drinfel'd double to a fully non-braided setting. In the particular
case when a Hopf monad is associated to a Hopf algebra $H$ in a rigid braided monoidal category $\C$,
the underlying object of the double $\Dd_H$ is $H\ot H^*\ot \Aut(\C)$,
%. Here $C$ is the coend \equref{coend} of $\C$ (assuming $\C$ admits it, e.g. $\C$ is cocomplete).
assuming that $\C$ admits the coend, e.g. $\C$ is cocomplete. (When $\C=Vec$, one recovers the
usual Drinfel'd double.) %$\Dd_H=H\bowtie H^{* cop}$.
In this case one has the braided monoidal isomorphisms (\cite[Theorem 8.13]{Al}):
\begin{eqnarray}\eqlabel{Alain}
\Z_l(\C_H)\iso\C_{\Dd_H}\iso {}_{\Dd_H}\C\iso\Z_r({}_H\C).
%\C_{\Dd_H}\iso\Z_l(\C_H).
\end{eqnarray} %This formally resembles the isomorphism in \equref{coend-H},
%however, the notions of the quasitriangular structures in the two cases differ, and so do as well
%the braided structures of the corresponding categories.
To prove the isomorphism between the left and the right %object, as well the isomorphism in the middle,
hand-side categories one applies identifications with objects in $\C^{cop}$. Moreover, the
isomorphism in the middle is possible since $\Dd_H$ is quasitriangular.
For $H=I$ the trivial Hopf algebra, it is $\Dd_I=\Aut(\C)$ and one recovers \equref{Majid}. On the
other hand, observe that for $H=I$ the center subcategory becomes $\Z_r^{\C}({}_I\C)\iso\C$.

We point out that the notions of a quasitriangular structure in \cite{Maj1a} and \cite{Al} differ. In the
latter case an R-matrix for a Hopf algebra $H\in\C$ is a morphism $\r:C\ot C\to H\ot H$ defined in such a way
that $H$ is quasitriangular if and only if the category of $H$-modules in $\C$ is braided. The R-matrix that
Majid uses \cite[Definition 1.3]{Maj2} (and which we apply) is a morphism $\R:I\to H\ot H$ obtained by
straightforward extension of the axioms in the classical case. Its existence implies that the {\em subcategory}
$\Oo(H, \Delta^{op})$ of the category of $H$-modules in $\C$ is braided. %\cite[Proposition 3.2]{Maj2}.
Though, both constructions recover the classical notion of a quasitriangular structure for the
category of vector spaces (in this case the coend is just the field).

\section{Particular cases and examples} \selabel{center}
\setcounter{equation}{0}

When a Hopf algebra $H\in\C$ is commutative or/and cocommutative, the symmetricity condition on
$\Phi_{H, X}$ for any $X\in\C$ that emerges throughout the paper obtains a certain interpretation.

\begin{prop} \cite[Proposition 3.12]{CF1}  \prlabel{braid-lin}
Let $H\in\C$ be a Hopf algebra.
\begin{enumerate}
\item[(i)] The braiding $\Phi$ of $\C$ is left $H$-linear if and only if $\Phi_{H, X}=\Phi^{-1}_{X, H}$ for
any $X\in\C$ and $H$ is cocommutative.

\item[(ii)] The braiding $\Phi$ of $\C$ is left $H$-colinear if and only if $\Phi_{H, X}=\Phi^{-1}_{X, H}$ for any $X\in\C$ and $H$ is commutative.
\end{enumerate}
\end{prop}

On the other hand, if the braiding $\Phi$ of $\C$ is left $H$-linear, then the category $_H\C$
is braided monoidal with the same braiding $\Phi$. Similarly, if $\Phi$ is left $H$-colinear, then the
category $^H\C$ is braided monoidal with the braiding $\Phi$.
\par\medskip

We illustrate the above cases by an example.
The following family of Hopf algebras was studied in \cite[Section 4]{N}. Let $n,m$ be
natural numbers, $k$ a field such that $char(k) \nmid 2m$ and $\omega$ a $2m$-th primitive root of unity.
For $i=1,...,n$ choose $1 \leq d_i < 2m$ odd numbers and set $d^{\leq n}=(d_1,...,d_n)$. Then
$$H(m,n,d^{\leq n})=k\langle g,x_1,...,x_n\vert g^{2m}=1, x_i^2=0, gx_i=\omega^{d_i} x_ig, x_ix_j=-x_jx_i \rangle$$
is a Hopf algebra, where $g$ is group-like and $x_i$ is a $(g^m, 1)$-primitive element, that is,
$\Delta(x_i)=1\ot x_i+x_i\ot g^m$ and $\Epsilon(x_i)=0$. The antipode is given by $S(g)=g^{-1}$ and
$S(x_i)=-x_ig^{m}$. We proved in \cite{CF1} that $H(m,n,d^{\leq n})$ decomposes as the Radford biproduct
(indeed a bosonization):
\begin{equation} \eqlabel{ex-biproduct}
H(m,n,d^{\leq n})\iso B\rtimes H(m,n-1,d^{\leq n-1})
\end{equation}
where the braided Hopf algebra is the exterior algebra $B=K[x_n]/(x_n^2)$. The isomorphism is given by:
$G \mapsto 1\times g, X_i \mapsto 1 \times x_i, X_n \mapsto x_n \times g^m$. We have that $B$
is a module over $H=H(m,n-1,d^{\leq n-1})$ by the action $g\cdot x_n=\omega^{d_n}x_n$ and $x_i \cdot x_n=0$
for $i=1,...,n-1$. It becomes a commutative and cocommutative Hopf algebra in $_H\M$ with $x_n$ being a
primitive element, i.e., $\Delta_B(x_n)=1\ot x_n+x_n\ot 1,\ \Epsilon_B(x_n)=0$ and $S_B(x_n)=-x_n.$
The Hopf algebra $H(m,n,d^{\leq n})$ is quasitriangular with the family of quasitriangular structures
\cite[(6.4) on p. 69]{CF1}:
\begin{eqnarray} \label{qtr-Hnu}
\R_s^n=\frac{1}{2m}\Big(\sum_{j,t=0}^{2m-1}\omega^{-jt}g^j \otimes g^{st}\Big)
\end{eqnarray}
where $0\leq s< 2m$ is such that $sd_i \equiv m \ (mod.\ 2m)$ for every $i=1,...,n$. Moreover, $\R_s^n$
is triangular if and only if $s=m$.
As it is well known (\cite{Maj}), every left $H$-module $M$ belongs to ${\ }^H_H\YD$ with the coaction
\begin{eqnarray} \eqlabel{leftcomMaj}
\lambda(m)=\R^{(2)}\ot \R^{(1)} m, \quad m\in M
\end{eqnarray}
- we denote $\R=\R_s^{n-1}$ for brevity -
and $({}_H\M, \Phi_{\R})$ can be seen as a braided monoidal subcategory of $({\ }^H_H\YD, \Phi^L)$.
Here $\Phi^L$ is given by \equref{braid-LR}, that is $\Phi^L(x\ot y)=x_{[-1]}\cdot y\ot x_{[0]}$,
whereas $\Phi_{\R}$ and its inverse are given by:
\begin{equation} \eqlabel{braid-Hnd}
\Phi_{\R}(x\ot y)=\R^{(2)} y\ot\R^{(1)} x; \qquad \Phi_{\R}^{-1}(x\ot y)=\R^{(1)} y\ot S^{-1}(\R^{(2)}) x.
\end{equation}
Thus $B$ becomes a Hopf algebra in $({\ }^H_H\YD, \Phi^L)$.
\par\medskip

Set $\C={\ }^H_H\YD$. Let us now prove that $\Phi^L_{B,M}$ is symmetric for any $M$ in $\C$.
Take $m\in M$ and let us check if $\Phi^L(b\ot m)=(\Phi^L)^{-1}(b\ot m)$ (see \equref{braid-Hnd} and
(\ref{qtr-Hnu})). For $b=1$ the computation is easier, we compute here the case $b=x_n$. We find:
$$\begin{array}{rl}
\Phi_{\R}(x_n\ot m)\hskip-1em&=
\frac{1}{2m}\Big(\sum_{j,t=0}^{2m-1}\omega^{-jt}g^{st} \cdot m \ot g^j\cdot x_n\Big) \\
&=\frac{1}{2m}\Big(\sum_{j,t=0}^{2m-1}\omega^{-jt}g^{st} \cdot m \ot \omega^{d_n j} x_n\Big) \\
&=\frac{1}{2m}\Big(\sum_{t=0}^{2m-1}[\sum_{j=0}^{2m-1}(\omega^{d_n-t})^j] g^{st} \cdot m\Big) \ot x_n \\
&=g^{sd_n}\cdot m\ot x_n
\end{array}$$
(the sum in the bracket in the penultimate expression is different from $0$ only for $j=-sd_n\ (mod.\ 2m)$,
when it equals $2m$). Similarly, it is:
$$\begin{array}{rl}
\Phi_{\R}^{-1}(x_n\ot m)\hskip-1em&=\frac{1}{2m}\Big(\sum_{j,t=0}^{2m-1}\omega^{-jt}g^j \cdot m \ot
 S^{-1}(g^{st})\cdot x_n\Big) \\
&=\frac{1}{2m}\Big(\sum_{j,t=0}^{2m-1}\omega^{-jt}g^j \cdot m \ot
 g^{-st}\cdot x_n\Big) \\
&=\frac{1}{2m}\Big(\sum_{j,t=0}^{2m-1}\omega^{-jt}g^j \cdot m \ot
 \omega^{-d_n st} x_n\Big) \\
&=\frac{1}{2m}\Big(\sum_{j=0}^{2m-1}[\sum_{t=0}^{2m-1}(\omega^{-(j+s d_n)})^t] g^j \cdot m\Big) \ot x_n \\
&=g^{-sd_n}\cdot m\ot x_n.
\end{array}$$
Recall that %$0\leq s< 2m$ is such that
$sd_i \equiv m \ (mod.\ 2m)$ for every $i=1,...,n$. Hence $g^{sd_n}=-1$ and the two expressions we
computed above are equal. Thus the wanted symmetricity condition is fulfilled for the described family
of Hopf algebras.
\par\medskip

This together with the fact that $B$ is both commutative and cocommutative in $\C$ means due to
\prref{braid-lin} that $\Phi^L$ is $B$-linear and $B$-colinear. Hence $_B\C$ and $^B\C$ are braided by
$\Phi^L$. Actually, we have more.
In \equref{ex-biproduct} the quasitriangular structure $\R$ extends from $H$ to $B\rtimes H$.
The extension is given by $\crta{\R}=(\iota\ot \iota)\R$, where $\iota: H \to B\rtimes H$ is the
Hopf algebra embedding. Consequently, the braiding $\Phi_{\R}$ in $_H\M$
- which determines simultaneously the braiding in $\C$ - extends to the braiding $\Phi_{\crta\R}$ in
$_{B\rtimes H}\M$ - which determines the braiding in ${}_{B\rtimes H} ^{B\rtimes H}\YD$. In other words, the
braiding in $\C$ extends to the braiding in ${}_{B\rtimes H} ^{B\rtimes H}\YD\iso {}_B ^B\YD(\C)$
(extension by trivial $B$-(co)actions). The latter braided monoidal isomorphism is due to
the left version of \cite[Proposition 4.2.3]{Besp}.

\par

\end{document}